\documentclass[sn-mathphys,Numbered]{sn-jnl} 

\usepackage{tikz}
\usetikzlibrary{arrows} 
\usepackage{pgfplots}

\usepackage{anyfontsize}
\usepackage{graphicx}  
\usepackage{amsthm,amsmath,amssymb,amsfonts} 
 
\usepackage{mathrsfs} 
\usepackage[title]{appendix} 
\usepackage{xcolor} 
\usepackage{float}
\usepackage{textcomp} 
\usepackage{manyfoot}  
\usepackage{algorithm} 
\usepackage{algorithmicx} 
\usepackage{algpseudocode}  
\usepackage{array}
 \usepackage{caption}
 \usepackage{subcaption}
 \usepackage[utf8]{inputenc}
\usepackage{pdflscape}
\usepackage{enumerate}
\usepackage{lineno}
\usepackage{framed}   
\usepackage{booktabs}
\usepackage{geometry}  
   
\usepackage{multirow}
\usepackage{tabulary}

\newtheorem{theorem}{Theorem}[section]
 
\newtheorem{lemma}{Lemma}[section] 
\newtheorem{corollary}{Corollary}[section]  
 
\newtheorem{remark}{Remark}[section] 
 
\newtheorem{definition}{Definition}[section]

\makeatletter
\newcommand{\thickhline}{%
    \noalign {\ifnum 0=`}\fi \hrule height 1pt
    \futurelet \reserved@a \@xhline
}

\newcommand{\argmin}{\text{argmin}}
\newcommand{\la}{\left\langle} 
\newcommand{\ra}{\right\rangle}

\begin{document}
\title[Cubic Regularized Newton Method for Vector Optimization]{Cubic Regularization Technique of the Newton Method for Vector Optimization}

\author[1]{\fnm{Debdas} \sur{Ghosh}}\email{debdas.mat@iitbhu.ac.in}

\affil[1]{\orgdiv{Department of Mathematical Sciences}, \orgname{Indian Institute of Technology (BHU)}, \orgaddress{\city{Varanasi}, \postcode{221005}, \state{Uttar Pradesh}, \country{India}}}

\abstract{This study proposes a cubic regularization of the Newton method for generating weakly efficient points of unconstrained vector optimization problems under no convexity assumption on the objective function. It is observed that at a given iterate, the cubic regularized Newton direction is not necessarily a descent direction. In generating the sequence of iterates, no line search is utilized to find a suitable step length to move along the cubic regularized Newton direction. Yet, the proposed method exhibits a global convergence property with $O(k^{-2/3})$ rate of convergence. Further, the local q-quadratic convergence of the Newton method is also retained in the cubic regularization. A new stopping condition 
is used, which enforces the proposed method to enter in close 
neighborhood of non-weakly efficient points that are stationary. Thus, the studied technique ends up generating weakly efficient points, not just Pareto critical points. In addition, conditions on the choice of regularization parameter value under which the full cubic regularized Newton step becomes descent are derived. Performance profiles and comparison of the derived method with the existing methods on several test examples are also provided. }



\maketitle

\section{Introduction} \label{introduction}
In this article, we are interested in generating weakly efficient points of the following vector optimization problem  
\begin{align}\tag{VOP} \label{vop}
\min_{x \in S} ~ (f_1(x), f_2(x), \ldots, f_p(x))^\top     
\end{align}
with respect to the partial ordering induced by a pointed closed convex cone $K \subset \mathbb{R}^p$, where $S \subseteq \mathbb{R}^n$ is a non-empty closed convex set with non-empty interior. We denote the vector-valued objective function of the problem \eqref{vop} by $f: S \to \mathbb{R}^p$, i.e.,  
\[f(x) := (f_1(x), f_2(x), \ldots, f_p(x))^\top. \]

The vector optimization problem \eqref{vop} under the canonical ordering cone, i.e., $K = \mathbb{R}^p_+$, is commonly referred to as a multi-objective optimization problem. Over the years, there have been many parameter-dependent scalarization methods \cite{ghosh2014new}, such as weighted-sum, $\varepsilon$-constraint, normal boundary intersection, normal constraint, physical programming, direct search domain, ideal cone, etc., to identify weakly Pareto optimal points or weakly efficient solutions of multi-objective optimization problems. However, all these techniques are found to be heavily dependent on parameters outside the problem data. In 2000, the pioneering work of Fliege and Svaiter \cite{fliege2000steepest} opened up a research direction of developing gradient-based parameter-free optimization methods for multi-objective and vector optimization. After the seminal work of Fliege and Svaiter \cite{fliege2000steepest}, there has been extensive research in the recent past on generalizing the conventional gradient and Hessian-based descent methods for vector optimization.

\subsection{Literature survey of parameter-free methods}
In \cite{fliege2000steepest}, steepest descent methods for constrained and unconstrained multi-objective optimization and their implementation strategy have been developed. Drummond and Iusem \cite{drummond2004projected} recollected the idea of using a non-linear scalarizing function to settle the computational issues with vector-valued functions and proposed a projected gradient method for constrained optimization of vector-valued functions. Drummond and Svaiter \cite{drummond2005steepest} introduced a $K$-steepest descent (Cauchy-like) method for smooth unconstrained vector optimization, where the concepts of $K$-critical points and ``Drummond-Svaiter non-linear scalarizing function" have been introduced.  Chuong and Yao \cite{chuong2012steepest} used the ``oriented distance" non-linear scalarizing function to find a steepest descent scheme for the vector optimization problem in Banach spaces, and provided general global convergence analysis. A worst-case complexity analysis of the steepest descent method for multi-objective optimization has been reported in \cite{fliege2019complexity}, where it has been found that for non-convex case, the global convergence rate is $O(k^{-1/2})$, for convex case it is $O(k^{-1})$, and for strongly convex case it is $O(\gamma^k)$ for some $\gamma \in (0, 1)$. To improve the performance of the steepest descent method, a generic majorization-minimization scheme has been recently provided in \cite{chen2024descent}, where it has been shown that the slow convergence rate is due to the large gap between the surrogate and the objective function; to reduce it, they recommended a strategic choice of the base of the dual of the ordering cone.    \\

In \cite{fliege2009newton}, an extension of Newton's method has been provided for unconstrained multi-objective optimization with the help of max-ordering non-linear scalarization function, where the local superlinear convergence rate is established; under Lipschitz continuity of the Hessian, the Kantorovich-like theorem on local q-quadratic convergence rate is reported. Drummond et al. \cite{drummond2014quadratically} extended the multi-objective Newton method of \cite{fliege2009newton} to the vector case and derived the q-quadratic local convergence. In \cite{wang2012regularized}, a Levenberg-Marquardt-type regularization of the Newton method for multi-objective optimization has been proposed to realize the global convergence under no convexity assumption. Wang et al. \cite{wang2019extended} further extended the Newton method of \cite{fliege2009newton} with the help of a majorizing function technique and extended the Armijo, Goldstein, and Wolfe line searches for multi-objective optimization to enhance the local convergence to the global quadratic convergence of the multi-objective Newton method. However, the majorization approach in  \cite{wang2019extended} is applicable for convex problems. Thus, Gonçalves et al. \cite{goncalves2022globally} derived two Newton methods for multi-objective optimization that are globally convergent and applicable to non-convex problems. They used non-monotone line searches and direction safeguard strategies with the help of the norm of the steepest-descent direction to achieve global convergence. Lu and Chen \cite{lu2014newton} analyzed exact and inexact versions of the Newton method for vector optimization, where to find the Newton direction, the scalarizing function is distributed to both the linear and quadratic terms of the local quadratic approximation of the objective function. \\

In the literature, several quasi-Newton methods, non-linear conjugate gradient methods, and trust-region methods have been proposed to enhance the applicability and global convergence of the Newton method. For a detailed survey on the quasi-Newton method for multi-objective optimization, we refer to \cite{kumar2023quasi,prudente2022quasi} and their references. However, to realize the efficiency of the quasi-Newton methods, we require tactful Hessian approximations, which may not be easy if the nature of the function cannot be apriori estimated. After the introduction of the standard non-linear conjugate methods for multi-objective optimization by Pérez and Prudente \cite{perez2018nonlinear}, there have been many articles on this topic; we refer to the recent article \cite{yahaya2025new} and its references for a review on conjugate gradient methods. Although conjugate gradient methods have the global convergence property for non-convex problems, there has been no convergence rate analysis of the non-linear conjugate gradient methods for vector optimization so far. However, recently, Lapuci \cite{lapucci2024convergence} showed that the worst-case convergence rate of gradient-related methods for non-convex vector optimization is $O(k^{-1/2})$. The trust region method for multi-objective optimization was first derived by Qu et al. \cite{qu2013trust} and subsequently extended by Villacorta et al. \cite{villacorta2014trust} where the problem is taken non-smooth.  Carrizo et al. \cite{carrizo2016trust} found that the standard trust-region approach for multi-objective optimization does not possess the descent property, and thus they proposed to add a linear constraint by which a globally convergent descent trust region method has been derived. 
Recently, Mohammadi and Custódio \cite{mohammadi2024trust} proposed a trust region method by employing the strategies of the extreme point step and scalarization step, where the first try to reach the extremities of the Pareto surface, and the latter fills out the entire Pareto region by inserting adequate intermediate points. By the similar two steps, in \cite{mohammadi2025trust}, a black-box trust region method for multi-objective optimization has been proposed without using the derivative and Hessian of the objective function. For a special trust region method, Garmanjani \cite{garmanjani2023complexity} showed that for the general convex problem, the rate of convergence is $O(k^{-1})$. \\

Among the other popular (sub)gradient-based methods for (non-smooth and composite) vector optimization, proximal gradient methods and their complexity analysis have been detailed in \cite{tanabe2012composite, zhao2024convergence, zhao2025proximal}. Cruz \cite{cruz2013subgradient} proposed and analyzed the complexity of a subgradient method for vector optimization without directly employing a non-linear scalarizing function. It has been found that, in general, the proximal gradient and subgradient methods for vector optimization do not possess the descent property, yet they have the global convergence property. 
For a detailed survey of the gradient-based descent methods for multi-objective optimization, we refer to the article by Fukuda and Drummond \cite{fukuda2014survey}. A detailed survey on the complexity bound of the existing methods for multi-objective optimization has been given in \cite{custodio2021worst}.

\subsection{Motivation}
We see from the literature on the parameter-free methods for vector optimization that several modifications of the Newton method have been done to improve its convergence for non-convex problems, but the cubic regularization of the Newton method has not been attempted so far. However, after the seminal work of cubic regularization of the Newton method for scalar optimization by Nesterov and Polyak \cite{nesterov2006cubic}, the strategy is well-known to be a powerful and theoretically appealing alternative in non-convex and high-dimensional settings. Unlike conjugate gradient and quasi-Newton methods---which primarily rely on first-order or approximate second-order information and may struggle with saddle points or ill-conditioning---cubic regularization explicitly incorporates a third-order term in the model to control the full cubic regularized step more robustly and ensure better global convergence guarantees. The method constructs a local model of the objective function by augmenting the second-order Taylor expansion with a cubic term, which acts as a natural and adaptive penalization for large steps that effectively balances the local curvature and global surrogate function. Compared to trust-region methods, which enforce a hard constraint on the step size through a region radius, cubic regularization introduces a smooth penalization that allows more flexible and theoretically sound step adjustments. This regularization not only prevents overly aggressive steps in poorly conditioned directions but also leads to strong worst-case iteration complexity bounds that outperform many standard methods, especially in escaping saddle points and converging to second-order stationary points.

\subsection{Work done}
In this article, we thus propose a cubic regularization of the Newton method for vector optimization where the objective function of the \eqref{vop} is twice continuously differentiable and has Lipschitz Hessian but not necessarily convex. The article is presented in the following sequence. The next section presents the notations and basic results of vector optimization that are used in the entire article. Section \ref{section-cubic-regularization} derives a cubic Newton direction and its computational strategies. A step-wise algorithm and its convergence behavior are also provided in the same section. In Section \ref{section-numerical-experiment}, the numerical performance of the method is shown on some test problems, where we also exhibit a comparison of the performance with the existing method through Dolan-Mor{\'e} performance profiles. The study is concluded in Section \ref{section-conclusion} by providing a few final remarks and potential future scopes.   \\

\section{Preliminaries} 
In this section, we provide notations, basic definitions and results, and the assumptions on \eqref{vop}, which are used throughout the study.

\begin{itemize}
\item All elements of Euclidean spaces $\mathbb{R}^n$ or $\mathbb{R}^p$ are presented by column vectors. 

\item The notations $\la \cdot, \cdot \ra$, and $\|\cdot\|$ present the usual inner product and the corresponding norm, respectively, of the elements of these Euclidean spaces. 

\item The notations int$(K)$ and int$(S)$ represent the interior of $K$ and $S$, respectively. 

\item The unit sphere in $\mathbb{R}^p$ is denoted by $\mathbb{S}^{p - 1}$. 

\item The set of all non-negative real numbers is presented by $\mathbb{R}_+$.  

\item $I_n$ is the identity matrix of order $n \times n$. 

\item The cubic regularized Newton step at $x \in S$ is represented by $d_M(x)$. Also, we denote $r_M(x) := \|d_M(x)\|$ and $r_{M_k}(x^k) := \|d_{M_k}(x^k)\|$. 

\item A generic element of the cone $K$ or of its (positive) polar cone $K^*$ is denoted by $\eta := (\eta_1, \eta_2, \ldots, \eta_p)^\top$ or $\xi := (\xi_1, \xi_2, \ldots, \xi_p)^\top$. Also, we present $\eta^i := (\eta^i_1, \eta^i_2, \ldots, \eta^i_p)^\top$. 

\item The subdifferential of a mapping $\theta : S \to \mathbb{R}$ at $x \in S$ is presented by $\partial \theta (x)$.  

\item For the twice differentiable function $f: S \to \mathbb{R}^p$, the Jacobian and Hessian at $x \in S$ are presented by 
\[Jf(x) : = 
\begin{pmatrix}
    \nabla f_1(x) ^\top \\ 
    \nabla f_2(x) ^\top  \\
    \vdots \\ 
    \nabla f_p(x) ^\top 
\end{pmatrix} 
~\mathrm{ and }~ 
\nabla^2 f(x) : = 
\begin{pmatrix}
    \nabla^2 f_1(x) \\ 
    \nabla^2 f_2(x)  \\
    \vdots \\ 
    \nabla^2 f_p(x) 
\end{pmatrix},  
\]
respectively. Note that $Jf(x)$ is a $p \times n$ matrix, and $\nabla^2 f(x)$ is a $pn \times n$ matrix. 
 
\item At an $x \in S$, for a given $d$ in $\mathbb{R}^n$, the notations $Jf(x) d$, $d^\top \nabla^2 f(x)$, and $d^\top \nabla^2 f(x) d$ represent  
\[Jf( x) d: = 
\begin{pmatrix}
    \nabla f_1(x) ^\top d \\ 
    \nabla f_2(x) ^\top d \\
    \vdots \\ 
    \nabla f_p(x) ^\top d
\end{pmatrix},~ 
d^\top \nabla^2 f(x): = 
\begin{pmatrix}
    d^\top \nabla^2 f_1(x) \\ 
    d^\top \nabla^2 f_2(x)\\
    \vdots \\ 
    d^\top \nabla^2 f_p(x)
\end{pmatrix}  
~\mathrm{ and }~
d^\top \nabla^2 f(x) d: = 
\begin{pmatrix}
    d^\top \nabla^2 f_1(x) d \\ 
    d^\top \nabla^2 f_2(x) d \\
    \vdots \\ 
    d^\top \nabla^2 f_p(x) d
\end{pmatrix}, 
\]
respectively. Note that $Jf( x) d$ is an element of $\mathbb{R}^p$, $d^\top \nabla^2 f(x)$ is a matrix of order $ p \times n$, and $d^\top \nabla^2 f(x)d$ is an element of $\mathbb{R}^p$.

\item At an $x \in S$, for a given $\xi := (\xi_1, \xi_2, \ldots, \xi_p)^\top$ in $K$ or $K^*$, the notations $\la Jf(x) d, \xi \ra $ and $\la d^\top \nabla^2 f(x) d, \xi \ra $ are defined by
\begin{align*}
& \la Jf(x) d, \xi \ra : = \xi_1 \nabla f_1(x) ^\top d + \xi_2 \nabla f_2(x) ^\top d + \cdots + \xi_p \nabla f_p(x) ^\top d \\ 
\text{ and } &  
\la d^\top \nabla^2 f(x) d, \xi \ra : = \xi_1 d^\top \nabla^2 f_1(x) d + \xi_2 d^\top \nabla^2 f_2(x) d + \cdots + \xi_p d^\top \nabla^2 f_p(x) d,~ \mathrm{ respectively. } 
\end{align*}

\item At an $x \in S$, for a given $\xi := (\xi_1, \xi_2, \ldots, \xi_p)^\top$ in $K$ or $K^*$, the notations $\la Jf(x), \xi \ra $ and $\la \nabla^2f(x), \xi \ra$ are given by 
\begin{align*}
& \la Jf(x), \xi \ra : = \xi_1 \nabla f_1(x)^\top + \xi_2 \nabla f_2(x)^\top + \cdots + \xi_p \nabla f_p(x)^\top \\ 
\text{ and } &  \la \nabla^2 f(x), \xi \ra : = \xi_1 \nabla^2 f_1(x) + \xi_2 \nabla^2 f_2(x) + \cdots + \xi_p \nabla^2 f_p(x), \text{ respectively}.  
\end{align*}

\item Eigenvalues of a symmetric matrix $P$ of order $n \times n$ are denoted by 
$\lambda_1(P), \lambda_2(P), \ldots, \lambda_n(P)$, where 
\[\lambda_1(P) \le \lambda_2(P) \le \cdots \le \lambda_n(P). \]

\item In this study, the norm of a symmetric matrix $P$ is taken as $\|P\|:=  |\lambda_n(P)|$. 

\end{itemize}

\begin{definition} \emph{(Partial ordering in $\mathbb{R}^p$ \cite{drummond2014quadratically})}.  
The cone $K \subset \mathbb{R}^p$ induces the following partial ordering $(\preceq)$ and strict  ordering $(\prec):$ for $u, v \in \mathbb{R}^p$, 
\[ 
u \preceq v \Longleftrightarrow u - v \in -K ~~\mathrm{ and }~~ u \prec v \Longleftrightarrow u - v \in - \mathrm{int}(K). 
\]
The notations $u \succeq v$ and $u \succ v$ represent $v \preceq u$ and $v \prec u$, respectively.  
\end{definition}

\medskip
\begin{definition}
\emph{(Weakly efficient point \cite{drummond2014quadratically})}. 
A point $\bar x \in S$ is called a weakly efficient point of \eqref{vop} if there is no $x \in S$ such that $f(x) \prec f(\bar x)$. \\ 
\end{definition}

\begin{definition}
\emph{(Pareto critical or stationary point)}. Let $f:S \to \mathbb{R}^p$ be continuously differentiable. A point $\bar x \in S$ is called a Pareto critical point or a stationary point of \eqref{vop} if 
\[ R(Jf(\bar x)) \cap (- \mathrm{int}(K)) = \emptyset, \]
\medskip
where $R(Jf(\bar x))$ is the range space of the linear map $Jf(\bar x) : \mathbb{R}^n \to \mathbb{R}^p$. 
\end{definition}

\medskip
\begin{lemma} \emph{\cite{drummond2014quadratically}} Let $f:S \to \mathbb{R}^p$ be continuously differentiable. If $\bar x \in S$ is a weakly efficient point of \eqref{vop}, then $\bar x$ is a stationary point of \eqref{vop}. 
\end{lemma}

\medskip
\begin{framed}
\noindent 
Throughout the article, we assume the following conditions to generate an iterative scheme for generating weakly efficient points of the problem \eqref{vop}.

\medskip
\begin{enumerate}[(i)] 
\item The objective function $f: S \to \mathbb{R}^p$ is twice continuously differentiable in $S$. 
\item \label{assumption_ii} For each of $f_1, f_2, \ldots, f_p$, the Hessian is Lipschitz continuous on $S$ with a common Lipschitz constant $L > 0$, i.e., for each $i = 1, 2, \ldots, p$, 
\[\| \nabla^2 f_i(x) - \nabla^2 f_i(y) \| \le L \|x - y\| \text{ for all } x, y \in S. \]
\item The set $S$ is so large that it contains the level set 
\[\mathcal{L}(f(x^0)) := \{x \in \mathbb{R}^n: f(x) \preceq f(x^0)\}\]
for the chosen initial point $x^0$ in int$(S)$. \item The ordering cone $K$ is finitely generated, i.e., there exists a compact set $C \subset K^* \cap \mathbb{S}^{p - 1}$ such that cone(conv$(C)) = K^*$, where $$\mathbb{S}^{p - 1} := \{\xi \in \mathbb{R}^p: \|\xi\| = 1\} \text{ and } 
K^* := \{ u \in \mathbb{R}^p : \la u, \xi \ra \ge 0 \text{ for all } \xi \in K \}.  $$  
We let $C := \{\xi^1, \xi^2, \ldots, \xi^\ell\}$.   
\end{enumerate}

\end{framed}

\medskip
The cone $K$ and its interior can be presented by the above subbase $C$ of $K^*$ in the following way: 
\begin{align*} 
K & = \{ u \in \mathbb{R}^p : \la u, \xi \ra \ge 0 \text { for all } \xi \in C\} \\ 
\text{ and } \mathrm{int}(K) & = \{ u \in \mathbb{R}^p : \la u, \xi \ra > 0 \text { for all } \xi \in C\}.
\end{align*}

Note that if $\bar x \in S$ is a minimum point of the function $\la f(x), \bar \eta \ra$ in $S$, for some $\bar \eta \in C$, then 
\begin{align}
& \la f(\bar x), \bar \eta \ra \le \la f(x), \bar \eta \ra \text{ for all } x \in S \notag \\ 
\implies & \la f(\bar x) - f(x), \bar \eta \ra \le 0  \text{ for all } x \in S \notag. 
\end{align}
This implies that  
\begin{align}
 & \not\exists ~ x \in S : \la f(\bar x) - f(x), \eta \ra > 0 \text{ for all } \eta \in C \notag \\ 
\implies & \not\exists ~ x \in S : f(x) \prec f(\bar x),  \notag   
\end{align}
i.e., $\bar x$ is a weakly efficient point of $\min_{x \in \mathbb{R}^n} f(x)$.  Thus, we have the following result. \\ 

\begin{lemma}\label{for_one_xi_is_weak_min}
If $\bar x$ is a minimum point of the function $\la f(x), \eta \ra$ in $S$ for some $\eta$ in $C$, then $\bar x$ is a weakly efficient point of $\min_{x \in S} f(x)$. \\ 
\end{lemma}

\begin{definition}
\emph{(Descent direction \cite{drummond2014quadratically}).} At an $x \in S$, the direction $d \in \mathbb{R}^n$ is called a $K$-descent direction of $f: S \to \mathbb{R}^p$ (or simply a descent direction) if there exists $\bar \alpha > 0$ such that 
\[ f(x + \alpha d) \prec f(x) \text{ for all } \alpha \in (0, \bar \alpha]. \]
\end{definition}

Under the assumption (\ref{assumption_ii}), we have the following result, which will be useful throughout the paper. \\

\begin{lemma}
For any $x$ and $y$ from $S$, there hold the following inequalities for all $\xi$ in $C$:  
\begin{align}
& \left\| \la Jf(y) - Jf(x) - (y - x)^\top \nabla^2 f(x),~ \xi \ra \right\| \le \frac{L}{2} \|y - x\|^2   \label{square_equation} \\  
\text{ and } & \left| \la f(y) - f(x) - Jf(x) (y - x) - \frac{1}{2} (y - x)^\top \nabla^2 f(x) (y - x), ~\xi \ra \right| \le \frac{L}{6} \|y - x\|^3. \label{cube_equation} 
\end{align}
\end{lemma}

\begin{proof}
Define a function $F: [0, 1] \to \mathbb{R}^p$ by $F(t) := f(x + t(y - x)).$ Then, $F$ is continuously differentiable on $[0, 1]$ and 
\begin{equation}\label{aux1_03_03_25}
f(y) - f(x) = \int_0^1 F'(t) ~\mathrm{d}t = \int_0^1 Jf(x + t(y - x)) (y - x) ~\mathrm{d}t. 
\end{equation}
\medskip
\noindent
Similarly, by defining a function $G: [0, 1] \to \mathbb{R}^{pn}$ as $G(t) := Jf(x + t (y - x))$, we obtain that 
\[Jf(y) - Jf(x) = \int_0^1 (y - x)^\top \left(\nabla^2 f(x + t(y - x))\right)~ \mathrm{d} t. \]
Therefore, 
\begin{align*}
~&~ \left\| \la Jf(y) - Jf(x) - (y - x)^\top \nabla^2 f(x),~ \xi \ra \right\| \\ 
= ~&~ \left\| \la \int_0^1  (y - x)^\top \left( \nabla^2 f(x + t(y - x)) - \nabla^2 f(x) \right) ~ \mathrm{d}t, ~ \xi \ra \right\|  \\ 
\overset{\eqref{assumption_ii}}{\le}  ~&~ L \| x - y \|^2 \| \xi \| \int_0^1 t ~ \mathrm{d}t  =   \frac{L}{2} \| y - x \|^2.  
\end{align*}
Further, 
\begin{align*}
~&~ \left| \la f(y) - f(x) - Jf(x) (y - x) - \frac{1}{2} (y - x)^\top \nabla^2 f(x) (y - x), ~\xi \ra \right| \\ 
\overset{\eqref{aux1_03_03_25}}{=} ~&~ \left| \la \int_0^1 Jf(x + t(y - x))(y - x) \mathrm{d}t - Jf(x)(y - x) - \frac{1}{2}(y - x)^\top \nabla^2f(x) (y - x),~ \xi \ra \right| \\ 
= ~&~ \left| \la \int_0^1 \left( Jf(x + t (y - x)) - Jf(x) - t(y - x)^\top \nabla^2 f(x)\right) (y - x) ~\mathrm{d} t,~ \xi \ra \right| \\
\overset{\eqref{square_equation}}{\le} ~&~ \frac{L}{2} \|y - x\|^2 \| \xi \| \int_0^1 t^2 \|y - x\| ~ \mathrm{d}t \\   
\le ~&~ \frac{L}{6} \|y - x\|^3.      
\end{align*}
\end{proof}

\section{Cubic Regularization} \label{section-cubic-regularization}

The conventional Newton method for vector optimization \cite{drummond2014quadratically} finds the Newton direction at $x$ in int$(S)$ by 
\begin{equation}\label{newton_direction}
d(x):= \underset{d \in \mathbb{R}^n}{\argmin} ~ \underset{\xi \in C}{\max} \left\{ \la Jf(x) d, \xi \ra + \frac{1}{2} \la d^\top \nabla^2 f(x) d, \xi \ra \right\}. 
\end{equation}
If the function $f$ is not strongly $K$-convex, the Newton direction may diverge, or there may be just local convergence. Thus, several quasi-Newton \cite{kumar2023quasi}, conjugate gradient \cite{perez2018nonlinear}, and trust-region \cite{carrizo2016trust} approaches have been proposed in the literature. Several limitations of these methods are mentioned in Section \ref{introduction}. In this study, we thus aim to propose a cubic regularization of the Newton method.

Let $M>0$ be a real parameter. For each $x$ in $S$, we define an auxiliary function $q_M(x, \cdot): \mathbb{R}^n \to \mathbb{R}$ by 
\begin{equation}\label{cubic_function} 
q_M(x, d) := \underset{\xi \in C}{\max} \left\{ \la Jf(x) d, \xi \ra + \frac{1}{2} \la d^\top \nabla^2 f(x) d, \xi \ra \right\} + \frac{M}{6} \|d\|^3.    
\end{equation}
Note that the function $q_M(x, \cdot)$ is well-defined because $C$ is a compact set, and the function inside the max operator is continuous in $d$. However, it must be noted that the function $q_M(x, \cdot)$ is non-smooth. Also, notice that the only difference in the expressions inside the max operators in \eqref{newton_direction} and \eqref{cubic_function} is the appearance of the \emph{cubic} term $\frac{M}{6} \|d\|^3$ in \eqref{cubic_function}. We define a revised Newton direction as 
\begin{equation}\label{cubic_direction} 
d_M(x) := \underset{d \in \mathbb{R}^n}{\argmin} ~ q_M(x, d)
\end{equation} 
and refer the direction $d_M(x)$ as a  \emph{cubic regularized} Newton direction at the point $x$ in int$(S)$. Here, the argmin in \eqref{cubic_direction} indicates that $d_M(x)$ is a \emph{global} minimizer of $q_M(x, \cdot)$, and the equality $(:=)$ in \eqref{cubic_direction} means that $d_M(x)$ is just one of the global minimizers of $q_M(x, \cdot)$. We denote the global minimum value of $q_M(x, \cdot)$ as  $\beta_M(x)$, i.e., 
\begin{align}\label{beta_function}
\beta_M(x) := ~ & \underset{d \in \mathbb{R}^n}{\min} \left( \underset{\xi \in C}{\max} \left\{ \la Jf(x) d, \xi \ra + \frac{1}{2} \la d^\top \nabla^2 f(x) d, \xi \ra \right\} + \frac{M}{6} \|d\|^3 \right)  \\ 
= ~ & \underset{d \in \mathbb{R}^n}{\min} ~ q_M(x, d) = q_M(x, d_M(x)). \notag  
\end{align}

Next, we present some basic results that relate the ``stationarity" or ``Pareto criticality" of a point $x$ with the value of $\beta_M(x)$ and the cubic regularized Newton direction $d_M(x)$.  \\ 

\begin{theorem}\label{stationarity_and_beta_value} 
Let $d_M : S \to \mathbb{R}^n$ and $\beta_M : S \to \mathbb{R}$ be the functions as defined in \eqref{cubic_direction} and \eqref{beta_function}, respectively. Then, the following results hold.  
\begin{enumerate}
\item[\textnormal{(i)}] $\beta_M(x) \le 0$ for any $x \in S$.
\item[\textnormal{(ii)}] If $\bar x \in S$ is not a stationary point of \eqref{vop}, then $\beta_M(\bar x) < 0$ and $d_M(\bar x) \neq 0$.  
\item[\textnormal{(iii)}] If $\beta_M(\bar x) = 0$ for some $\bar x \in S$, then $\bar x$ is a stationary point of \eqref{vop}.
\end{enumerate}
\end{theorem}

\begin{proof} 
(i) 
Note from \eqref{beta_function} that for any $x$ in $S$, 
\[\beta_M(x) = \underset{d \in \mathbb{R}^n}{\min}~ q_M(x, d) \le q_M(x, 0) = 0. \]

\noindent
(ii)  Let $\bar x$ be not a stationary point of \eqref{vop}. Then, there exists $h$ in $\mathbb{R}^n$ such that
$Jf(\bar x) h \in - \text{int} (K)$. Thus, 
\begin{equation}\label{aux_15_12_1} 
\la Jf(\bar x) h, \xi \ra < 0 \text{ for all } \xi \in C. 
\end{equation} 
We show that there exists a real number $\bar t > 0$ such that the direction $\bar h := \bar t h$ satisfies 
\begin{equation}\label{aux_a5_12_2}
\la Jf(\bar x) \bar h, \xi \ra + \frac{1}{2}\la \bar{h}^\top \nabla^2 f(\bar{x}) \bar{h}, \xi \ra + \frac{M}{6} \|\bar{h}\|^3 < 0 \text{ for all } \xi \in C.  
\end{equation}
Fix any $\xi$ in $C$ and define 
\[a := \la Jf(\bar x) h, \xi \ra,~ 
b := \frac{1}{2}\la h^\top \nabla^2 f(\bar{x}) h, \xi \ra, \text{ and } 
c := \frac{M}{6} \|h\|^3. \]
Note from \eqref{aux_15_12_1} that $a < 0$ and $h \neq 0$. As $h \ne 0$, we have $c > 0$. \\
\medskip
Denote the roots of the equation $\frac{a}{c} + \frac{b}{c} t + t^2 = 0$ by 
\[r_1 : = \frac{1}{2} \left( -\frac{b}{c} - \sqrt{\frac{b^2}{c^2} - \frac{4a}{c}}\right) \text{ and } r_2 : = \frac{1}{2} \left( -\frac{b}{c} + \sqrt{\frac{b^2}{c^2} - \frac{4a}{c}}\right). \]
As $\frac{a}{c} < 0 $, evidently $r_1 < 0$ and $r_2 > 0$. So, we notice that 
\begin{align*}
~&~ a + b t + ct^2 < 0 \\     
\Longleftrightarrow ~&~ c (t - r_1)(t - r_2) < 0 \\ 
\Longleftrightarrow ~&~ r_1 < t < r_2. 
\end{align*}
Hence, for any $t$ in $(0, r_2)$, we have 
\begin{align*}
~&~ t(a + bt + c t^2) < 0 \\ 
\text{i.e.,} ~&~ 
t \la Jf(\bar x) h, \xi \ra   
+ \frac{t^2}{2}\la h^\top \nabla^2 f(\bar{x}) h, \xi \ra 
+ \frac{M t^3}{6} \|h\|^3 < 0. 
\end{align*}
Taking $\bar t = \frac{r_2}{2} > 0$ and $\bar h = \bar t h$, we see that 
\[ \la Jf(\bar x) \bar h, \xi \ra   
+ \frac{1}{2} \la \bar{h}^\top \nabla^2 f(\bar{x}) \bar{h}, \xi \ra 
+ \frac{M}{6} \|\bar{h}\|^3 < 0.\]  
Arbitrariness of $\xi$ in $C$ implies \eqref{aux_a5_12_2}. From \eqref{aux_a5_12_2}, we have $q_M(\bar x, \bar h) < 0$. Hence, 
\[ \beta_M(\bar x) = \underset{d \in \mathbb{R}^n}{\min} ~ q_M(\bar x, d) \leq q_M(\bar x, \bar h) < 0. \]
As $\beta_M(\bar x) = q_M(\bar x, d_M(\bar x))$ and $\beta_M(\bar x) < 0$, we trivially have $d_M(\bar x) \neq 0$. Hence, the result follows. \\

\noindent
(iii) From the contrapositive statement of (ii), we observe that if $\beta_M(\bar x) \ge 0$, then $\bar x$ is a stationary point of \eqref{vop}. However, from (i), $\beta_M(\bar x) \ge 0$ is equivalent to $\beta_M(\bar x) = 0$. Hence, the statement in (iii) follows.

\end{proof}

\begin{theorem}\label{decrement_if_M_bigger_than_L}
If $M \ge L$ and $\bar x$ is such a point in $S$ that $\beta_M(\bar x) < 0$, then $d_M(\bar x) \neq 0$ and $$f(\bar x + d_M(\bar x)) \prec f(\bar x).$$
\end{theorem}

\begin{proof}
From $\beta_M(\bar x) < 0$, it is trivially followed that $d_M(\bar x) \neq 0$ because $d_M(\bar x) = 0$ implies $\beta_M(\bar x) = 0$.  \\ 
Let $\xi'$ be an element in $C$. From the inequality \eqref{cube_equation}, we have
\begin{equation}\label{aux_15_12_3}
\la f(\bar x + d_M(\bar x)) - f(\bar x),~ \xi' \ra - \la Jf(\bar x) d_M(\bar x),~ \xi' \ra - \frac{1}{2} \la d_M(\bar x)^\top \nabla^2 f(\bar x) d_M(\bar x),~ \xi' \ra \le \frac{L}{6} \|d_M(\bar x)\|^3.     
\end{equation}
As $\beta_M(\bar x) < 0$, we get from the definition of $\beta_M(\bar x)$ that 
\begin{equation}\label{aux_15_12_40}
\la Jf(\bar x) d_M(\bar x),~ \xi' \ra + \frac{1}{2} \la d_M(\bar x)^\top \nabla^2 f(\bar x) d_M(\bar x),~ \xi'\ra + \frac{M}{6} \|d_M(\bar x)\|^3 < 0.     
\end{equation}
Adding \eqref{aux_15_12_3} and \eqref{aux_15_12_40}, we obtain from $M \ge L$ that 
\[\la f(\bar x + d_M(\bar x)) - f(\bar x),~ \xi' \ra < \frac{L - M}{6} \|d_M(\bar x)\|^3 \le 0  \]
Arbitrariness of $\xi'$ from $C$ shows that  
\begin{align*}
~&~ \la f(\bar x + d_M(\bar x)) - f(\bar x),~ \xi \ra < 0 \text{ for all } \xi \in C\\ 
\text{i.e.}, ~&~ f(\bar x + d_M(\bar x)) - f(\bar x) \in -\text{int}(K) \\ 
\text{i.e.}, ~&~ f(\bar x + d_M(\bar x)) \prec f(\bar x). 
\end{align*}
\end{proof}

Note that if $\bar x \in \mathcal{L}(f(x^0)) \subseteq S$, then from Theorem \ref{decrement_if_M_bigger_than_L}, we obtain $\bar x + d_M(\bar x) \in S$. So, there is neither a well-definedness issue for $f(\bar x + d_M(\bar x))$ nor an infeasibility issue of $\bar{x} + d_M(\bar x)$ for the problem \eqref{vop}.    \\

\begin{remark}\label{aux1_17_02}
Theorem \ref{decrement_if_M_bigger_than_L} does not necessarily imply that if  $ M \ge L$ and $\beta_M(\bar x) < 0$, then $d_M(\bar x)$ is a descent direction of $f$ at $\bar x$. For example, take $S = [-1.5, 1]$, $K = \mathbb{R}^2_+$ and $f: S \to \mathbb{R}^2$ as 
\[f(x) := \begin{pmatrix} f_1(x) \\ f_2(x) \end{pmatrix} := \begin{pmatrix} x^2 + 4 \sin x \\ x^3 - 2 x^2 \end{pmatrix}. \]

\begin{figure}[]
\centering
\begin{subfigure}[b]{0.45\textwidth}
\centering 
\begin{tikzpicture}
\begin{axis}[
color= black, 
thick,
xmin=-2.5, 
xmax=3.5, 
ymin=-3.5, 
ymax=3.7, 
axis equal image, 
font=\footnotesize,
xtick distance=1,
ytick distance=1,
inner axis line style={stealth-stealth},
xlabel = {$x$}, 
ylabel = {$y$}, 
axis x line=middle,
axis y line=middle 
] 
\addplot [
    domain=-1.7:1.8, 
    samples=200, 
    color=purple,
]
{x^2 + 4*sin(deg(x))};
\addplot [
    domain=-1.9:2.8, 
    samples=200, 
    color=blue,
    ]
    {x^3 - 2*x^2}; 
\end{axis}
\node at (2.5,4) {\footnotesize $f_1$};
\node at (4,4) {\footnotesize $f_2$};
\end{tikzpicture} 
\caption{Graphs of $f_1$ and $f_2$}\label{subfig_1_a}
\end{subfigure}%
\qquad 
\begin{subfigure}[b]{0.45\textwidth}
\centering
\begin{tikzpicture} 
\begin{axis}[
color= black, 
thick,
xmin=-.39, 
xmax=0.08, 
ymin=-0.028,
ymax = .03, 
axis lines=middle, 
font=\footnotesize,
xtick distance=0.1,
ytick distance=.05,
inner axis line style={stealth-stealth},
axis x line=middle,
axis y line=middle, 
xlabel = {$d$},
ylabel = {$q_M$}
] 
\addplot [
    thick, 
    domain=-0.375:.05, 
    samples=200, 
    color=red,
]{max(4.0768*x+0.92*x^2,-0.1552*x-1.88*x^2)+4*abs(x)^3}; 
\coordinate (1) at (140,6);  
\filldraw[black] (125,11.5) circle(0.05cm);
\node at (140,7.5) {\footnotesize $(-0.264, -0.016)$};  
\end{axis}
\end{tikzpicture} 
\caption{The graph of $q_M(\bar x, d)$ at $\bar x = 0.04$ for $M = 24$}\label{subfig_1_b}
\end{subfigure}
\caption{Objective components and the function $q_M(0.04, d)$ (Remark \ref{aux1_17_02})}\label{figure_in_remark}
\end{figure}
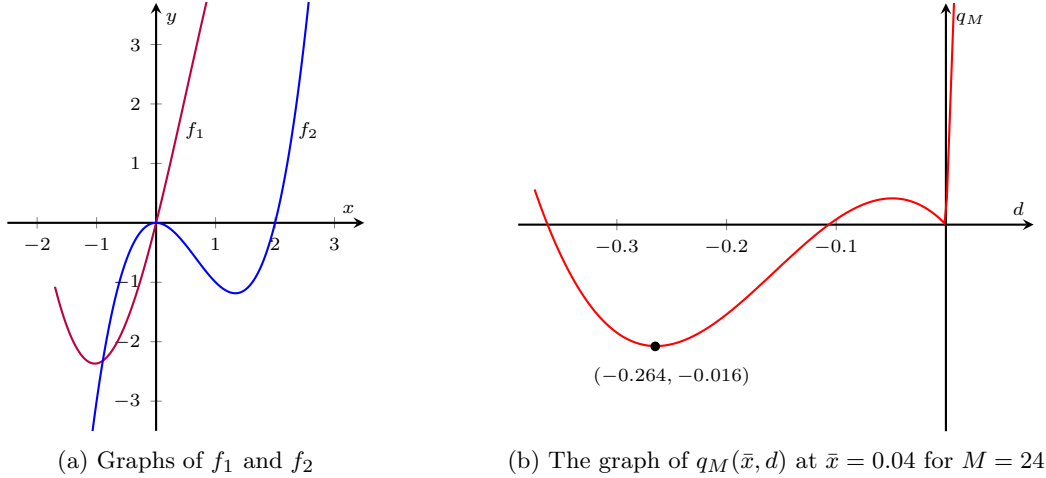

\noindent
Here, Hessians of $f_1$ and $f_2$ are Lipschitz continuous on $S$ with a common Lipschitz constant $L = 6$. We choose $M = 24 \ge L$ and $\bar x = 0.04$. For this $M$, we have 
\begin{align*}
~&~ q_M(\bar x, d) = \max \left\{4.0768 d+0.92 d^2,-0.155 d-1.88 d^2\right\} + 4 |d|^3,~ \\ 
~&~ d_M(\bar x) = \underset{d \in \mathbb{R}}{\emph{\argmin}}~ q_M(\bar x, d) = -0.264~\text{ and }~ \beta_M(\bar x) = - 0.016 < 0. 
\end{align*}
The graphs of $f_1, f_2$, and $q_M(\bar x, \cdot)$ are depicted in Figure \ref{figure_in_remark}. Note here that although 
\[f(\bar x + d_M(\bar x)) = 
f(-0.224) = \begin{pmatrix} -0.838 \\ -0.116\end{pmatrix} 
< \begin{pmatrix} 0.162 \\ -0.003\end{pmatrix}= f(\bar x), \] but from Figure \ref{subfig_1_a}, we clearly see that 
$f(\bar x + \alpha d_M(\bar x)) \not\prec f(\bar x) \text{ for all } \alpha \in [0, 0.1).$ Hence, $d_M(\bar x)$ is not a descent direction of $f$ at $\bar x = 0.04$. \\ 
\end{remark}

Even if the observation in Remark \ref{aux1_17_02} holds, note that if $M$ is chosen no smaller than $L$, then Theorems \ref{stationarity_and_beta_value} and \ref{decrement_if_M_bigger_than_L} together give an attractive observation towards developing a ``descent" iterative method ``without requiring convexity of $f$'' to find stationary points of \eqref{vop}: 
\medskip
\begin{enumerate}[Step 1.]
\item \label{three-step-process}
Take $M_k \ge L$ for all $k = 0, 1, 2, \ldots$. 
\item Start from any initial point $x^0 \in S$, and keep generating $x^{k + 1}$ from $x^k$ by \[x^{k + 1} : = x^k + d_{M_k}(x^k). \]
\item Stop if $\beta_{M_k}(x^k) \ge 0$.
\end{enumerate}

\medskip
According to Theorem \ref{decrement_if_M_bigger_than_L}, this iterative process has \emph{descent property} because $f(x^k + d_{M_k}(x^k)) \prec f(x^k)$ for all $k$; moreover, from Theorem \ref{stationarity_and_beta_value} (iii), this process ends producing a stationary point of $f$. 
(The convergence $\lim_{k \to \infty} {\beta}_{M_k}(x_k) = 0$ is ensured by Theorem \ref{aux2_28_01_25}).

\medskip
Although this observation is attractive, our aim is to find not just stationary points but weakly efficient points of \eqref{vop}. In addition, difficult parts of this iterative process are 
\begin{enumerate}[(i)]
\item an estimation of the Lipschitz constant $L$ and 
\item an effective way to compute $d_{M_k}(x^k)$. 
\end{enumerate}

\medskip 
Note that this iterative process ends at a stationary point of $f$ because of the stopping condition $\beta_{M_k}(x^k) \ge 0$. So, this stopping condition must be appropriately revised to ensure that the method terminates at a weakly efficient point. Such a stopping condition is proposed later with the help of $\mu_M$ (as defined in \eqref{mu_M_definition}) instead of $\beta_M$.

\medskip
As an estimation of the Lipschitz constant is difficult, below we also propose a way to bypass using $M \ge L$ yet having the descent property of the generated sequence $\{x^k\}$, i.e., $f(x^{k + 1}) \prec f(x^k)$. This discussion is deferred until Lemma \ref{upper_bound_of_J}, where we define an auxiliary function $h_M$ and use it to ensure the descent property of the generated sequence by the proposed method (in Algorithm \ref{algo}).

\medskip 
For an effective computation of $d_M(\bar x)$ at a given $\bar x$ in int$(S)$, we describe a process below. Note that the problem in \eqref{cubic_direction} is non-convex and may have multiple minima. So, how do we solve it? We show below (in Theorem \ref{d_M_is_one_dim_convex_min} and Remark \ref{lemma_d_M_is_1d_convex}) that it is a one-dimensional convex optimization problem. We need to have a few observations and lemmas to arrive at this result, as presented below.   \\

At a given $\bar x$ from $S$, to find a cubic regularized Newton direction $d_M(\bar x)$, according to \eqref{cubic_direction}, we need to find a global minimum point of $q_M(\bar x, \cdot) : \mathbb{R}^n \to \mathbb{R}$. Note that $q_M(\bar x, \cdot)$ is a non-smooth function. So, $d_M(\bar x)$ necessarily satisfies 
\begin{equation}\label{aux_15_12_4}
0 \in \partial q_M(\bar x, d_M(\bar x)),     
\end{equation}
where $\partial q_M(\bar x, d_M(\bar x))$ is the subdifferential of $q_M(\bar x, \cdot)$ at $d_M(\bar x)$. Define a set 
\[ \mathscr{A}(d_M(\bar x)) := 
\left\{ \eta \in C: q_M(\bar x, d_M(\bar x)) = \la Jf(\bar x) d_M(\bar x), \eta \ra + \frac{1}{2} \la d_M(\bar x)^\top \nabla^2 f(\bar x) d_M(\bar x), \eta \ra + \frac{M}{6} \|d_M(\bar x)\|^3  \right\}. \]
As $C$ is finite, the set $\mathscr{A}(d_M(\bar x))$ is finite. Let 
\begin{equation} \label{aux4_01_03_25}
\mathscr{A}(d_M(\bar x)) = \left\{\eta^1(\bar x), \eta^2(\bar x), \ldots, \eta^{s(\bar x)}(\bar x) \right\}. 
\end{equation}
Then, the inclusion relation \eqref{aux_15_12_4} implies that there exist  real constants $\lambda_i(\bar x) \ge 0$, $i = 1, 2, \ldots, s(\bar x)$, with $\sum_{i = 1}^{s(\bar x)} \lambda_i(\bar x) = 1$ such that 
\begin{align} 
~&~\sum_{i = 1}^{s(\bar x)} \lambda_i(\bar x) \left\{\la Jf(\bar x), \eta^i(\bar x) \ra + \la \nabla^2f(\bar x) d_M(\bar x), \eta^i(\bar x) \ra + \frac{M}{2} \|d_M(\bar x)\| d_M(\bar x) \right\} = 0,  \label{replica_of_2_5} \\ 
\text{i.e., } ~&~ \left( \sum_{i = 1}^{s(\bar x)} \lambda_i(\bar x) \left\{ \left(\eta^i_1(\bar x) \nabla^2 f_1(\bar x) +  \eta^i_2(\bar x) \nabla^2 f_2(\bar x) + \cdots +  \eta^i_p(\bar x) \nabla^2 f_p(\bar x)\right) + \frac{M}{2} r_M(\bar x) I_n \right\} \right) d_M(\bar x) \notag \\ 
~&~ = - \sum_{i = 1}^{s(\bar x)} \lambda_i(\bar x) \la Jf(\bar x), \eta^i(\bar x) \ra, \notag \\ 
\text{i.e., } ~&~ \left( \sum_{i = 1}^{s(\bar x)} \lambda_i(\bar x) \left\{ \la \nabla^2 f(\bar x), \eta^i(\bar x) \ra + \frac{M}{2} r_M(\bar x) I_n \right\}\right) d_M(\bar x) = - \sum_{i = 1}^{s(\bar x)} \lambda_i(\bar x) \la Jf(\bar x), \eta^i(\bar x) \ra. \label{10_01_25_aux1} 
\end{align}
As $\la \nabla^2 f(\bar x), \eta^i(\bar x) \ra + \frac{M}{2} r_M(\bar x) I_n$ is invertible (see Corollary \ref{invertable_for_all_xi}), we get from \eqref{10_01_25_aux1} that 
\begin{align} 
d_M(\bar x) = & - \left( \sum_{i = 1}^{s(\bar x)} \lambda_i(\bar x) \left\{ \la \nabla^2 f(\bar x), \eta^i(\bar x) \ra + \frac{M}{2} r_M(\bar x) I_n \right\}\right)^{-1} \sum_{i = 1}^{s(\bar x)} \lambda_i (\bar x) \la Jf(\bar x), \eta^i(\bar x) \ra \notag \\ 
= & - \left( \la \nabla^2 f(\bar x), \eta(\bar x) \ra + \frac{M}{2} r_M(\bar x) I_n \right)^{-1} \la Jf(\bar x), \eta(\bar x) \ra,   
\label{10_01_25_aux2}
\end{align} 
where $\eta(\bar x) = \sum_{i = 1}^{s(\bar x)} \lambda_i(\bar x) \eta^i(\bar x)$. \\

At a given $x$ in $S$, define a set 
\[\mathscr{D}_x := \left\{r \in \mathbb{R}_+ : \la \nabla^2 f(x),~ \xi \ra + \frac{M}{2} r I_n \text{ is positive definite for all } \xi \in C \right\}. \]
It is to observe that the set $\mathscr{D}_x$ is non-empty and unbounded  for any $x \in S$ since for any $r$ that satisfies 
\[ r > - \frac{2}{M} \min_{\xi \in C} ~ \lambda_1 \left(\la \nabla^2 f(x),~ \xi \ra\right), \]
the matrix $\la \nabla^2 f(x),~ \xi \ra + \frac{M}{2} r I_n$ is positive definite for all $\xi \in C$. \\

\begin{theorem}\label{d_M_is_one_dim_convex_min}
For any $\bar x \in \mathrm{int} (S)$ and $M>0$, the following equality holds: 
\begin{equation}\label{10_1_25_aux3} 
\min_{z \in \mathbb{R}^n} v_u(z) = \sup_{r \in \mathscr{D}_{\bar{x}}} v_l(r), 
\end{equation} 
where the functions $v_u: \mathbb{R}^n \to \mathbb{R}$ and $v_l: \mathscr{D}_{\bar x} \to \mathbb{R}$ are defined by 
\begin{align}
& v_u(z) := \max_{\xi \in C} 
 \la Jf(\bar x) z + \frac{1}{2} z^\top  \nabla^2 f(\bar x) z, \xi \ra + \frac{M}{6} \| z \|^3, \notag \\ 
& v_l(r) := \frac{1}{2} \la d(\bar x, r), \sum_{i = 1}^{s(\bar x)} \lambda_i(\bar x) \la Jf(\bar x), \eta^i(\bar x) \ra \ra - \frac{M}{12} r^3,  \label{v_l_expression}\\
\text{ and }~ & d(\bar x, r) = - \left( \sum_{i = 1}^{s(\bar x)} \lambda_i(\bar x) \left\{ \la \nabla^2 f(\bar x), \eta^i(\bar x) \ra + \frac{M}{2} r ~ I_n \right\}\right)^{-1} \sum_{i = 1}^{s(\bar x)} \lambda_i(\bar x) \la Jf(\bar x), \eta^i(\bar x) \ra.  \label{d_bar_x_r}
\end{align}

\end{theorem}

\begin{proof}
\begin{align}\label{21_01_25_aux5}
\min_{z \in \mathbb{R}^n} v_u(z) 
& = \min_{z \in \mathbb{R}^n \atop  \|z\|^2 = \alpha} 
\left[ \max_{\xi \in C} \la Jf(\bar x)z + \frac{1}{2} z^\top \nabla^2 f(\bar x) z, \xi \ra + \frac{M}{6} \alpha^{3/2}   \right] \notag  \\ 
& = \min_{ z \in \mathbb{R}^n \atop \alpha \in \mathbb{R}_+} \sup_{\gamma \in \mathbb{R}} \left[ \max_{\xi \in C} \la Jf(\bar x)z + \frac{1}{2} z^\top \nabla^2 f(\bar x) z, \xi \ra + \frac{M}{6} \alpha^{3/2} + \frac{M}{4} \gamma \left(\|z\|^2 - \alpha \right)  \right] \notag \\ 
& \ge \sup_{\gamma \in \mathbb{R}} \min_{ z \in \mathbb{R}^n \atop \alpha \in \mathbb{R}_+}  \left[ \max_{\xi \in C} \la Jf(\bar x)z + \frac{1}{2} z^\top \nabla^2 f(\bar x) z, \xi \ra + \frac{M}{6} \alpha^{3/2} + \frac{M}{4} \gamma \left(\|z\|^2 - \alpha \right)  \right] \notag \\ 
& \ge \sup_{r \in \mathscr{D}_{\bar x}} \min_{z \in \mathbb{R}^n \atop \alpha \in \mathbb{R}_+}  \left[ \max_{\xi \in C} \la Jf(\bar x)z + \frac{1}{2} z^\top \nabla^2 f(\bar x)z, \xi \ra + \frac{M}{6} \alpha^{3/2} + \frac{M}{4} r \left(\|z\|^2 - \alpha \right)  \right] \notag \\ 
& = \sup_{r \in \mathscr{D}_{\bar x}} \min_{ z \in \mathbb{R}^n \atop \alpha \in \mathbb{R}_+} \vartheta (z, \alpha, r),    
\end{align}    
where the function $\vartheta : \mathbb{R}^n \times \mathbb{R}_+ \times \mathscr{D}_{\bar x} \to \mathbb{R}$ is given by 
\begin{align*} 
\vartheta (z, \alpha, r) & : = \max_{\xi \in C} \la Jf(\bar x)z + \frac{1}{2} z^\top \nabla^2 f(\bar x) z, \xi \ra + \frac{M}{6} \alpha^{3/2} + \frac{M}{4} r \left(\|z\|^2 - \alpha \right) \\
& = \max_{\xi \in C} \left[ \la Jf(\bar x)z + \frac{1}{2} z^\top \nabla^2 f(\bar x) z, \xi \ra + \frac{M}{4} r  \|z\|^2 \right] + \frac{M}{6} \alpha^{3/2} - \frac{M r \alpha }{4} \\ & = \max_{\xi \in C} \left[ \la Jf(\bar x)z, \xi \ra + \frac{1}{2} z^\top \left( \la \nabla^2 f(\bar x), \xi \ra + \frac{1}{2} M r I_n \right) z \right] + \frac{M}{6} \alpha^{3/2} - \frac{M r \alpha }{4}. 
\end{align*}
Note that for any given $r \in \mathscr{D}_{\bar x}$, the function $\vartheta (\cdot, \cdot, r) : \mathbb{R}^n \times \mathbb{R} \to \mathbb{R}$ is strictly convex due to the following two observations:  
\begin{enumerate}[(i)]
\item For each $r \in \mathscr{D}_{\bar x}$ and $\xi \in C$, the matrix $\la \nabla^2 f(\bar x), \xi \ra + \frac{Mr}{2} I_n$ is positive definite, and hence 
\[\max_{\xi \in C} \left[ \la Jf(\bar x)z, \xi \ra + \frac{1}{2} z^\top \left( \la \nabla^2 f(\bar x), \xi \ra + \frac{1}{2} M r I_n \right) z \right] \] 
is a strictly convex function of $z \in \mathbb{R}^n$.

\item For each $r \in \mathscr{D}_{\bar x} \subseteq [0, \infty)$, the function 
\[\alpha \mapsto \frac{M}{6} \alpha^{3/2} - \frac{M r \alpha }{4} \]
is convex on $\mathbb{R}_+$. \\ 
\end{enumerate}

\noindent
Therefore, for any given $\bar r \in \mathscr{D}_{\bar x}$, the function $\vartheta (\cdot, \cdot, \bar r) : \mathbb{R}^n \times \mathbb{R} \to \mathbb{R}$ gets its (global) minimum value at a point $(\tilde z, \tilde \alpha) \in \mathbb{R}^n \times \mathbb{R}_+$ at which 
\[0 \in \partial \vartheta (\widetilde z, \widetilde \alpha, \bar r). \]
Noting  that 
\[\vartheta (z, \alpha, r) = q_M(\bar x, z) - \frac{M}{6} \|z\|^3 + \frac{M}{6} \alpha^{3/2} + \frac{M}{4} r \left(\|z\|^2 - \alpha \right),  \] 
we obtain
\begin{align}\label{21_01_25_aux1}
&~ 0 \in \partial \vartheta (\widetilde z, \widetilde \alpha, \bar r) \notag \\ 
\Longleftrightarrow &~ 0 \in \partial q_M(\bar x, \widetilde z) - \frac{M}{2} \left( \|\widetilde z\| - \bar r \right)  \widetilde z, ~ {\widetilde \alpha} = \bar r^2.  
\end{align}
Observe that any $\widetilde{z} = d(\bar x, \bar r)$ with $\|d(\bar x, \bar r)\| = \bar r$ satisfies the inclusion relation in the right of \eqref{21_01_25_aux1}. It is noteworthy that the set of all $d(\bar x, \bar r)$ with $\|d(\bar x, \bar r)\| = \bar r$ is non-empty since the direction $d_M(\bar x )$, as in \eqref{10_01_25_aux2}, satisfies $\|d(\bar x, \bar r_M)\| = \bar r_M$. \\

\noindent
As the set of minimum points of $\vartheta (\cdot, \cdot, r) : \mathbb{R}^n \times \mathbb{R} \to \mathbb{R}$ is a singleton set, we get 
\begin{align}\label{21_01_25_aux2}  
~&~ \min_{ z \in \mathbb{R}^n \atop \alpha \in \mathbb{R}_+} \vartheta (z, \alpha, r) \notag \\ 
= ~&~ \min_{ \|d(\bar x, r)\| = r} \vartheta (d(\bar x, r), r^2, r) \notag \\ 
= ~&~ \min_{ \|d(\bar x, r)\| = r} \max_{\xi \in C} \left[  \la Jf(\bar x)d(\bar x, r) + \frac{1}{2} d(\bar x, r)^\top \nabla^2 f(\bar x) d(\bar x, r), \xi \ra + \frac{M}{6} r^3 \right] \notag \\ 
\ge  ~&~ \min_{ \|d(\bar x, r)\| = r} \sum_{i = 1}^{s(\bar x)} \lambda_i(\bar x) \left[  \la Jf(\bar x)d(\bar x, r) + \frac{1}{2} d(\bar x, r)^\top \nabla^2 f(\bar x) d(\bar x, r), \eta^i(\bar x) \ra + \frac{M}{6} r^3 \right] 
\end{align}
since for each $i = 1, 2, \ldots, s(\bar x)$, we have 
\begin{align*} 
~&~ \max_{\xi \in C} \left[  \la Jf(\bar x)d(\bar x, r) + \frac{1}{2} d(\bar x, r)^\top \nabla^2 f(\bar x) d(\bar x, r), \xi \ra + \frac{M}{6} r^3 \right] \\
\ge ~&~ \la Jf(\bar x)d(\bar x, r) + \frac{1}{2} d(\bar x, r)^\top \nabla^2 f(\bar x) d(\bar x, r), \eta^i(\bar x) \ra + \frac{M}{6} r^3. 
\end{align*}
We notice from the expression \eqref{d_bar_x_r} of $d(\bar x, r)$ that  
\begin{align} \label{21_01_25_aux6}
~&~ \left( \sum_{i = 1}^{s(\bar x)} \lambda_i(\bar x) \left\{ \la \nabla^2 f(\bar x), \eta^i(\bar x) \ra + \frac{M}{2} r ~ I_n \right\}\right) d(\bar x, r) + \sum_{i = 1}^{s(\bar x)} \lambda_i(\bar x) \la Jf(\bar x), \eta^i(\bar x) \ra = 0 \\
\text{i.e., } ~&~ \sum_{i = 1}^{s(\bar x)} \lambda_i(\bar x)  \la Jf(\bar x) + \nabla^2 f(\bar x) d(\bar x, r), \eta^i(\bar x) \ra  = - \frac{Mr}{2} d(\bar x, r) \notag \\
\text{i.e., } ~&~ \sum_{i = 1}^{s(\bar x)} \lambda_i(\bar x)  \la Jf(\bar x) d(\bar x, r) + d(\bar x, r)^\top \nabla^2 f(\bar x) d(\bar x, r), \eta^i(\bar x) \ra  = - \frac{M}{2} r \|d(\bar x, r)\|. \notag 
\end{align}
Hence, for $\|d(\bar x, r)\| = r$, we have 
\begin{align} \label{aux1_01_03_25}
~&~ \sum_{i = 1}^{s(\bar x)} \lambda_i(\bar x) \left[  \la Jf(\bar x)d(\bar x, r) + \frac{1}{2} d(\bar x, r)^\top \nabla^2 f(\bar x) d(\bar x, r), \eta^i(\bar x) \ra + \frac{M}{6} r^3 \right] \\
= ~&~ \frac{1}{2} \sum_{i = 1}^{s(\bar x)} \lambda_i(\bar x) \la Jf(\bar x) d(\bar x, r), \eta^i(\bar x) \ra - \frac{M}{12} r^3 \notag \\
= ~&~ v_l(r). \notag 
\end{align}
Thus, from \eqref{21_01_25_aux2}, we get 
\[ \min_{ z \in \mathbb{R}^n \atop \alpha \in \mathbb{R}_+} \vartheta (z, \alpha, r) \ge v_l(r), \]
and hence \eqref{21_01_25_aux5} gives 
\begin{equation}\label{22_01_25_aux1}
\min_{z \in \mathbb{R}^n} v_u(z) \ge \sup_{r \in \mathscr{D}_{\bar{x}}} v_l(r). 
\end{equation}
Next, to show that the inequality in \eqref{22_01_25_aux1} must be a perfect equality, we prove the following two assertions. \\

\begin{enumerate}[\text{Assertion} 1.]
\item 
Any critical point of $v_l:(0, \infty) \to \mathbb{R}$ lies in $W : = \{r \in (0, \infty): \|d(\bar x, r)\| = r\}$. \\ 
\item 
For any $r \in \mathscr{D}_{\bar x}$ and $d(\bar x, r)$ with $\|d(\bar x, r)\| = r$, we have $v_u(d(\bar x, r)) - v_l(r) =  0$. \\ 
\end{enumerate}

To prove Assertion 1, we denote 
$\zeta(\bar x) := \sum_{i = 1}^{s(\bar x)} \lambda_i(\bar x) \la Jf(\bar x), \eta^i(\bar x) \ra$. \\ \\ 
Note that 
$v_l(r) = \frac{1}{2} \la d(\bar x, r), \zeta(\bar x) \ra - \frac{M}{12} r^3$, and 
the equation \eqref{d_bar_x_r} gives 
\begin{align}
& \left( \sum_{i = 1}^{s(\bar x)} \lambda_i(\bar x) \left\{ \la \nabla^2 f(\bar x), \eta^i(\bar x) \ra + \frac{M}{2} r ~ I_n \right\}\right) d(\bar x, r) + \zeta(\bar x) = 0 \notag \\
\implies &  \sum_{i = 1}^{s(\bar x)} \lambda_i(\bar x) \la \nabla^2 f(\bar x), \eta^i(\bar x) \ra d(\bar x, r) + \frac{Mr}{2} d(\bar x, r) + \zeta(\bar x) = 0 \notag \\ 
\implies & \left( \sum_{i = 1}^{s(\bar x)} \lambda_i(\bar x) \la \nabla^2 f(\bar x), \eta^i(\bar x) \ra + \frac{Mr}{2} I_n \right) d'(\bar x, r) = - \frac{M}{2} d(\bar x, r) \label{aux3_06_03_25} \\ 
& \mbox{ (by differentiation with respect to $r$)}, \notag 
\end{align}
where $d'(\bar x, r)$ is the derivative of $d(\bar x, r)$ with respect to $r$. 
Therefore, 
\begin{align}
v_l'(r) = ~&~ \frac{1}{2} \la d'(\bar x, r), \zeta(\bar x) \ra - \frac{M}{4} r^2 \notag \\ 
= ~&~  - \frac{M}{4} \la \left( \sum_{i = 1}^{s(\bar x)} \lambda_i(\bar x) \la \nabla^2 f(\bar x), \eta^i(\bar x) \ra + \frac{Mr}{2} I_n \right)^{-1} d(\bar x, r), \zeta(\bar x)  \ra - \frac{M}{4} r^2 \notag \\ 
= ~&~ \frac{M}{4} \la \left( \sum_{i = 1}^{s(\bar x)} \lambda_i(\bar x) \la \nabla^2 f(\bar x), \eta^i(\bar x) \ra + \frac{Mr}{2} I_n \right)^{-2} \zeta(\bar x), \zeta(\bar x) \ra - \frac{M}{4} r^2 \notag \\ 
\overset{\eqref{aux1_01_03_25}}{=} &~ \frac{M}{4} \left(\| d(\bar x, r) \|^2 - r^2 \right). \label{aux1_06_03_25}
\end{align}
Hence, any $r$ that satisfies $v_l'(r) = 0$ must lie in $W$, which proves Assertion 1.  \\

To prove Assertion 2, let $r \in \mathscr{D}_{\bar x}$ and $\|d(\bar x, r)\| = r$. Then, from the defining expression of $v_u$, we get  
\[v_u(d (\bar x, r)) = \la Jf(\bar x) d(\bar x, r) + \frac{1}{2} d(\bar x, r)^\top \nabla^2 f(\bar x) d(\bar x, r), \sum_{i = 1}^{s(\bar x)} \lambda_i \eta^i(\bar x) \ra + \frac{M}{6} r^3,  \] 
which is identical with the expression of $v_l(r)$ (see \eqref{aux1_01_03_25}).  \\

From Assertion 1, Assertion 2, and \eqref{22_01_25_aux1}, we obtain the equation \eqref{10_1_25_aux3}. Hence, the result follows.

\end{proof}

\begin{remark}\label{lemma_d_M_is_1d_convex}
Note from \eqref{aux1_06_03_25} that for any $r \in \mathscr{D}_{\bar{x}}$, 
\begin{align*}
v''_l(r) = &~ \frac{M}{2} \left( d(\bar x, r)^\top d'(\bar x, r) - r \right) \\ 
\overset{\eqref{aux3_06_03_25}}{=} &~ - \frac{M^2}{4}   d(\bar x, r)^\top \left(\sum_{i = 1}^{s(\bar x)} \lambda_i(\bar x) \la \nabla^2 f(\bar x), \eta^i(\bar x) \ra + \frac{Mr}{2} I_n \right)^{-2} d(\bar x, r) - \frac{Mr}{2} < 0. 
\end{align*}
Thus, the maximization problem on the right is a minimization of the convex function $-v_l(r)$ on the convex set $\mathscr{D}_{\bar{x}}$. Hence, for any $x \in \mathrm{int}(S)$, there exists a cubic Newton direction $d_M(x)$, as defined in \eqref{cubic_direction}.    \\ 
\end{remark}

Note that the direction $d_M(\bar x)$ given by \eqref{10_01_25_aux2} is the direction $d(\bar x, r)$ given by \eqref{d_bar_x_r} with $r = r_M(\bar x)$. As $d(\bar x, r_M(\bar x))$ satisfies $\|d(\bar x, r_M(\bar x))\| = r_M(\bar x)$ and $d(\bar x, r_M(\bar x))$ is a solution point of $\min_{z \in \mathbb{R}^n} v_u(z)$, we get from the right of \eqref{10_1_25_aux3} that $r_M(\bar x) \in \mathscr{D}_{\bar{x}}$ and the expression of $d(\bar x, r_M(\bar x))$ is well-defined. Hence, we have the following result. \\

\begin{corollary}\label{invertable_for_all_xi}
For any $x \in \mathrm{int}(S)$ and $\xi \in C$, the matrix $\la \nabla^2 f(x),~ \xi \ra + \frac{M}{2} r_M(x) I_n$ is positive definite. \\ 
\end{corollary}

\begin{lemma}\label{upper_bound_of_J}
At a given $\bar x$ in $\mathrm{int}(S)$, if $d_M(\bar x)$ is such that $\bar x + d_M(\bar x)$ lies in $S$, then 
\[ \min_{\xi \in C} \|\la Jf(\bar x + d_M(\bar x)), \xi \ra \| \le \frac{1}{2} (L + M) r^2_M(\bar x). \]
\end{lemma}

\begin{proof}
From \eqref{replica_of_2_5}, we have 
\begin{equation}\label{aux_04_01_2}
\left\| \sum_{i = 1}^{s(\bar x)} \lambda_i \left\{\la Jf(\bar x), \eta^i \ra + \la d_M(\bar x)^\top \nabla^2 f(\bar x), \eta^i \ra \right\} \right\| = \frac{M}{2} r^2_M(\bar x).      
\end{equation}
From \eqref{square_equation}, with $y = \bar x + d_M(\bar x)$, $x = \bar x$, and $\xi$ taken from $\mathscr{A}(d_M(\bar x))$, we have 
\begin{equation}\label{aux_04_01_1}
\left\| \la Jf(\bar x + d_M(\bar x)) - Jf(\bar x) - d_M(\bar x)^\top \nabla^2 f(\bar x),~ \eta^i \ra \right\| \le \frac{L}{2} r_M^2(\bar x), i = 1, 2, \ldots, s(\bar x). 
\end{equation}
Therefore, 
\begin{align}
~&~ \left\| \sum_{i = 1}^{s(\bar x)} \lambda_i \la Jf(\bar x + d_M(\bar x)) - Jf(\bar x) - d_M(\bar x)^\top \nabla^2 f(\bar x),~ \eta^i \ra \right\|  \notag \\ 
\le ~&~ \sum_{i = 1}^{s(\bar x)} \lambda_i \left\| \la Jf(\bar x + d_M(\bar x)) - Jf(\bar x) - d_M(\bar x)^\top \nabla^2 f(\bar x),~ \eta^i \ra \right\| \notag \\ 
\overset{\eqref{aux_04_01_1}}{\le} ~&~ \frac{L}{2} r_M^2(\bar x). \label{aux_04_01_3}
\end{align}
Hence, from \eqref{aux_04_01_2} and \eqref{aux_04_01_3}, we get 
\begin{align} 
~&~ \left\| \sum_{i = 1}^{s(\bar x)} \lambda_i \la Jf(\bar x + d_M(\bar x)), \eta^i \ra \right\| \notag \\
\le ~&~ 
\left\| \sum_{i = 1}^{s(\bar x)} \lambda_i 
\left\{ 
\left( \la Jf(\bar x + d_M(\bar x)) - Jf(\bar x) - d_M(\bar x)^\top \nabla^2 f(\bar x),~ \eta^i \ra \right) + 
\left( \la Jf(\bar x) + d_M(\bar x)^\top \nabla^2 f(\bar x),~ \eta^i \ra \right) 
\right\}
\right\| \notag \\ 
\le ~&~  \frac{L + M}{2} r_M^2(\bar x). \notag  
\end{align}
Hence, 
\[\min_{\xi \in C} \|\la Jf(\bar x+ d_M(\bar x)), \xi \ra \| \le 
\min_{\xi \in \mathscr{A}(d_M(\bar x))} \|\la Jf(\bar x+ d_M(\bar x)), \xi \ra \| \le \frac{L + M}{2} r_M^2(\bar x). 
\]
\end{proof}

\begin{lemma}\label{aux3_18_Feb_2025}
Define a function $h_M: \mathbb{R}^n \to \mathbb{R}$ by 
\[h_M(x) := \min_{y \in S} \left( \max_{\xi \in C} 
\left\{\la f(x) + Jf(x) (y - x) + \frac{1}{2} (y - x)^\top \nabla^2 f(x) (y - x), \xi \ra \right\} + \frac{M}{6} \|y - x\|^3 \right). \]
Then, at any given $\bar x \in S$, the following inequalities hold:  
\begin{equation}\label{h_M_is_less}
h_M(\bar x) \le \min_{y \in S} \left( \max_{\xi \in C} \la f(y), \xi \ra + \frac{L + M}{6} \|y - \bar x\|^3 \right)     
\end{equation}
and 
\begin{equation}\label{f_minus_h_is_M_by_12}
\max_{\xi \in C} \la f(\bar x), \xi \ra - h_M(\bar x) \ge \frac{M}{12} r^3_M(\bar x). 
\end{equation}
Also, for $M \ge L$, we have 
\begin{equation} \label{aux5_06_03_25} 
\max_{\xi \in C} \la f(\bar x + d_M(\bar x)), \xi \ra \le h_M(\bar x). \end{equation}
\end{lemma}

\begin{proof}
From \eqref{cube_equation}, we obtain for all $\xi \in C$ and $y \in S$ that 
\begin{align*}
& \la f(\bar x) + Jf(\bar x) (y - \bar x) + \frac{1}{2} (y - \bar x)^\top \nabla^2 f(\bar x) (y - \bar x), \xi \ra + \frac{M}{6} \|y - \bar x\|^3  \le \la f(y), \xi \ra + \frac{L + M}{6} \|y - \bar x \|^3, \\  
\text{i.e., } & h_M(\bar x) \le \min_{y \in S} \left( \max_{\xi \in C} \la f(y), \xi \ra + \frac{L + M}{6} \|y - \bar x\|^3 \right), 
\end{align*}
which is \eqref{h_M_is_less}. \\

From the definition of the function $h_M$ and the direction $d_M(x)$, we have 
\begin{align} 
~&~ h_M(\bar x) - \max_{\xi \in C} \la f(\bar x), \xi \ra \notag \\ 
= ~&~ \min_{y \in S} \left( \max_{\xi \in C} 
\left\{\la Jf(\bar x) (y - \bar x) + \frac{1}{2} (y - \bar x)^\top \nabla^2 f(\bar x) (y - \bar x), \xi \ra \right\} + \frac{M}{6} \|y - \bar x\|^3 \right) \notag \\ 
= ~&~ \la Jf(\bar x) d_M(\bar x) + \frac{1}{2} d_M(\bar x)^\top \nabla^2 f(\bar x) d_M(\bar x), \eta^i(\bar x) \ra  + \frac{M}{6} r^3_M(\bar x) \text{ for all } \eta^i(\bar x) \in \mathscr{A}(d_M(\bar x)) \notag \\ 
= ~&~ \sum_{i = 1}^{s(\bar x)} \lambda_i(\bar x) \la Jf(\bar x) d_M(\bar x) + \frac{1}{2} d_M(\bar x)^\top \nabla^2 f(\bar x) d_M(\bar x), \eta^i(\bar x) \ra  + \frac{M}{6} r^3_M(\bar x), \label{aux1_24_01}
\end{align} 
where $\lambda_i(\bar x) \ge 0$ for all $i = 1, 2, \ldots, s(\bar x)$, and 
$\sum_{i = 1}^{s(\bar x)} \lambda_i(\bar x) = 1. $ \\

\noindent
By multiplying $d_M(\bar x)$ on both sides of \eqref{10_01_25_aux1}, we obtain

\begin{align} 
~&~  \sum_{i = 1}^{s(\bar x)} \lambda_i(\bar x) \left\{ \la d_M(\bar x)^\top \nabla^2 f(\bar x) d_M(\bar x), \eta^i(\bar x) \ra + \frac{M}{2} r^3_M(\bar x)\right\}   = - \sum_{i = 1}^{s(\bar x)} \lambda_i(\bar x) \la Jf(\bar x) d_M(\bar x), \eta^i(\bar x) \ra \notag \\  
\text{i.e., } ~&~ ~~~~\sum_{i = 1}^{s(\bar x)} \lambda_i(\bar x) \la Jf(x) d_M(\bar x) + \frac{1}{2} d_M(\bar x)^\top \nabla^2 f(\bar x) d_M(\bar x), \eta^i(\bar x)\ra \notag  \\
~&~ = - \frac{1}{2} \sum_{i = 1}^{s(\bar x)} \lambda_i(\bar x) \la d_M(\bar x)^\top \nabla^2 f(\bar x) d_M(\bar x), \eta^i(\bar x) \ra - \frac{M}{2} r^3_M(\bar x) \notag \\
~&~ \le - \frac{M}{4} r^3_M(\bar x) \text{ by Corollary \ref{invertable_for_all_xi}.}  \label{aux2_01_03_25}
\end{align} 
Hence, the equation \eqref{aux1_24_01} gives 
\[h_M(\bar x) - \max_{\xi \in C} \la f(\bar x), \xi \ra \le - \frac{M}{12} r^3_M(\bar x), \] 
which proves \eqref{f_minus_h_is_M_by_12}. \\

Now, for $M \ge L$, note from \eqref{cube_equation} that for all $\xi \in C$ and $y \in S$,  
\begin{align*}
~&~ \la f(y) - f(\bar x) - Jf(\bar x) (y - \bar x) - \frac{1}{2} (y - \bar x)^\top \nabla^2 f(\bar x) (y - \bar x), ~\xi \ra \le \frac{L}{6}  \| y - \bar x\|^3 \le \frac{M}{6} \| y - \bar x\|^3 \\
\text{i.e., } ~&~  \la f(y), \xi \ra \le 
\la f(\bar x) + Jf(\bar x) (y - \bar x) + \frac{1}{2} (y - \bar x)^\top \nabla^2 f(\bar x) (y - \bar x), ~\xi \ra + \frac{M}{6} \|y - \bar x\|^3. 
\end{align*}
Hence, for all $y \in S$, we have 
\[ \max_{\xi \in C} \la f(y), \xi \ra \le 
\max_{\xi \in C}  \la f(\bar x) + Jf(\bar x) (y - \bar x) + \frac{1}{2} (y - \bar x)^\top \nabla^2 f(\bar x) (y - \bar x), ~\xi \ra + \frac{M}{6} \|y - \bar x\|^3, \]
which gives for $y = \bar x + d_M(\bar x)$ that 
\begin{equation} \label{26_01_25_aux1}
\max_{\xi \in C} \la f(\bar x + d_M(\bar x)), \xi \ra \le 
\max_{\xi \in C}  \la f(\bar x) + Jf(\bar x) d_M(\bar x) + \frac{1}{2} d_M(\bar x)^\top \nabla^2 f(\bar x) d_M(\bar x), ~\xi \ra + \frac{M}{6} r_M^3(\bar x).  
\end{equation}
By the definitions of $d_M(\bar x)$ (see \eqref{cubic_direction}) and $h_M(\bar x)$, notice that 
\begin{align*} 
h_M(\bar x) = ~&~ \min_{y \in S} \left( \max_{\xi \in C} 
\left\{\la f(\bar x) + Jf(\bar x) (y - \bar x) + \frac{1}{2}  (y - \bar x)^\top \nabla^2 f(\bar x) (y - \bar x), \xi \ra \right\} + \frac{M}{6} \|y - \bar x\|^3 \right) \\ 
= ~&~ \min_{d \in \mathbb{R}^n} \left( \max_{\xi \in C} 
\left\{\la f(\bar x) + Jf(\bar x) d + \frac{1}{2} d^\top \nabla^2 f(\bar x) d, \xi \ra \right\} + \frac{M}{6} \|d\|^3 \right) \\ 
= ~&~ \max_{\xi \in C} 
\la f(\bar x) + Jf(\bar x) d_M(\bar x) + \frac{1}{2} d_M(\bar x)^\top \nabla^2 f(\bar x) d_M(\bar x), \xi \ra + \frac{M}{6} r_M^3(\bar x),  
\end{align*}
which combinedly with \eqref{26_01_25_aux1} yield that for $M \ge L$, 
\[\max_{\xi \in C} \la f(\bar x + d_M(\bar x)), \xi \ra \le h_M(\bar x). \] 

\end{proof}

\begin{remark} \label{remark_for_M_exists}
Note that \eqref{f_minus_h_is_M_by_12} holds irrespective of $M \ge L$. Further, for $M \ge L$, there holds the inequality \eqref{aux5_06_03_25}. This implies that the set 
\[M(x) := \left\{\bar M: \max_{\xi \in C} \la f(x + d_{\bar M}(x)), \xi \ra \le h_{\bar M}(x)\right\}\]
is non-empty if the assumption \eqref{assumption_ii} holds, i.e., if there exists an $L$-value satisfying \eqref{assumption_ii}.  \\ 
\end{remark}

\begin{lemma}\label{mu_M_le_r_M}
For all $x \in \mathrm{int}(S)$, there holds $\mu_M(x + d_M(x)) \le r_M(x)$, where 
$\mu_M : S \to [0, \infty)$ is the function   
\begin{equation}\label{mu_M_definition}
\mu_M(x) := \max\left\{ \sqrt{\frac{2}{M + L} \min_{\xi \in C} \|\la Jf(x), \xi\ra \|}, - \frac{2}{M + 2pL} \max_{\xi \in C} \lambda_1 \left(\la \nabla^2 f(x), \xi \ra \right)  \right\}.     
\end{equation}
\end{lemma}

\begin{proof}
From the assumption \eqref{assumption_ii} on Lipschitz condition, we have for any $\xi \in C$ that  
\begin{align*}
~&~ \| \la \nabla^2 f(x) - \nabla^2 f(x + d_M(x)), \xi \ra \| \\ 
= ~&~ \left \| \sum_{i = 1}^p \xi_i \left( \nabla^2 f_i(x) - \nabla^2 f_i(x + d_M(x))\right) \right\| \\ 
\le ~&~\sum_{i = 1}^p |\xi_i| \left\| \nabla^2 f_i(x) - \nabla^2 f_i(x + d_M(x)) \right\| \\  
\le ~&~ L r_M(x) (|\xi_1| + |\xi_2| + \cdots + |\xi_p|)  \\
\le ~&~ p L r_M(x). 
\end{align*}
Thus, for all $\xi \in C$, we have  
\begin{align*} 
\la \nabla^2 f(x + d_M(x)), \xi \ra ~\ge~ & \la \nabla^2 f(x), \xi \ra - p L r_M(x) I_n, 
\end{align*}
which gives by Corollary \ref{invertable_for_all_xi} that  
\begin{align} 
& \la \nabla^2 f(x + d_M(x)), \xi \ra 
\ge \la \nabla^2 f(x), \xi \ra - pL r_M(x) I_n \ge  
- \left(pL + \frac{M}{2} \right) r_M(x) I_n ~ \text{ for all } \xi \in C \label{aux1_23_02} \\ 
\text{i.e., } &  \max_{\xi \in C} \lambda_1 \left(\la \nabla^2 f(x + d_M(x)), \xi \ra \right) \ge - \frac{2pL + M}{2} r_M(x) \notag \\ 
\text{i.e., } & - \frac{2}{2pL + M} \max_{\xi \in C} \lambda_1 \left(\la \nabla^2 f(x + d_M(x)), \xi \ra \right)  \le r_M(x).  \label{aux8_18_Feb} 
\end{align} 
From Lemma \ref{upper_bound_of_J}, we have 
\[ \sqrt{\frac{2}{L + M} \min_{\xi \in C} \|\la Jf(x + d_M(x)), \xi \ra \| } \le r_M(x), \]
which combinedly with \eqref{aux8_18_Feb} give 
\[\mu_M(x + d_M(x)) \le r_M(x). \]  
\end{proof}

Results in Lemmas \ref{upper_bound_of_J} and \ref{mu_M_le_r_M} indicate that define a function $\mu_M : S \to [0, \infty)$ by \eqref{mu_M_definition} and use $\mu_M(x) = 0$ as a stopping condition instead of $\beta_M(x) \ge 0$. In fact, there are two major benefits of using $\mu_{M_k}(x^k) = 0$ as stopping condition than using $\beta_{M_k}(x^k) \ge 0$: \\  
\begin{enumerate}
\item $\beta_{M_k}(x^k) = 0$ ensures that $x^k$ is merely a stationary point and may not be a weakly efficient point. On the other hand, if $\{M_k\}$ is bounded above, then $\mu_{M_k}(x^k) = 0$ implies that $x^k$ is a weakly efficient point of \eqref{vop} since $\mu_{M_k}(x^k) = 0$ gives 
\begin{align*}
& \frac{2}{M_k + L} \min_{\xi \in C} \|\la Jf(x^k), \xi \ra \| \le \mu_{M_k}(x^k) = 0 \\ 
& \text{ and } - \frac{2}{M_k + 2pL} \max_{\xi \in C} \lambda_1 (\la \nabla^2 f(x^k), \xi \ra ) \le \mu_{M_k}(x^k) = 0\\ 
\text{i.e., } & \|\la Jf(x^k), \xi \ra \| = 0 \text{ for all } \xi \in C \text{ and }  \lambda_1 (\la \nabla^2 f(x^k), \xi' \ra ) \ge 0 \text{ for some } \xi' \in C,  
\end{align*}
which further implies that $\la Jf(x^k), \xi' \ra = 0$ and $\la \nabla^2 f(x^k), \xi' \ra )$ is positive definite for some $\xi' \in C$. Thus, by Lemma \ref{for_one_xi_is_weak_min}, $x^k$ is a weakly efficient point of \eqref{vop}, and not just a stationary point.  \\

\item 
$\beta_{M_k}(x^k)$ has no explicit formula, and to find its value, one needs to solve a minimization problem, namely, \eqref{beta_function}. However, $\mu_{M_k}(x^k)$ can be directly computed from the explicit formula \eqref{mu_M_definition}.  \\   
\end{enumerate}

Next, with the help of the cubic regularized Newton direction $d_M(x)$ and the stopping condition $\mu_M(x) = 0$, we provide a step-wise algorithm to find weakly efficient points of \eqref{vop}.

\begin{algorithm}[H]
\caption{Cubic regularization of the Newton method for solving \eqref{vop}}{\label{algo}}
\begin{enumerate}[\text{Step} 1]
\item  (\emph{Problem data and initialization}). \\
Provide the functions $f_1, f_2, \ldots, f_p$ and the set $S$ for the problem \eqref{vop}. \\ 
Provide a generator $C$ of the ordering cone $K$. \\ 
Provide an estimate $L_0$ of the Lipschitz constant $L$ (as in assumption \eqref{assumption_ii}).   \\ 
Choose an initial point $x^0 \in S$, and $M_0$ from $[L_0, 2L]$. \\  
Provide a termination scalar $\varepsilon > 0$.  \\
Set the iteration count $k \leftarrow 0$. \\

\item \label{aux2_18_Feb_2025} (\emph{$k$-th iteration}). \\ 
Find $M_k$ such that \[\max_{\xi \in C} \la f\left(x^k + d_{M_k}(x^k)\right), \xi \ra \le h_{M_k}(x^k), \]
where 
\begin{align*}
~&~ h_{M_k}(x^k) := \min_{d \in \mathbb{R}^n} \left( \max_{\xi \in C} 
\la f(x^k) + Jf(x^k) d + \frac{1}{2} d^\top \nabla^2 f(x^k) d, \xi \ra  + \frac{M_k}{6} \|d\|^3 \right) \\ 
\mbox{ and } ~&~  d_{M_k}(x^k) := \underset{d \in \mathbb{R}^n}{\argmin}  ~\underset{\xi \in C}{\max} \left\{ \la Jf(x^k) d, \xi \ra + \frac{1}{2} \la d^\top \nabla^2 f(x^k) d, \xi \ra + \frac{M_k}{6} \|d\|^3 \right\}.   
\end{align*}

\item \label{aux5_18_Feb_2025} (\emph{Termination condition}). \\ 
Compute $\mu_{M_k} (x^k)$ by \eqref{mu_M_definition}. \\ 
If $\mu_{M_k}(x^k) < \varepsilon$, then go to Step 4. \\ 
Otherwise, update  
\[x^{k + 1} \leftarrow x^k + d_{M_k}(x^k) \mbox{ and } k \leftarrow k + 1; \]
choose $M_k$ from $[L_0, 2L]$, and go to Step 2.\\

\item (\emph{Output}).\\
Provide $x^{k}$ as an $\varepsilon$-precise weakly efficient point of the vector optimization problem \eqref{vop}.  
\end{enumerate}
\end{algorithm}

The \emph{well-defindness} of Algorithm \ref{algo} depends only on Step 2 on the following two points: \\ 
\begin{enumerate}
\item Does there exist an $M_k$ that satisfies $\max_{\xi \in C} \la f\left(x^k + d_{M_k}(x^k)\right), \xi \ra \le h_{M_k}(x^k)$?  
\item For the chosen $M_k$ that satisfies $\max_{\xi \in C} \la f\left(x^k + d_{M_k}(x^k)\right), \xi \ra \le h_{M_k}(x^k)$, is the direction $d_{M_k}(x^k)$ well-defined? \\ 
\end{enumerate}

For the first point, we note from Remark \ref{remark_for_M_exists} that under the assumption \eqref{assumption_ii}, the set $M_k(x^k)$ (defined in Remark \ref{remark_for_M_exists}) is non-empty. Hence, there exists an $M_k>0$ that satisfies the condition in Step 2. For the second point, we refer to Remark \ref{lemma_d_M_is_1d_convex}, which ensures that for the chosen $M_k$ at the beginning of Step 2, there exists a cubic Newton step $d_{M_k}(x^k)$. Hence, all the steps of Algorithm \ref{algo} are well-defined. \\

Although the steps of Algorithm \ref{algo} are well-defined, a few remarks are in order for the decency property of $\{f(x^k)\}$ and for the computation of the parameter $M_k$.  \\

\begin{remark}\label{no_descent_property_remark}
Note that for the chosen $M_k$ in Step 2 of Algorithm \ref{algo}, although we have 
\begin{align*} 
& \max_{\xi \in C} \la f\left(x^k + d_{M_k}(x^k)\right), \xi \ra \le h_{M_k}(x^k) \\ 
\overset{\eqref{f_minus_h_is_M_by_12}}{\implies} &  \max_{\xi \in C} \la f\left(x^{k + 1}\right), \xi \ra \le  \max_{\xi \in C} \la f\left(x^k)\right), \xi \ra,  
\end{align*} 
but this does not give any guarantee that $f(x^{k + 1}) \prec f(x^k)$. From the trajectory of the movement of $f(x^k)$ in Figure \ref{for-fifty-random-initial-points} (Far1) and Figure \ref{ten-points-performance} (Hil1, KW2, MOP3, FON, SLCDT2), a non-monotonic nature of the sequence $\{f(x^k)\}$ is clearly visible. However, in case we have $L_0 \ge \frac{2L}{3}$, then Theorem \ref{x_k_conv_to_weak_min}(i) below ensures that any sequence $\{f(x^k)\}$ generated by Algorithm \ref{algo} is monotonically decreasing with respect to the ordering cone $K$. \\ 
\end{remark}

\begin{remark}\label{process_to_evaluate_M_k}
To find an $M_k$ at $x^k$ for Step 2 in Algorithm \ref{algo}, we can employ the following technique: 
\begin{enumerate}[Step 1.]
    \item Make an initial guess $M_k$.  
    \item While $\max_{\xi \in C} \la f\left(x^k + d_{M_k}(x^k)\right), \xi \ra > h_{M_k}(x^k)$ update $M_k \leftarrow 2M_k$. 
    \item Output $M_k$. 
\end{enumerate}
Note that this technique is well-defined because $d_{M_k}(x^k)$ exists for any $M_k$ (see Remark \ref{lemma_d_M_is_1d_convex}) and there exists a finite $\bar M$ that satisfies $\max_{\xi \in C} \la f(x^k + d_{\bar M}(x^k)), \xi \ra \le h_{\bar M}(x^k)$ (see Remark \ref{remark_for_M_exists}).  \\ 
\end{remark}

\begin{lemma}
Let $\{x^k\}$ be a sequence generated by Algorithm \ref{algo} and $ \lim_{k \to \infty} r_{M_k}(x^k) = 0.$ Then,  
$ \lim_{k \to \infty} \mu_{M_k}(x^k) = 0$ and  $ \lim_{k \to \infty} \beta_{M_k}(x^k) = 0$. 
\end{lemma}

\begin{proof}
As $\lim_{k \to \infty} r_{M_k}(x^k) = 0$, we get from Lemma \ref{mu_M_le_r_M} that $\lim_{k \to \infty} \mu_{M_k}(x^{k + 1}) = 0.$ Then, from the expression \eqref{mu_M_definition} of $\mu_{M_k}(x^k)$, we get  
\begin{align}
& \min_{\xi \in C} \|\la Jf(x^{k+1}), \xi \ra \| \to 0 \text{ and } \max_{\xi \in C} \lambda_1 \la \nabla^2 f(x^{k + 1}), \xi \ra \ge 0 \text{ for large } k \notag \\ 
\text{i.e., } & \min_{\xi \in C} \|\la Jf(x^{k}), \xi \ra \| \to 0 \text{ and } \max_{\xi \in C} \lambda_1 \la \nabla^2 f(x^{k}), \xi \ra \ge 0 \text{ for large } k \label{aux01_07_03_25} \\ 
\text{i.e., } & \frac{2}{M_k + L} \min_{\xi \in C} \|\la Jf(x^k), \xi \ra \| \to 0 \text{ and } - \frac{2}{M_k + 2pL} \max_{\xi \in C} \lambda_1 (\la \nabla^2 f(x^k), \xi \ra ) \le 0 \text{ for large } k \notag \\ 
& \text{ because } M_k \in [L_0, 2L] \text{ for all } k \notag \\ 
\text{i.e., } & \mu_{M_k}(x^k) \to 0 \text{ as } k \to \infty.  \notag 
\end{align}
From the definition of $\beta_{M_k}$ in \eqref{beta_function}, we obtain 
\[\beta_{M_k}(x^k) = \la Jf(x^k), \eta^k(x^k) \ra + \frac{1}{2} \la {d_{M_k}(x^k)}^\top \nabla^2 f(x^k) d_{M_k}(x^k), \eta^k(x^k) \ra + \frac{M_k}{6} r_{M_k}(x^k), \]
where $\eta^k(x^k)$ is $\eta(\bar x)$ as defined in \eqref{10_01_25_aux2} for $\bar x = x^k$. Thus, from Corollary \ref{invertable_for_all_xi}, we get 
\begin{align}\label{aux_02_07_03_25}
\beta_{M_k}(x^k) \ge - \| \la Jf(x^k), \eta^k(x^k) \ra \| r_{M_k}(x^k) - \frac{M_k}{4} r^3_{M_k}(x^k) + \frac{M_k}{6} r_{M_k}^3(x^k)  \to 0 \text{ as } k \to \infty 
\end{align}
since $\{M_k\}$ and $\{\| \la Jf(x^k), \eta^k(x^k) \ra \|\}$ are bounded; boundedness of $\{\| \la Jf(x^k), \eta^k(x^k) \ra \|\}$ is followed from \eqref{aux01_07_03_25}. Note from Theorem \ref{stationarity_and_beta_value}(i) that $\beta_{M_k}(x^k) \le 0$ for all $k$. Thus, we get from \eqref{aux_02_07_03_25} that 
$\lim_{k \to \infty} \beta_{M_k}(x^k) = 0.$ Hence, the result follows. 
\end{proof}

\medskip 
\begin{theorem}\label{aux2_28_01_25}
Consider the notation $\mathscr{A}(d_{M_j}(x^j))$ given in \eqref{aux4_01_03_25} with $\bar x = x^j$, and denote $\eta^j$ as a generic element of $\mathscr{A}(d_{M_j}(x^j))$.  
Let $\{x^k\}$ be a sequence of non-weakly efficient points generated by Algorithm \ref{algo}. Suppose that there exists $\sigma \in \mathbb{R}^p$ such that $f(x) \succeq \sigma \mbox{ for all } x \in \mathbb{R}^n. $
Then, 
\[\sum_{j = 0}^\infty r^3_{M_j}(x^j) \le \frac{12}{L_0} \left( \la f(x^0), \eta^0 \ra + \| \sigma \|\right). \]
Furthermore, $\lim_{j \to \infty} \mu_L(x^j) = 0$ and 
\begin{equation}\label{aux7_23_02}
\min_{1 \le j \le k} \mu_L(x^j) \le \frac{8}{3} \left(\frac{3 \la f(x^0), \eta^0 \ra + \| \sigma \|}{2 k L_0}\right)^{\frac{1}{3}}.    
\end{equation} 
\end{theorem}

\begin{proof}
Notice by Step \ref{aux2_18_Feb_2025} of Algorithm \ref{algo} that 
\begin{align}\label{aux4_18_Feb_2025} 
\la f(x^j), \eta^j \ra - \la f(x^j + d_{M_j}(x^j)), \eta^{j + 1} \ra  
\ge ~&~ \la f(x^j), \eta^j \ra - \max_{\xi \in C} \la f(x^j + d_{M_j}(x^j)), \xi \ra \notag \\ 
\ge ~&~ \max_{\xi \in C} \la f(x^j), \xi \ra - h_{M_j}(x^j) \ge \frac{M_j}{12} r_{M_j}^3 (x^j), 
\end{align}
where the last inequality is followed from Lemma \ref{aux3_18_Feb_2025}. As $M_j \in [L_0, 2L]$, the equation \eqref{aux4_18_Feb_2025} yields 
\begin{align}
& \sum_{j = 0}^{k - 1} \left( \la f(x^j), \eta^j \ra - \la f(x^j + d_{M_j}(x^j)), \eta^{j + 1} \ra \right) \ge \frac{L_0}{12} \sum_{j = 0}^{k - 1} r_{M_j}^3 (x^j) \notag \\
\text{i.e., } & 
\sum_{j = 0}^{k - 1} r^3_{M_j} (x^j) 
\le \frac{12}{L_0} \left( \la f(x^0), \eta^0 \ra - \la f(x^k), \eta^{k} \ra \right)     
\le \frac{12}{L_0} \left( \la f(x^0), \eta^0 \ra - \la \sigma, \eta^k \ra \right) \label{aux5_23_02} \\ 
\text{i.e., } & 
k \left( \min_{0 \le j \le k-1 } r_{M_j}^3(x^j) \right) \le \sum_{j = 0}^{k - 1} r^3_{M_j} (x^j) \le \frac{12}{L_0} \left( \la f(x^0), \eta^0 \ra + \|\sigma \| \right). \notag 
\end{align}
Therefore, 
\begin{equation}\label{aux9_18_Feb} 
\min_{0 \le j \le k-1 } r_{M_j}^3(x^j) \le \frac{12}{k L_0} \left( \la f(x^0), \eta^0 \ra + \|\sigma \| \right). 
\end{equation}

\medskip
\noindent
As for any non-critical point $x^{j+1}$, we have $\max_{\xi \in C} \lambda_1 \left(\la \nabla^2 f(x^{j+1}), \xi \ra \right) \le 0$ and $L_0 \le M_{j+1} \le 2L$, from Lemma \ref{mu_M_le_r_M}, we obtain  
\begin{align}
r_{M_j}(x^j) ~\ge~ & \mu_{M_{j+1}}(x^{j+1}) \notag \\ 
~\overset{\eqref{mu_M_definition}}{=} & \max\left\{ \sqrt{\frac{2}{M_{j+1} + L} \min_{\xi \in C} \|\la Jf(x^{j+1}), \xi\ra \|}, - \frac{2}{M_{j+1} + 2pL} \max_{\xi \in C} \lambda_1 \left(\la \nabla^2 f(x^{j+1}), \xi \ra \right)  \right\} \notag \\ 
~\ge~ & \max \left\{ \sqrt{\frac{2}{3L} \min_{\xi \in C} \|\la Jf(x^{j+1}), \xi\ra \|}, - \frac{1}{(p+1)L} \max_{\xi \in C} \lambda_1 \left(\la \nabla^2 f(x^{j+1}), \xi \ra \right)  \right\} \notag \\ 
~\ge~ & \frac{3}{4} ~ \mu_{L} (x^{j + 1})\label{aux6_23_02} 
\end{align}
since $\frac{2p+1}{2(p+1)} \ge \frac{3}{4}$ for all $p \ge 1$. 
Hence, with the help of \eqref{aux5_23_02}, we obtain $\lim_{j \to \infty} \mu_L(x^j) = 0$ because 
\begin{align*}
0 \le \lim_{j \to \infty} \mu_L(x^j) \le \lim_{j \to \infty} \frac{4}{3} r_{M_j}(x^j) = 0.    
\end{align*}

\noindent
Further, from \eqref{aux9_18_Feb} and \eqref{aux6_23_02}, we get 
\begin{align*} 
& \min_{1 \le j \le k} \mu^3_L(x^j) \le \frac{64}{27} \min_{0 \le j \le k-1} r^3_{M_j}(x^j) \le \frac{256}{9} \frac{1}{k L_0} \left( \la f(x^0), \eta^0 \ra + \|\sigma\| \right),   \\ 
\text{i.e., } &  \min_{1 \le j \le k} \mu_L(x^j) \le \frac{8}{3} \left(\frac{3 \la f(x^0), \eta^0 \ra + \| \sigma \|}{2 k L_0}\right)^{\frac{1}{3}}.  
\end{align*} 

\end{proof}

\begin{remark}
From \eqref{aux7_23_02}, we see that the smallest value among  $\mu_L(x^1), \mu_L(x^1), \ldots, \mu_L(x^k)$ does not exceed  
$ \beta k^{-\frac{1}{3}} $, where $\beta := \frac{8}{3} \left(\frac{3 \la f(x^0), \eta^0 \ra + \| \sigma \|}{2 k L_0}\right)^{\frac{1}{3}}$. This indicates that at the end of $k$-th iteration of Algorithm \ref{algo}, at least one point of $x^0$, $x^1$, \ldots, $x^k$ has $\mu_L$-value smaller than $O\left(k^{-\frac{1}{3}}\right)$; precisely, 
\[\mu_L(x^j) \le O\left(k^{-\frac{1}{3}}\right) \text{ for some } j \in \{1, 2, \ldots, k\}. \] 
This yields, with the help of Lemma \ref{mu_M_le_r_M}, that for some $j \in \{1, 2, \ldots, k\}$, 
\[\min_{\xi \in C} \| \la Jf(x^j), \xi \ra \| \le O\left(k^{-\frac{2}{3}}\right).\]
Thus, the proposed Algorithm \ref{algo} has global convergence property (since $\lim_{k \to \infty} \mu_L(x^k)  = 0$) and the rate of convergence to reach (a stationary point, i.e.,) $\la Jf(x), \bar \xi \ra = 0$ for some $\bar \xi \in C$ is $O\left(k^{-{2}/{3}}\right)$. \\ 
\end{remark}

\begin{theorem}\label{x_k_conv_to_weak_min}
Let $L_0 \ge \frac{2L}{3}$ and $\{x^k\}$ be a sequence generated by Algorithm \ref{algo}. If $\mathcal{L}(f(x^\ell))$ is bounded for some $\ell \in \mathbb{N}$, then the following results hold. 
\begin{enumerate}[(i)]
\item 
The sequence $\{f(x^k)\}$ is monotonic decreasing and convergent.  

\item 
The set $X$ of all subsequential limits of $\{x^k\}$ is non-empty, and $f(x) = \Gamma$ for all $x \in X$, where $\Gamma := \lim_{k \to \infty} f(x^k)$. 

\item \label{third_th_3_5}
If $\bar x$ is a subsequential limit of $\{x^k\}$, then there exists $\bar \xi \in C$ such that $\la Jf(\bar x), \bar \xi\ra = 0$ and $\lambda_1\left(\la\nabla^2 f(\bar x), \bar \xi \ra\right) \ge 0$. 
\end{enumerate}
\end{theorem}

\begin{proof}
\begin{enumerate}[(i)]
\item 
We get from \eqref{cube_equation} with $x = x^k$ and $y = x^k + d_{M_k}(x^k)$ that  for all $\xi \in C$, 
\begin{align*} 
& \la f(x^{k + 1}), \xi \ra - \la f(x^k), \xi \ra  \\ 
\le & \la Jf(x^k) d_{M_k}(x^k) + \frac{1}{2} d_{M_k}(x^k)^\top \nabla^2 f(x^k) d_{M_k}(x^k), \xi \ra + \frac{L}{6} r^3_{M_k}(x^k)\\ 
= &  \max_{\xi \in C} \la Jf(x^k) d_{M_k}(x^k) + \frac{1}{2} d_{M_k}(x^k)^\top \nabla^2 f(x^k) d_{M_k}(x^k), \xi \ra + \frac{L}{6} r^3_{M_k}(x^k) \\ 
\le & \la Jf(x^k) d_M(x^k) + \frac{1}{2} \la \nabla^2 f(x^k) d_M(x^k), d_M(x^k) \ra, \eta^i(x^k) \ra  + \frac{L}{6} r_{M_k}^3(x^k) \\ 
& \text{ for all } \eta^i(x^k) \in \mathscr{A}(d_{M_k}(x^k)) \notag \\
\le & - \left(\frac{M}{4} - \frac{L}{6} \right) r^3_{M_k}, 
\end{align*}
which is followed from \eqref{aux2_01_03_25} with $\bar x = x^k$ and $M = M_k$. Therefore, from $M_k \ge L_0 \ge \frac{2L}{3}$, we get 
\begin{align*}
~&~ \min_{ \xi \in C} \la f(x^{k+1}), \xi \ra \le \min_{ \xi \in C} \la f(x^{k}), \xi \ra \text{ for all } \xi \in C, \\ 
\text{ i.e., } ~&~ f(x^{k + 1}) \preceq f(x^k), 
\end{align*}
which implies that the sequence $\{f(x^k)\}$ is monotonic decreasing. \\

\noindent
As $f(x^{k + 1}) \prec f(x^k)$ for all $k \in \mathbb{N}$, the elements in the sequence $\{x^k\}$ after the $\ell$-th term lies in the set $\mathcal{L}(f(x^\ell)) := \{x \in \mathbb{R}^n: f(x) \prec f(x^\ell)\}$.  As $\mathcal{L}(f(x^\ell))$ is bounded, the sequence 
$\{x^k\}$ is bounded. That is, there exists $\alpha > 0$ such that 
\begin{equation} \label{aux1_28_01_25}
\|x^k\| \le \alpha \mbox{ for all } k \in \mathbb{N}. 
\end{equation}
As $f$ is a continuous function, the set $f(\{x \in \mathbb{R}^n: \|x\| \le \alpha\})$ is compact. Thus, there exists $\gamma > 0$ such that 
\begin{align*} 
~&~ \|f(x)\| \le \gamma \mbox{ for all } x \in S \mbox{ with } \|x\| \le \alpha \\ 
\overset{\eqref{aux1_28_01_25}}{\implies} ~&~ \|f(x^k)\| \le \gamma \mbox{ for all } k \in \mathbb{N}.  
\end{align*}
Thus, for each $\xi \in C$, $|\la f(x^k), \xi \ra| \le \|f(x^k)\| \|\xi\| \le \gamma$  for all  $k \in \mathbb{N}$. 
Hence, for each $\xi \in C$, the sequence $\la f(x^k), \xi \ra$ is convergent. \\ 

\noindent
For any given $\xi \in C$, denote $\bar{f}_\xi := \lim_{k \to \infty} \la f(x^k), \xi \ra$. Then, 
\begin{align*}
~&~ \lim_{k \to \infty} \la \xi, f(x^k) - \xi \bar{f}_\xi \ra = 0\\ 
\implies ~&~ \lim_{k \to \infty} \| f(x^k) - \xi \bar{f}_\xi \| = 0 \\ 
\implies ~&~ \left\| \lim_{k \to \infty} \left(f(x^k) - \xi \bar{f}_\xi \right) \right\| = 0, 
\end{align*}
i.e., $\{f(x^k)\}$ is convergent. \\ 

\item 
It is directly followed from \eqref{aux1_28_01_25} that the set $X$ of all subsequential limits of $\{x^k\}$ is non-empty. \\ 

\noindent
Let $x', x'' \in X$. Then, there exist two subsequences 
$\{x^{n'_k}\}$ and $\{x^{n''_k}\}$ of $\{x^k\}$ which converge to $x'$ and $x''$, respectively. \\ 

\noindent
As $f$ is a continuous function and $\lim_{k \to \infty} f(x^k) = \Gamma$, we have 
\begin{align*}
~&~ \lim_{k \to \infty} f(x^{n'_k}) =  \lim_{k \to \infty} f(x^{k}) = \lim_{k \to \infty} f(x^{n''_k}) \\ 
\mbox{i.e., } ~&~ f(x') = \Gamma = f(x''), 
\end{align*}
which proves that $f(x) = \Gamma$ for all $x \in X$. \\

\item 
Let $\{x^{k_j}\}$ be the subsequence of $\{x^k\}$ that converge to $\bar x$. Then, from Theorem \ref{aux2_28_01_25}, we obtain $\lim_{j \to \infty} r_{M_{k_j}}(x^{k_j}) = 0$. So, we get from Lemma \ref{upper_bound_of_J} that $\lim_{j \to \infty} \min_{\xi \in C} \| \la J f(x^{k_j}), \xi \ra \| =  0$. That is, 
$\min_{\xi \in C} \la J f(\bar x), \bar \xi \ra = 0$.  \\

\noindent
As $C$ is compact and the mapping $\xi \mapsto \lambda_1(\la \nabla^2 f(\bar x), \xi \ra)$ is continuous, there exists $\bar \xi \in C$ such that 
\[ \la Jf(\bar x), \bar \xi \ra = \min_{\xi \in C} \la J f(\bar x), \bar \xi \ra = 0. \]

\noindent
Since $\{M_{k_j}\}$ is bounded, we get from $\lim_{j \to \infty} r_{M_j}(x^j) = 0$ and \eqref{aux1_23_02} that for all $\xi \in C$, 
\[ \lambda_1 \left(\la \nabla^2 f(\bar x), \xi \ra \right) = \lim_{j \to \infty} \lambda_1 \left( \la \nabla^2 f(x^{k_j}), \xi \ra \right) \ge - \lim_{j \to \infty} \left\{ \left( L + \frac{M_{k_j}}{2} \right) r_{M_{kj}}(x^{k_j}) \right\} = 0.  \]

\noindent
Therefore, in particular, choosing $\xi$ as $\bar \xi$, we get $\lambda_1(\la \nabla^2 f(\bar x), \bar \xi \ra) \ge 0$. Hence, the result follows. 

\end{enumerate} 

\end{proof}

\begin{remark}
Note from Theorem  \ref{x_k_conv_to_weak_min} (\ref{third_th_3_5}) that any limit point of any sequence generated by Algorithm \ref{algo} satisfies first order and second order necessary conditions of a minimum point of $\la f(x), \bar \xi \ra$ for some $\bar \xi \in C$. As from Lemma \ref{for_one_xi_is_weak_min}, a minimum point of $\la f(x), \bar \xi \ra$ is a weakly efficient point of \eqref{vop}, it is likely that subsequential limits of a sequence $\{x^k\}$, generated by Algorithm \ref{algo}, are weakly efficient points. However, such weakly efficient points ($\bar x$'s) may be ``degenerate" in the sense that at least one eigenvalue of $\la \nabla^2 f(\bar x), \bar \xi \ra$ is zero.  \\ 
\end{remark}

Next, we analyze the behavior of a sequence generated by Algorithm \ref{algo} around a local maximum of $\la f(x), \xi \ra$. \\ 
 
\begin{theorem}\label{theorem_max_is_skipped}
Let $ \bar x \in \mathrm{int} (S)$ be such a point that $\la Jf(\bar x), \xi \ra = 0$ and $\lambda_n\left( \la \nabla^2 f(\bar x), \xi \ra \right) < 0$ for all $\xi \in C$. Suppose that $\{x^k\}$ is a sequence generated by Algorithm \ref{algo}. Then, there exists $\varepsilon_1 > 0$ and $\varepsilon_2 > 0$ such that whenever $x^j$ satisfies 
\[\|x^j - \bar x\| < \varepsilon_1 \mbox{ and } f(x^j) \succeq f(\bar x),    \]
the next point $x^{j + 1}$ holds 
\[\max_{\xi \in C} \la f(x^{j + 1}), \xi \ra - \max_{\xi \in C} \la f(\bar x), \xi \ra \le - \varepsilon_2. \] 
\end{theorem}

\begin{proof}
Let $\Lambda := \max_{\xi \in C} \lambda_n\left( \la \nabla^2 f(\bar x), \xi \ra \right)$. Then, by Rayleigh inequality, for any $u \in \mathbb{S}^{p - 1}$, we have for all $\xi \in C$ that 
\begin{align} \label{aux1_29_01}
u^\top \la  \nabla^2 f(\bar x), \xi \ra u \le \lambda_n \left( \la \nabla^2 f(\bar x), \xi \ra \right) \le \Lambda < 0. 
\end{align}
As $\bar x \in \mathrm{int} (S)$, for any $\bar u \in \mathbb{R}^n$, there exists $\bar t \in \left(0, -\frac{3 \Lambda}{26L}\right]$ such that $\bar x + \bar t \bar u \in S$ for all $t \in (-\bar t, \bar t)$. \\ 

\noindent
Fix an $\bar u \in \mathbb{R}^n$. From Step 2 of Algorithm \ref{algo}, we get for any $\eta \in C$ and $-\bar t < t < \bar t$ that 
\allowdisplaybreaks
\begin{align*} 
~&~ \la f(x^{j + 1}), \eta \ra \\
\le ~&~ h_{M_j} (x^j) \\      
\le ~&~ \min_{y \in S} \left( \max_{\xi \in C} \la f(y), \xi \ra + \frac{L + M_j}{6} \| y - x^j \|^3 \right), \mbox{ by } \eqref{h_M_is_less} \\ 
\le ~&~ \max_{\xi \in C} \la f(\bar x + t \bar u), \xi \ra + \frac{L + M_j}{6} \| (\bar x + t \bar u) - x^j \|^3 \\ 
\le ~&~ \max_{\xi \in C} \la f(\bar x + t \bar u), \xi \ra + \frac{L}{2} \| (\bar x + t \bar u) - x^j \|^3 \mbox{ since } M_j \le 2L \\
\le ~&~ \max_{\xi \in C} \la f(\bar x + t \bar u), \xi \ra + \frac{L}{2} \left( \|\bar x - x^j\|^2 + 2 t \la \bar x - x^j, \bar u \ra + t^2 \right)^{\frac{3}{2}} \\ 
\le ~&~ \max_{\xi \in C} \la f(\bar x) + t Jf(\bar x) \bar u + \frac{t^2}{2} {\bar u}^\top \nabla^2 f(\bar x) \bar u, ~\xi \ra + \frac{L}{6} |t|^3 
+ \frac{L}{2} \left( \varepsilon_1^2 + 2 t \la \bar x - x^j, \bar u \ra + t^2 \right)^{\frac{3}{2}}, 
\mbox{ by } \eqref{cube_equation} \\
\le ~&~ \max_{\xi \in C} \la f(\bar x), \xi \ra + \frac{\Lambda t^2}{2} + \frac{L}{6} |t|^3 
+ \frac{L}{2} \left( \varepsilon_1^2 + 2 t \la \bar x - x^j, \bar u \ra + t^2 \right)^{\frac{3}{2}}, 
\mbox{ by } \eqref{aux1_29_01}. 
\end{align*}
Since $t$ can be chosen positive or negative, we have 
\[\max_{\xi \in C} \la f(x^{j + 1}), \xi \ra - \max_{\xi \in C} \la f(\bar x), \xi \ra 
\le 
\frac{\Lambda {\bar t}^2}{2} + \frac{L {\bar t}^3}{6}  
+ \frac{L}{2} \left( \varepsilon_1^2 + {\bar t}^2 \right)^{\frac{3}{2}}.  
\]
Taking $\varepsilon' := \min \{\bar t, \varepsilon_1\}$, we obtain 
\allowdisplaybreaks
\begin{align*}
 ~&~ \max_{\xi \in C} \la f(x^{j + 1}), \xi \ra - \max_{\xi \in C} \la f(\bar x), \xi \ra  \\ 
\le ~&~ \frac{\Lambda {\varepsilon'}^2}{2} + \frac{L {\varepsilon'}^3}{6} 
+ \sqrt{2} L \varepsilon'^3  
< \frac{\Lambda {\varepsilon'}^2}{2} + \frac{13 L {\varepsilon'}^3}{6} \\ 
< ~&~ \frac{\Lambda {\varepsilon'}^2}{2} + \frac{13 L {\varepsilon'}^2}{6} \left( - \frac{3 \Lambda}{26 L}\right) \mbox{ since } \varepsilon' \le \bar t <   - \frac{3 \Lambda}{26 L} \\ 
= ~&~ \frac{\Lambda {\varepsilon'}^2}{4}.  
\end{align*}
Taking $\varepsilon_2 : = - \frac{\Lambda {\varepsilon'}^2}{4} > 0$, the result follows. 
\end{proof}

\medskip
\begin{remark} \label{non-weakly-effieient-remark}
The result in Theorem \ref{theorem_max_is_skipped} implies that $\max_{\xi \in C} \la f(x^{j + 1}), \xi \ra < \max_{\xi \in C} \la f(\bar x), \xi \ra$, which is a necessary condition of $f(x^{j + 1}) \prec f(\bar x)$. Hence, whenever $x^j$ comes inside the following $\varepsilon_1$-neighborhood of a local maximum of $\la f(x), \xi \ra$ for all $\xi \in C$: 
\[N_{\varepsilon_1} (\bar x):= \left\{ x \in \mathbb{R}^n: \|x - \bar x\| < \varepsilon_1,~ f(x) \succeq f(\bar x) \right\}, \]
immediately the next iterate $x^{j + 1}$ leaves the neighborhood $N_{\varepsilon_1} (\bar x)$. Thus, none of the subsequential limits of the sequence $\{x^k\}$ generated by Algorithm \ref{algo} is a 
non-weakly efficient point. \\ 
\end{remark}

\begin{theorem}\label{theorem_3_7}
Let the initial point $x^0$ be such that $\min_{\xi \in C} \lambda_1(\la \nabla^2 f(x^0), \xi \ra) > 0$, and $\{x^k\}$ be the corresponding sequence generated by Algorithm \ref{algo}. Suppose that $\eta^k := \eta(x^k)$ is the vector $\sum_{i = 1}^{s(\bar x)} \lambda_i(\bar x) \eta^i(\bar x)$ as defined in \eqref{10_01_25_aux2} when $\bar x = x^k$  and $M = M_k$. Let $\{x^k\}$ lie inside a compact set $B \subset S$. Then, the sequence $\{\gamma_k\}$ defined by 
\[\gamma_k : = \frac{L \left\| \la Jf(x^k), \eta^k \ra \right\|}{\left(\min_{\xi \in C} \lambda_1 \left(\la \nabla^2 f(x^k), \xi \ra \right)\right)^2} \]
is well-defined if $\gamma_0 \le \frac{1}{4}$, and there hold the following results. 
\medskip
\begin{enumerate}[(i)] 
\item The sequence $\{\gamma_k\}$ satisfies  
\begin{equation}\label{gamma_k_ineqality}
\gamma_{k + 1} \le \frac{3}{2} \left(\frac{\gamma_k}{1 - \gamma_k}\right)^2 \le \frac{8}{3} \gamma_k^2 \le \frac{2}{3} \gamma_k ~\text{ for all }~ k = 0, 1, 2, \ldots 
\end{equation}  
and $q$-quadratically converges to zero. 

\medskip
\item For any $k$, the minimum eigenvalue of $\la \nabla^2 f(x^k), \xi \ra$ satisfies 
\begin{equation}\label{hessian_eigen_bound} 
\frac{1}{e}  \le \frac{\min_{\xi \in C} \lambda_1 \left(\la \nabla^2 f(x^k), \xi \ra \right)}{\min_{\xi \in C} \lambda_1 \left(\la \nabla^2 f(x^0), \xi \ra \right)} \le e^{\frac{3}{4}}. 
\end{equation}

\item 
The sequence $\{x^k\}$ is q-quadratically convergent and it satisfies 
\[ \left\| \la Jf(x^k), \eta^k \ra \right\| 
\le \frac{3 e \sqrt{e}}{8L} \left(\min_{\xi \in C} \lambda_1 \left(\la \nabla^2 f(x^0), \xi \ra \right) \right)^2 \left(\frac{2}{3}\right)^{2^k}. \]
\end{enumerate}
\end{theorem}

\begin{proof}
\begin{enumerate}[(i)]
\item 
At first, by the method of mathematical induction, we prove that for all $k = 0, 1, 2, \ldots$, we have $\min_{\xi \in C} \lambda_1\left(\la \nabla^2 f(x^k), \xi \ra \right) > 0$. This will then prove that the sequence $\{\gamma_k\}$ is well-defined. \\ 

\noindent
Let $\min_{\xi \in C} \lambda_1\left(\la \nabla^2 f(x^j), \xi \ra \right) > 0$ for some $j \ge 0$. Then, $\gamma_j$ is well-defined. Suppose that $\gamma_j \le \frac{1}{4}$. \\ 

\noindent
We show that $\min_{\xi \in C} \lambda_1 \left(\la \nabla^2 f(x^{j+1}), \xi \ra \right) > 0$ and $\gamma_{j + 1} \le \frac{1}{4}$. \\ 

\noindent
We note from \eqref{10_01_25_aux2} that 
since $\lambda_1\left(\la \nabla^2 f(x^j), \eta^j \ra \right) > 0$, we have 
\allowdisplaybreaks
\begin{align}\label{aux1_20_02_25}
\left\| d_{M_j}(x^j) \right\| = ~&~ \left\| \left( \la \nabla^2 f(x^j), \eta^j \ra + \frac{M_j}{2} r_{M_j}(x^j) I_n \right)^{-1} \la Jf(x^j), \eta^j \ra \right \| \notag \\ 
\le ~&~ \frac{\left\| \la Jf(x^j), \eta^j \ra \right\|}{\lambda_1\left( \la \nabla^2 f(x^j), \eta^j \ra + \frac{M_j}{2} r_{M_j}(x^j) I_n \right)} \notag \\ 
\le ~&~ \frac{\left\| \la Jf(x^j), \eta^j \ra \right\|}{\lambda_1\left( \la \nabla^2 f(x^j), \eta^j \ra \right)} ~\le~ 
\frac{ \|\la Jf(x^j), \eta^j \ra \|}{\min_{\xi \in C} \lambda_1 \left( \la \nabla^2 f(x^j), \xi \ra \right)} = \frac{\gamma_j}{L} \min_{\xi \in C} \lambda_1 \left( \la \nabla^2 f(x^j), \xi \ra\right). 
\end{align}
From the assumption \eqref{assumption_ii} on Lipschitz condition, we have 
for all $\xi \in C$ that 
\begin{align}
& \|\la \nabla^2 f(x^{j + 1}) - \nabla^2 f(x^j), \xi \ra \| \le \| \nabla^2 f(x^{j + 1}) - \nabla^2 f(x^j)\| \le L \|x^{j + 1} - x^j\| \label{aux4_22_02_24} \\ 
\implies & \la \nabla^2 f(x^{j + 1}), \xi \ra \succeq \la \nabla^2 f(x^j), \xi \ra - L \| d_{M_j}(x^j) \| I_n. \notag
\end{align}
Thus, we get 
\begin{align}\label{aux1_22_02_25}
\lambda_1\left( \la \nabla^2 f(x^{j + 1}), \xi \ra \right) \ge  \lambda_1(\la \nabla^2 f(x^j), \xi \ra) - r_{M_j}(x^j) L I_n \text{ for all } \xi \in C. 
\end{align}

\noindent
So, from \eqref{aux1_22_02_25}, with the help of \eqref{aux1_20_02_25}, we get for all $\xi \in C$ that 
\begin{align} 
\lambda_1\left( \la \nabla^2 f(x^{j + 1}), \xi \ra \right) 
& \ge \lambda_1(\la \nabla^2 f(x^j), \xi \ra) - r_{M_j}(x^j) L  \notag \\ 
& \ge \lambda_1(\la \nabla^2 f(x^j), \xi \ra) - \gamma_j ~ \min_{\xi \in C} \lambda_1 \left( \la \nabla^2 f(x^j), \xi \ra \right)   \notag \notag \\
& \ge (1 - \gamma_j) \min_{\xi \in C} \lambda_1 \left( \la \nabla^2 f(x^j), \xi \ra \right).  
\label{aux2_22_02_25}
\end{align}
Therefore, $\min_{\xi \in C} \lambda_1 \left(\la \nabla^2 f(x^{j+1}), \xi \ra \right) > 0$. \\

\item 
Now, notice from Lemma \ref{upper_bound_of_J} and $M \le 2L$ that 
\begin{align}
\gamma_{j + 1} ~&~ = \frac{L \| \la Jf(x^{j + 1}), \eta^{j + 1} \ra \|}{\left(\min_{\xi \in C} \lambda_1 \left(\la \nabla^2 f(x^{j+1}), \xi \ra \right)\right)^2} \le \frac{L (L + M)}{2} \frac{r^2_{M_j}(x^j)}{\left(\min_{\xi \in C} \lambda_1\left(\la \nabla^2 f(x^{j + 1}), \xi \ra \right)\right)^2}\notag \\
~&~ \le  \frac{3 \gamma_j^2 }{2}  \left( \frac{\min_{\xi \in C} \lambda_1\left(\la \nabla^2 f(x^{j}, \xi \ra\right)}{\min_{\xi \in C} \lambda_1\left(\la \nabla^2 f(x^{j + 1}, \xi \ra\right)}\right)^2 
\overset{\eqref{aux2_22_02_25}}{\le} \frac{3}{2} \left( \frac{\gamma_j}{1 - \gamma_j}\right)^2.  \notag 
\end{align}
Therefore, as $\gamma_j \le \frac{1}{4}$, we have 
\[\gamma_{j + 1} \le \frac{3}{2} \left( \frac{\gamma_j}{1 - \gamma_j} \right)^2 \le \frac{8 \gamma^2_j}{3} \le \frac{2 \gamma_j}{3} \le \frac{1}{4}. \]
Thus, by induction, $\{\gamma_k\}$ is well-defined, and \eqref{gamma_k_ineqality} is followed. \\ 

\noindent
Next, we notice from \eqref{aux2_22_02_25} that 
\begin{align}
& \frac{\min_{\xi \in C} \lambda_1(\la \nabla^2 f(x^{j+1}), \xi \ra)}{\min_{\xi \in C} \lambda_1(\la \nabla^2 f(x^j), \xi \ra)} \ge (1 - \gamma_j) \text{ for all }j = 0, 1, 2, \ldots \notag \\ 
\implies & \frac{\min_{\xi \in C} \lambda_1(\la \nabla^2 f(x^{k}), \xi \ra)}{\min_{\xi \in C} \lambda_1(\la \nabla^2 f(x^0), \xi \ra)} \ge \prod_{j = 0}^k (1 - \gamma_j) \ge \prod_{j = 0}^\infty (1 - \gamma_j)    \text{ as } \gamma_j \le \frac{1}{4} \text{ for all } j \notag \\ 
\implies & \ln \left( \frac{\min_{\xi \in C} \lambda_1(\la \nabla^2 f(x^{k}), \xi \ra)}{\min_{\xi \in C} \lambda_1(\la \nabla^2 f(x^0), \xi \ra)} \right) \ge \sum_{j = 0}^{\infty} \ln(1 - \gamma_j) \ge - \sum_{j = 0}^\infty \frac{\gamma_j}{1 - \gamma_j} \text{ as }\ln \theta \ge 1 - \frac{1}{\theta} \text{ for all } \theta > 0 \notag \\ 
\implies & \ln \left( \frac{\min_{\xi \in C} \lambda_1(\la \nabla^2 f(x^{k}), \xi \ra)}{\min_{\xi \in C} \lambda_1(\la \nabla^2 f(x^0), \xi \ra)} \right) \ge - \frac{1}{1 - \gamma_0} \sum_{j = 0}^\infty \gamma_j \ge -\frac{3\gamma_0}{1 - \gamma_0} \text{ since } \gamma_{j + 1} \le \frac{2}{3}\gamma_j \notag \\  
\implies & \ln \left( \frac{\min_{\xi \in C} \lambda_1(\la \nabla^2 f(x^{k}), \xi \ra)}{\min_{\xi \in C} \lambda_1(\la \nabla^2 f(x^0), \xi \ra)} \right) \ge - 1 \text{ since } \gamma_0 \le \frac{1}{4} \notag \\ 
\implies & \frac{1}{e} \le \frac{\min_{\xi \in C} \lambda_1(\la \nabla^2 f(x^{k}), \xi \ra)}{\min_{\xi \in C} \lambda_1(\la \nabla^2 f(x^0), \xi \ra)}, \notag
\end{align}
which proves the lower bound in \eqref{hessian_eigen_bound}. To prove  the upper bound in \eqref{hessian_eigen_bound}, we see from \eqref{aux4_22_02_24} that for all $j = 0, 1, 2, \ldots$, 
\allowdisplaybreaks
\begin{align*}
~&~ \min_{\xi \in C} \lambda_1 (\la \nabla^2 f(x^{j + 1}), \xi \ra) \le   \min_{\xi \in C} \lambda_1 (\la \nabla^2 f(x^{j}), \xi \ra) + L r_{M_j}(x^j) \overset{\eqref{aux1_20_02_25}}{\le} (1 + \gamma_j) \min_{\xi \in C} \lambda_1 (\la \nabla^2 f(x^{j}), \xi \ra)\\ 
\implies ~&~ \frac{\min_{\xi \in C} \lambda_1 (\la \nabla^2 f(x^{k + 1}), \xi \ra)}{\min_{\xi \in C} \lambda_1 (\la \nabla^2 f(x^{0}), \xi \ra)} \le \prod_{j = 0}^k (1 + \gamma_j) \le  \prod_{j = 0}^\infty (1 + \gamma_j) \\ 
\implies ~&~ \ln \left(\frac{\min_{\xi \in C} \lambda_1 (\la \nabla^2 f(x^{k + 1}), \xi \ra)}{\min_{\xi \in C} \lambda_1 (\la \nabla^2 f(x^{0}), \xi \ra)} \right) \le \sum_{j = 0}^\infty \ln (1 + \gamma_j) \le \sum_{j = 0}^\infty \gamma_j \text{ since } \ln (1 + \theta) \le \theta \text{ for } \theta > 0 \\ 
\implies ~&~ \ln \left(\frac{\min_{\xi \in C} \lambda_1 (\la \nabla^2 f(x^{k + 1}), \xi \ra)}{\min_{\xi \in C} \lambda_1 (\la \nabla^2 f(x^{0}), \xi \ra)} \right) \le 3 \gamma_0 \le \frac{3}{4} \\ 
\implies ~&~ \frac{\min_{\xi \in C} \lambda_1 (\la \nabla^2 f(x^{k + 1}), \xi \ra)}{\min_{\xi \in C} \lambda_1 (\la \nabla^2 f(x^{0}), \xi \ra)} \le e^{\frac{3}{4}}. 
\end{align*}

\item 
From \eqref{hessian_eigen_bound} and \eqref{aux1_20_02_25}, we obtain for every $j = 0, 1, 2, \ldots$ that 
\begin{equation}\label{aux1_13_03_25}
\|x^{j + 1} - x^j \| \le \frac{\gamma_j}{L} \min_{\xi \in C} \lambda_1 (\la \nabla^2 f(x^j), \xi \ra) \le \frac{e^{\frac{3}{4}} \gamma_j}{L} \min_{\xi \in C} \lambda_1 \left(\la \nabla^2 f(x^0), \xi \ra\right) \le \alpha \left(\frac{2}{3} \right)^j,   
\end{equation}
where $\alpha : = \frac{e^{\frac{3}{4}} \gamma_0}{L} \min_{\xi \in C} \lambda_1 \left(\la \nabla^2 f(x^0), \xi \ra\right)$. Therefore, $\{x^k\}$ is a Cauchy sequence and converges $q$-quadratically. \\

\noindent
We see from the definition of $\gamma_k$ that 
\begin{equation}\label{aux7_22_02} 
\left\| \la Jf(x^k), \eta^k \ra \right\| = \frac{\gamma_k}{L}  \left(\min_{\xi \in C} \lambda_1 \left(\la \nabla^2 f(x^k), \xi \ra \right)\right)^2 \overset{\eqref{hessian_eigen_bound}}{\le} \frac{ \gamma_k ~ e \sqrt{e}}{L} \left(\min_{\xi \in C} \lambda_1 \left(\la \nabla^2 f(x^0), \xi \ra \right) \right)^2. 
\end{equation}
From \eqref{gamma_k_ineqality}, we observe for any $k$ that 
\begin{align} 
& \gamma_{k + 1} \le \frac{3}{2} \frac{\gamma_k^2}{(1 - \gamma_k)^2} \le \frac{3}{2} \frac{\gamma_k^2}{(1 - \gamma_{k - 1})^2} \le \cdots \le \frac{3}{2} \frac{\gamma_k^2}{(1 - \gamma_0)^2} \le  \frac{8}{3}  \gamma_k^2 ~\text{ since }~ \gamma_0 \le \frac{1}{4} \notag \\ 
\implies & \gamma_k \le  \frac{8}{3} \gamma^2_{k-1} \le \left(\frac{8}{3}\right)^{2^2 - 1} \gamma_{k - 2}^{2^2} \le \cdots \le \left(\frac{8}{3}\right)^{2^k - 1} \gamma_{0}^{2^k} = \frac{3}{8} \left(\frac{2}{3}\right)^{2^k}.  \notag 
\end{align}
Thus, \eqref{aux7_22_02} yields for all $k$ that 
\[\left\| \la Jf(x^k), \eta^k \ra \right\| 
\le \frac{3 e \sqrt{e}}{8L} \left(\min_{\xi \in C} \lambda_1 \left(\la \nabla^2 f(x^0), \xi \ra \right) \right)^2 \left(\frac{2}{3}\right)^{2^k}. \]
\end{enumerate}

\end{proof}

\begin{remark}
Theorem \ref{theorem_3_7} shows that if the initial point $x^0$ is chosen such a way that $\gamma_0 \le \frac{1}{4}$ and $\min_{\xi \in C} \lambda_1 \la \nabla^2 f(x^0), \xi \ra > 0$, then the sequence $\{x^k\}$ generated by Algorithm \ref{algo} converges to a point $\bar x$ that satisfies $\min_{\xi \in C}\lambda_1 \la \nabla^2 f(\bar x), \xi \ra > 0$. 
Indeed, because continuity of the function $\xi \mapsto \lambda_1 \left(\la \nabla^2 f(x^j), \xi \ra \right)$ for each $x^j \in B$ implies that from \eqref{aux2_22_02_25}, we have the following for all $\xi \in C$: 
\[ \lambda_1 \la \nabla^2 f(\bar x), \xi \ra = \lim_{j \to \infty} \lambda_1 \la \nabla^2 f(x^j), \xi \ra > 0. \]    
Hence, Theorem \ref{theorem_3_7}(iii) implies that $\{x^k\}$ converges $q$-quadratically to a weakly efficient point of \eqref{vop}. Further, note from \eqref{aux1_13_03_25} that the sequence $\{x^j\}$ is a Cauchy sequence. Thus, the restrictions $\gamma_0 \le \frac{1}{4}$ and $\min_{\xi \in C} \lambda_1 \la \nabla^2 f(x^0), \xi \ra > 0$ on the initial point $x^0$ and the compactness of $B$ indicate that the $q$-quadratic convergence of $\{x^k\}$ is local. Indeed local, because there is an assurance of the condition $\min_{\xi \in C} \lambda_1 \la \nabla^2 f(x^0), \xi \ra > 0$ only when $x^0$ is chosen close enough to a minimizer of $\la f(x), \xi \ra$ for all $\xi \in C$. 
\end{remark}

\section{Numerical Experiments} \label{section-numerical-experiment}
In this section, we execute the performance of Algorithm \ref{algo} on some test problems of vector optimization problems. The test problems are mentioned in Table \ref{performance-table}. The detailed statement and reference of all the test problems can be found in \cite{huband2006review, zhao2024convergence}. To identify the numerical data of the performance of Algorithm \ref{algo} on these test problems, a Matlab code of Algorithm \ref{algo} is executed in Matlab 2020a software installed on a computer with Intel(R) Xeon(R) CPU E5-2620 v4 @ 2.10GHz processor with 32.0 GB RAM and Windows 10 Pro OS. \\

The parameters and expressions used during the Matlab code implementation of each step of Algorithm \ref{algo} are as follows. 
\begin{itemize}
\item 
The set $S$ is as given in Table \ref{performance-table}. 

\item 
The ordering cone for each tri-objective test problem is taken as $K = \mathbb{R}^3_+$. For each bi-objective test problem, we take  
\[K = \{(y_1, y_2) \in \mathbb{R}^2_+: y_1 \le 3y_2, y_2 \le 3y_1\}. \]

\medskip
\item The generator $C$ of $\mathbb{R}^3_+$ is taken as the standard basis of $\mathbb{R}^3$, and for $K \subset \mathbb{R}^2$, we take 
\[C = \left\{
\frac{1}{\sqrt{10}} 
\begin{pmatrix} 
-1 \\ 3
\end{pmatrix}, 
\frac{1}{\sqrt{10}} 
\begin{pmatrix}
3 \\ -1     
\end{pmatrix} \right\}. \]

\item 
For all the test problems, we take $L_0 = 1$, $L = 1.5$, $M_0 = 3$ and $\varepsilon = 0.001. $

\item 
The initial point $x^0$ is chosen randomly from $S$ using the ``rand" function of Matlab. 

\item
To execute $M_k$ for each $k$, we follow the process mentioned in Remark \ref{process_to_evaluate_M_k}. 

\item The expressions of $h_{M_k}(x^k)$ and $d_{M_k}(x^k)$ are executed by using ``fmincon" constrained optimization Matlab function with ``Algorithm" option as ``interior-point." 

\end{itemize}

With the help of these parameters, we show a comparison of the proposed method with the existing steepest descent \cite{drummond2005steepest}, (BFGS) quasi-Newton \cite{kumar2023quasi}, (PRP) conjugate gradient \cite{perez2018nonlinear}, and trust region method \cite{carrizo2016trust} for multi-objective or vector optimization. We do not include the Newton method in the comparison because the conventional Newton methods in \cite{drummond2014quadratically, fliege2009newton} are applicable only for (strongly) convex problems. For the sake of convenient description, we use the abbreviations SD, CN, QN, CG, and TR to mean the steepest descent, cubic Newton, (BFGS) quasi-Newton, (PRP) conjugate gradient, and trust-region methods, respectively. The parameters used in the computation for the SD, QN, PRP-CG and TR are provided in Table \ref{parameters_methods}.

\begin{table}[!h]
\caption{Used parameters for different methods}\label{parameters_methods}
\begin{tabular}{lll}
\thickhline
Method                      & Reference & Parameters \\
\thickhline
Steepest Descent (SD)       & \cite[Section 3]{chuong2012steepest}      & $\beta = 0.0001$  \\ 
\hline 
BFGS Quasi-Newton (QN)      & \cite[Algorithm 1]{qu2011quasi}      &  $\beta = 0.0001$ \\
\hline 
PRP-Conjugate Gradient (CG) & \cite[NLCG algorithm]{perez2018nonlinear}      & \begin{tabular}[c]{@{}l@{}}$\beta_k = \max\left\{\beta^{\mbox{PRP}}_k,~0\right\}$, \\ $\rho = 0.0001$,\\ $\sigma = 0.1$, $\delta = 1.1$\end{tabular}   \\
\hline 
Trust Region (TR)           & \cite[Algorithm 1]{carrizo2016trust}      & \begin{tabular}[c]{@{}l@{}}$\Delta_0 = 1$, $\eta_1 = \gamma_1 = 0.1$, \\ $\eta_2 = 0.9 = \gamma_2 = 0.9$\end{tabular} \\ 
\thickhline
\end{tabular}
\end{table}

At first, we take a bi-objective and a tri-objective test problems and compare the error graph (the graph of $\|d_{M_k}(x^k)\|$ or $\|d^k\|$ versus $k$) of the proposed cubic Newton scheme with SD, QN, CG, and TR methods. \\

Take the PNR problem from Table \ref{performance-table}, which is a bi-objective convex optimization problem. For this problem, we randomly choose an initial point from $S = [-2, 2]^2$ and employ all of the CN, SD, QN and CG methods. For the initial point $x^0 = (-0.6049, -0.1946)^\top$, the graphs of the norm of the direction of movement and the iteration number are exhibited in Figure \ref{error-graph-PNR}. From Figure \ref{error-graph-PNR}(a) in linear scale, we clearly see that CN outperforms all other methods. Figure \ref{error-graph-PNR}(b) in log scale shows that the rate of convergence of $\|d^k\|$ is linear for CN and CG methods and superlinear for CN and QN methods.    \\

For the tri-objective case, we take the SLCDT2 problem and depict the error graph in Figure \ref{error-graph-SLCDT2} for a randomly chosen initial point 
$$x^0 = (-0.8654, -0.9445, -0.6519, 0.6733, -0.9563, 0.2677, -0.6400, -0.8537, 0.9082, 0.6217)^\top.$$ 
From this graph, we see that at the initial few iterations, CN is quite faster than all other methods but crawls when it comes closer to the optimum. \\

Next, to find detailed performance profiles and to make a comparison of the proposed method with SD, QN, CG, and TR methods, we use six metrics---median time, median iteration number, hypervolume, purity, spread $\Delta_p$ and spread $\Gamma_p$---whose detailed description and formula can be found in \cite{audet2021performance, zhao2024convergence}. For a particular test problem, to find the metric values for all five methods, we randomly choose a set of fifty initial points from the problem domain $S$ and use the same set of initial points for all the methods. Then, the median value of the fifty data for each of the six metrics is displayed in Table \ref{performance-table}. A dash in the table indicates that the corresponding method for the problem converges for none of the fifty initial points. For instance, QN does not converge for KW2 problem. CG does not converge for MOP3 problem. Possibly, the non-convergence of CG for MOP3 is due to the fact that the convergence of CG is guaranteed only under the assumption that the conjugate gradient direction is $K$-descent \cite{perez2018nonlinear}, which is not followed in MOP3 problem.   \\

Movement of the iterative points $\{f(x^k)\}$ corresponding to $\{x^k\}$, generated by Algorithm \ref{algo}, for the fifty and ten random initial points are depicted in Figure \ref{for-fifty-random-initial-points} and Figure \ref{ten-points-performance}, respectively. In all the sub-figures of these two figures, the blue-colored bullet points are the initial points. The pink-colored path starting from a blue bullet and ending at a green bullet is the trajectory traced by the sequence $\{f(x^k)\}$ corresponding to the blue-colored initial point. The black-colored bullets in the trajectory are the intermediate points of the sequence $\{f(x^k)\}$. The green-colored bullets are terminal points. The grey-shaded region in each figure is the feasible region in the objective space. The sky-blue shaded region in all the two-dimensional figures represents the ordering cone $K$. From the sub-figures corresponding to the PNR and JOS1 problems in Figures \ref{for-fifty-random-initial-points} and \ref{ten-points-performance}, respectively, we see that the proposed Algorithm \ref{algo} is able to successfully generate the weakly efficient points with respect to the general ordering cone $K$; note that the green-colored bullet points generated by Algorithm \ref{algo} on the right-down part of PNR and top-left part of JOS1 feasible regions are not weakly efficient points with respect to the standard ordering cone $\mathbb{R}^2_+$, but those points are weakly efficient points with respect to the cone $K$.      \\

Associated with the data in Table \ref{performance-table}, the performance profiles of the four metrics are exhibited in Figure \ref{figures_of_performance_profiles}. We do not evaluate the performance profile graphs for median time and median iteration because, for some problems, these values are either zero or very close to zero. During the performance profile graphs, we have also not included the problems KW2 and MOP3 because for these two problems, either the QN or CG method does not converge. So, the performance profiles depicted in Figure \ref{figures_of_performance_profiles} are for the metric values in Table \ref{performance-table} for the rest of 14 problems out of 16 test problems. From the performance profiles graphs in Figure \ref{figures_of_performance_profiles}, we see that in terms of the spread metrics, CN outperforms all other methods; in terms of hypervolume and purity, all methods perform almost similarly.   \\

From Table \ref{performance-table}, we see that the median time and iteration taken by the cubic Newton method is either smallest or very close to the second smallest. The reason that CN is not always the smallest is that CN  aims to find weakly efficient points, but all other methods converge just at a stationary point. For instance, we see from the MLF1 figure in Figure \ref{ten-points-performance} that although most of the initial points are stationary points, the CN method does not stop at the initial point and mostly bypasses on its way all the local stationary points and reaches to a global weakly efficient point. Thus, CN takes one iteration to converge, but all other methods just converge at the initial point. Thus, in Table \ref{performance-table}, we see that the median iteration number for CN is one,  but it is zero for all other methods.


\begin{landscape}

\begin{figure}[!h]
\vspace{-1cm}
\centering
\begin{minipage}{.65\textwidth}
\centering
\begin{subfigure}{0.5\textwidth}
\centering 
\includegraphics[scale=0.35]{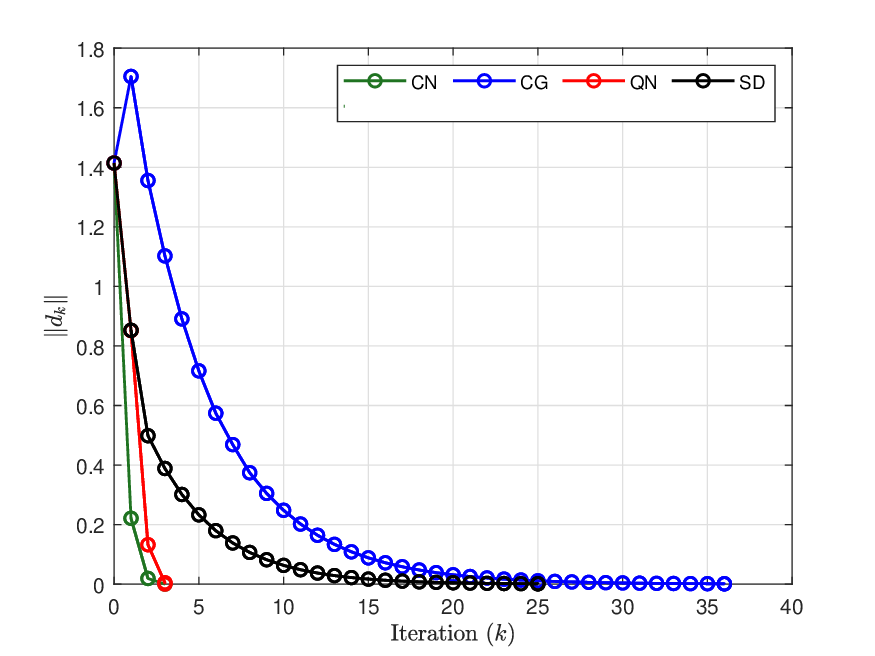}
\caption{In the linear scale}\label{subfig_2_a}
\end{subfigure}%
\begin{subfigure}{0.5\textwidth}
\centering
\includegraphics[scale=0.35]{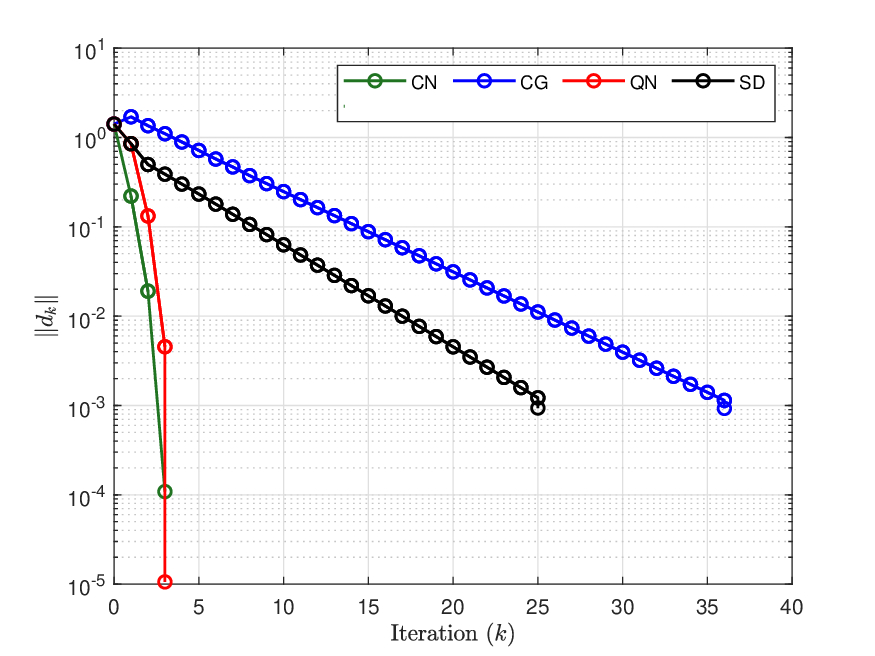}
\caption{In the log scale}\label{subfig_2_b}
\end{subfigure}
\caption{Error graph for the PNR problem}\label{error-graph-PNR}
\end{minipage}%
\quad
\begin{minipage}{0.65\textwidth}
\begin{subfigure}[b]{0.5\textwidth}
\centering 
\includegraphics[scale=0.35]{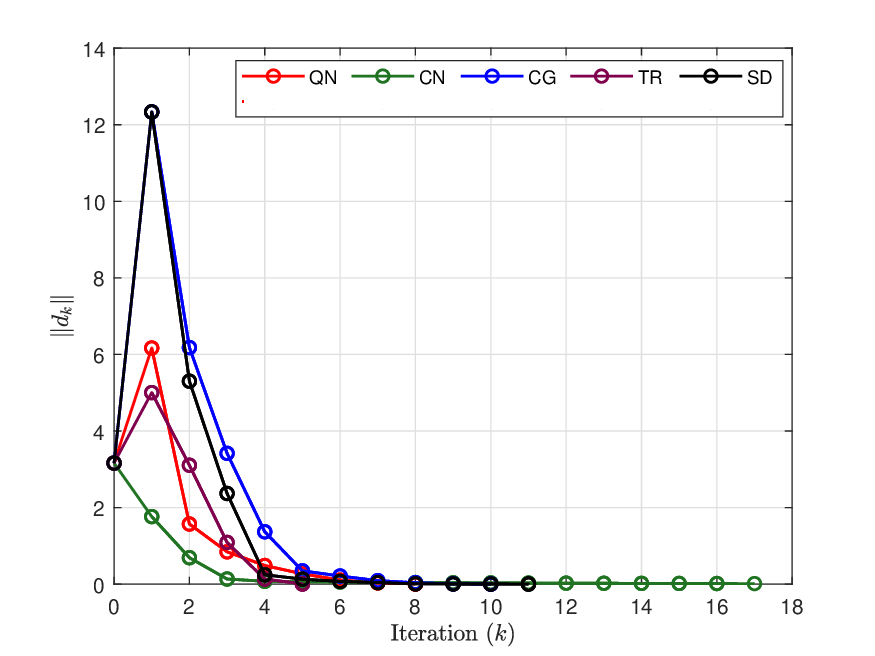}
\caption{In the linear scale}\label{subfig_3_a}
\end{subfigure}%
\begin{subfigure}[b]{0.5\textwidth}
\centering
\includegraphics[scale=0.35]{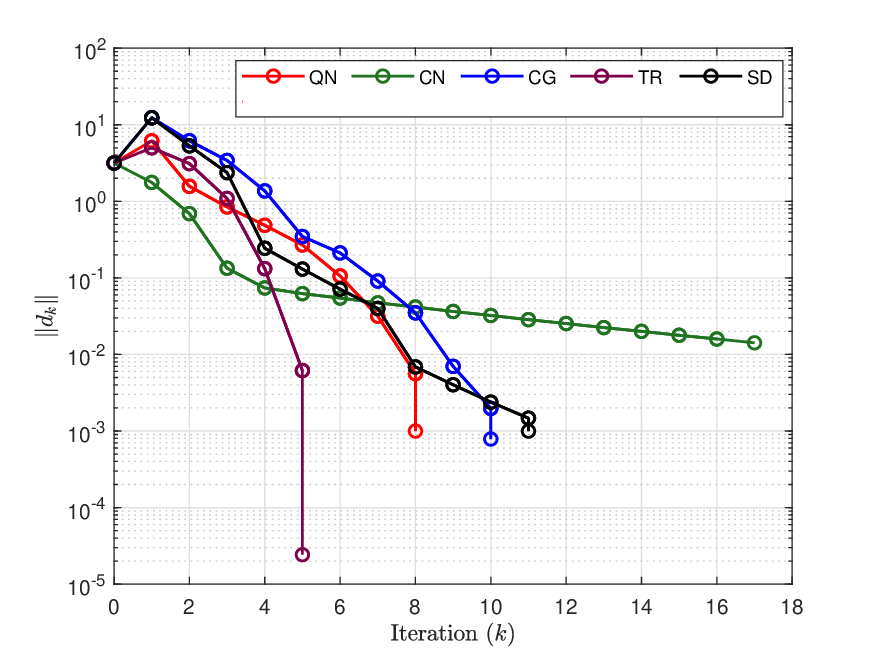}
\caption{In the log scale}\label{subfig_3_b}
\end{subfigure}
\caption{Error graph for the SLCDT2 problem}\label{error-graph-SLCDT2}
\end{minipage}
\end{figure}    


\begin{figure}[!h]
\vspace{-1cm}
\captionsetup[subfigure]{labelformat=empty}
\centering
\includegraphics[scale=0.35]{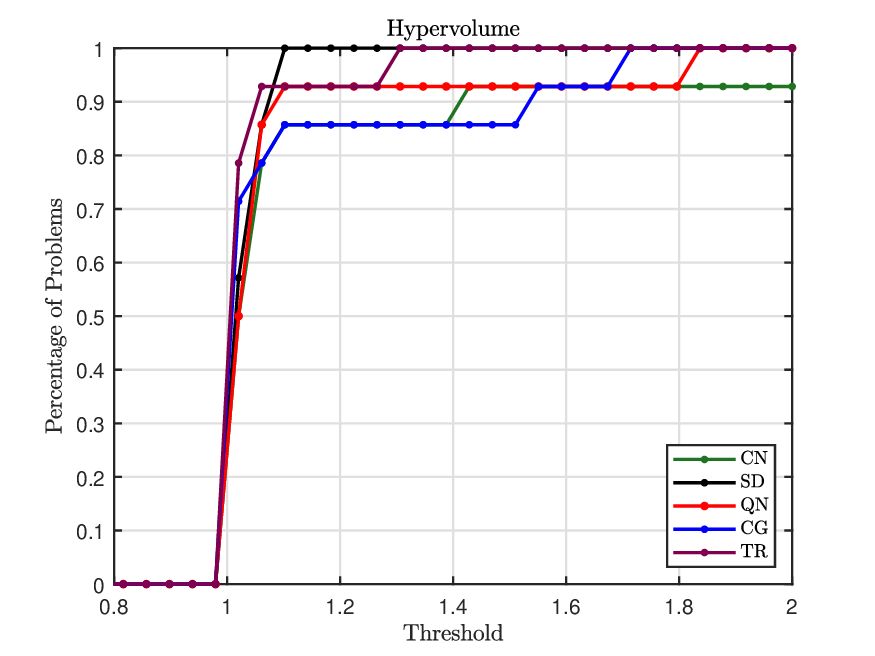}
\includegraphics[scale=0.35]{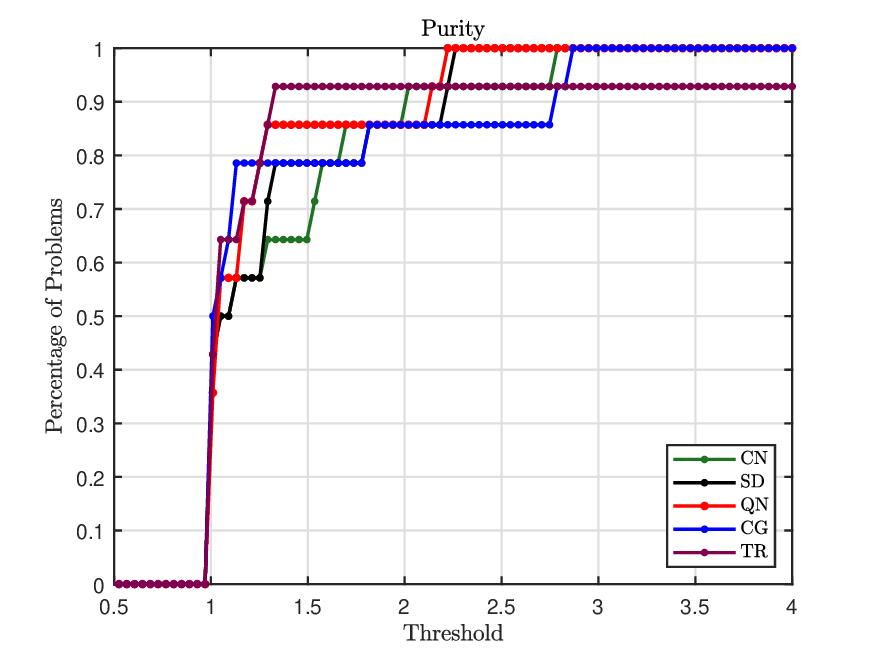} 
\includegraphics[scale=0.35]{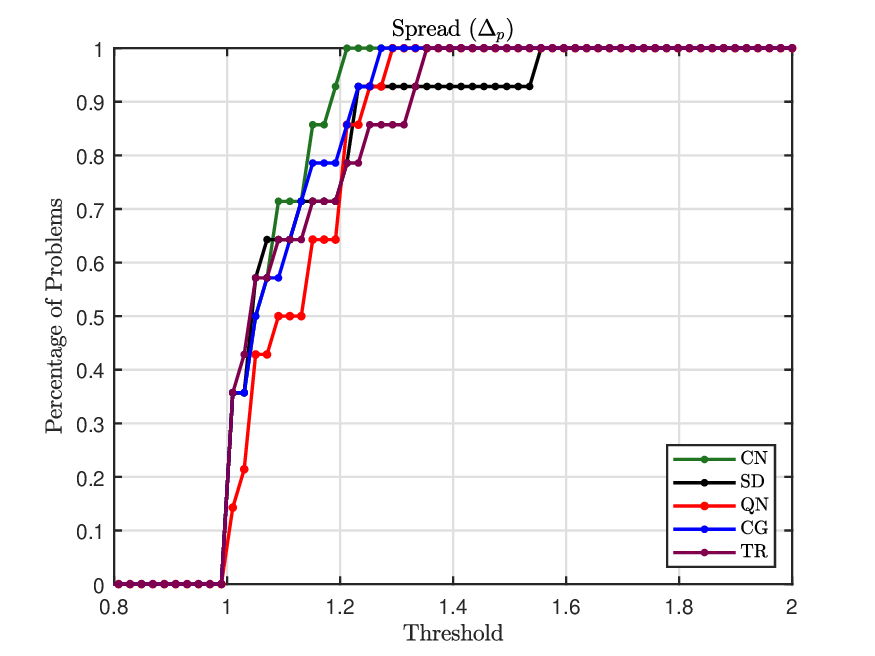}
\includegraphics[scale=0.35]{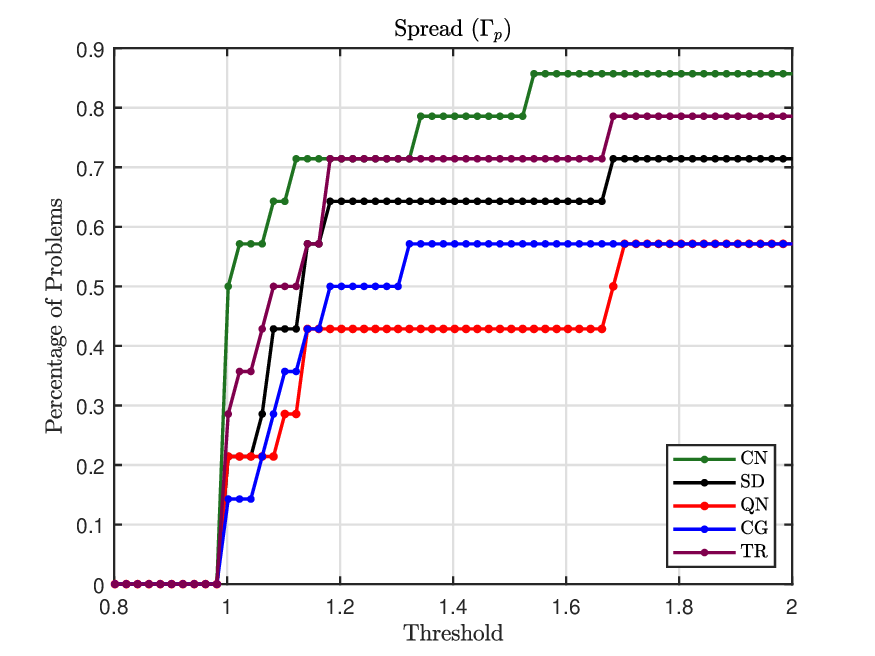} 
\caption{Performance profiles}\label{figures_of_performance_profiles} 
\end{figure}

\vspace{-0.5cm}

\begin{figure}[!h]
\captionsetup[subfigure]{labelformat=empty}
\centering
\includegraphics[scale=0.35]{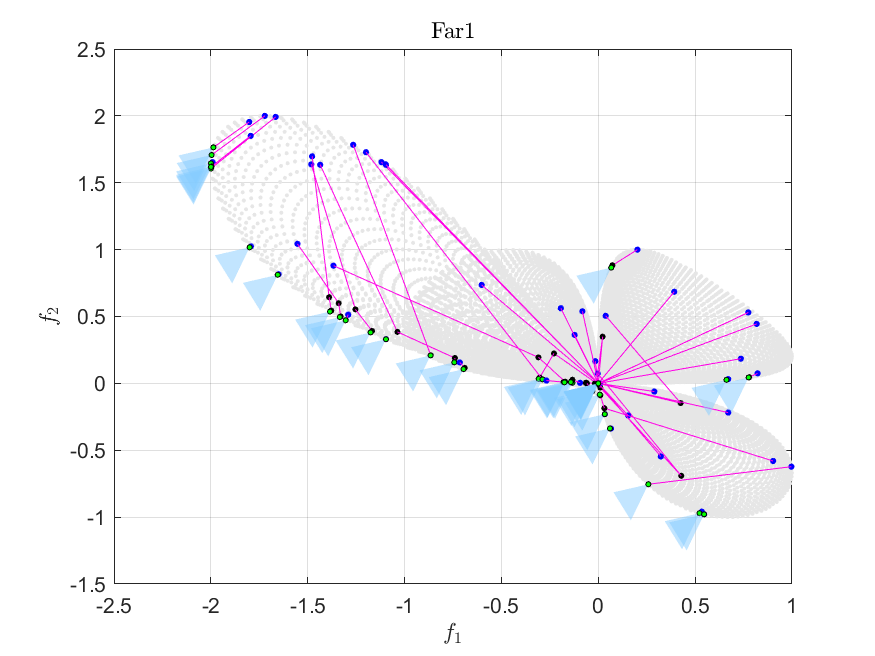}
\includegraphics[scale=0.35]{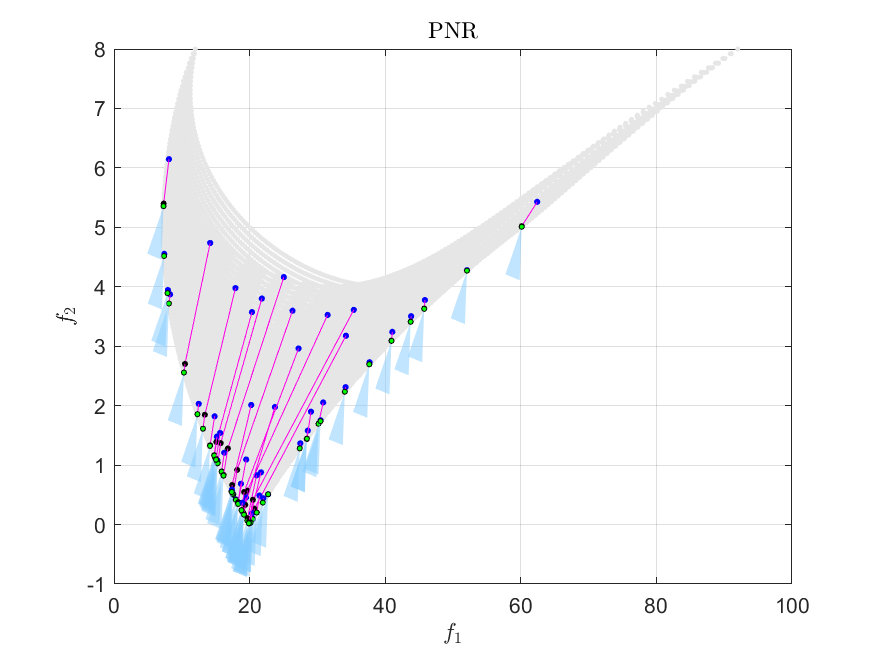}
\includegraphics[scale=0.35]{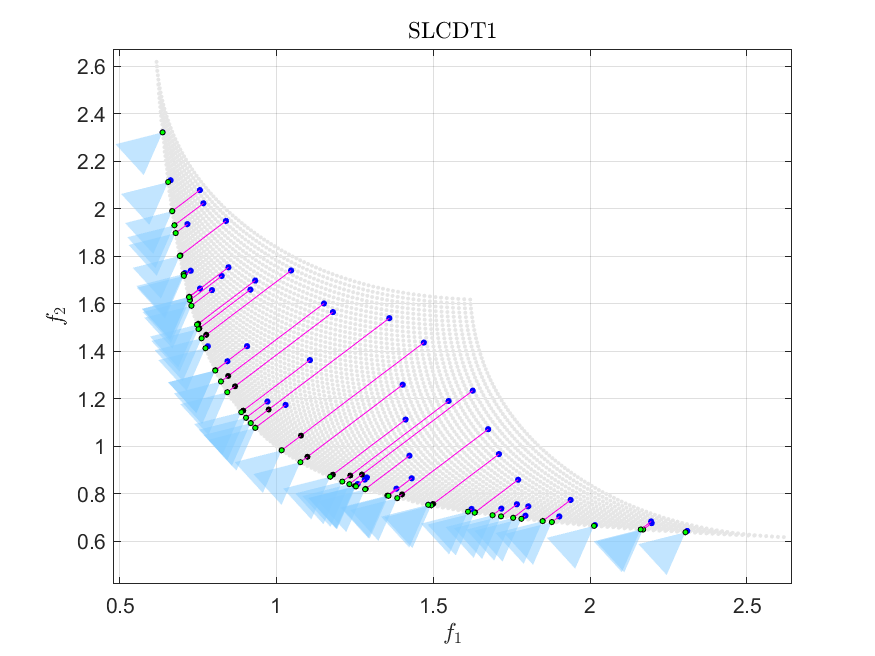}
\includegraphics[scale=0.35]{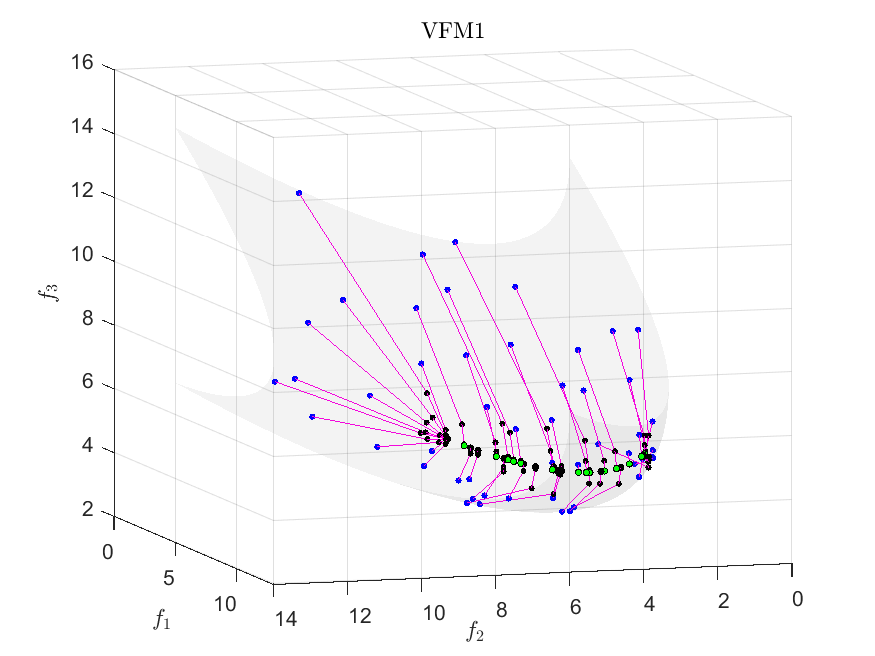}
\caption{Weakly Pareto critical points generated by Algorithm \ref{algo} for a few test problems for fifty random initial points}\label{for-fifty-random-initial-points} 
\end{figure}
\end{landscape}

\begin{landscape}

\begin{figure}[!h]
\captionsetup[subfigure]{labelformat=empty}
\subfloat[]{
\hspace{-0.95cm}
\includegraphics[width=0.35\textwidth]{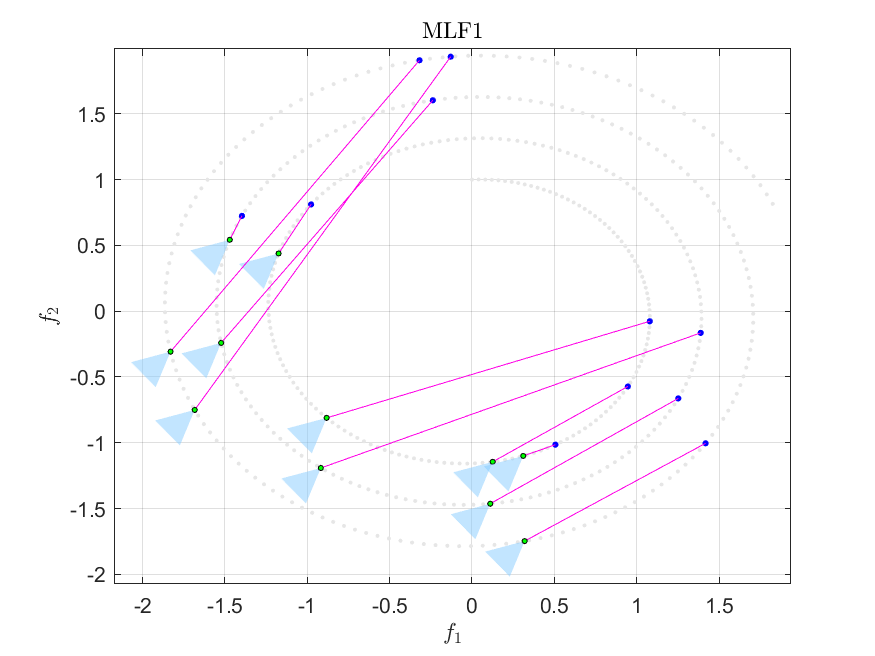}
\includegraphics[width=0.35\textwidth]{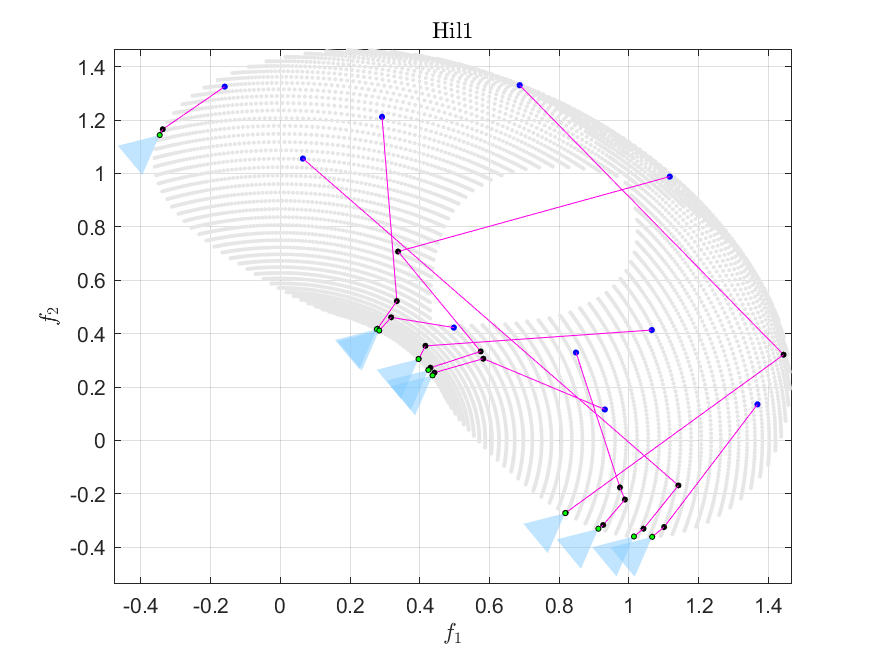}
\includegraphics[width=0.35\textwidth]{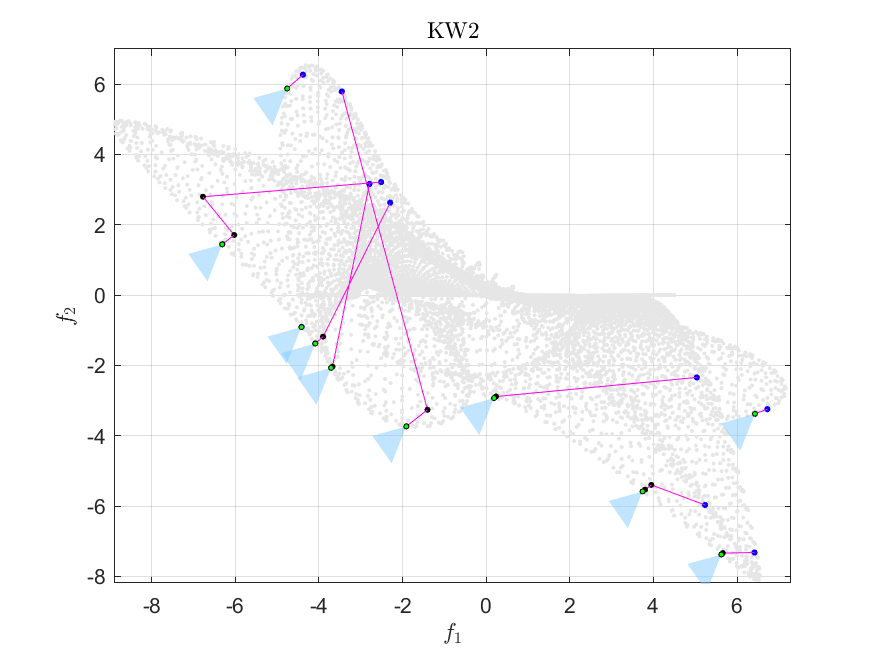}
\includegraphics[width=0.35\textwidth]{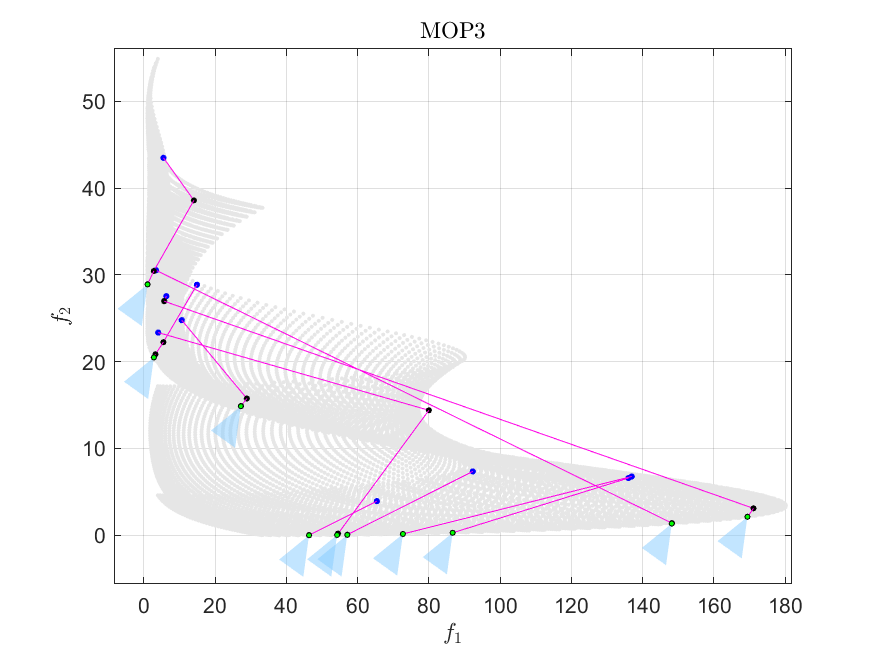}
}
\vspace{-0.6cm}
\subfloat[]{ 
\hspace{-0.95cm}
\includegraphics[width=0.35\textwidth]{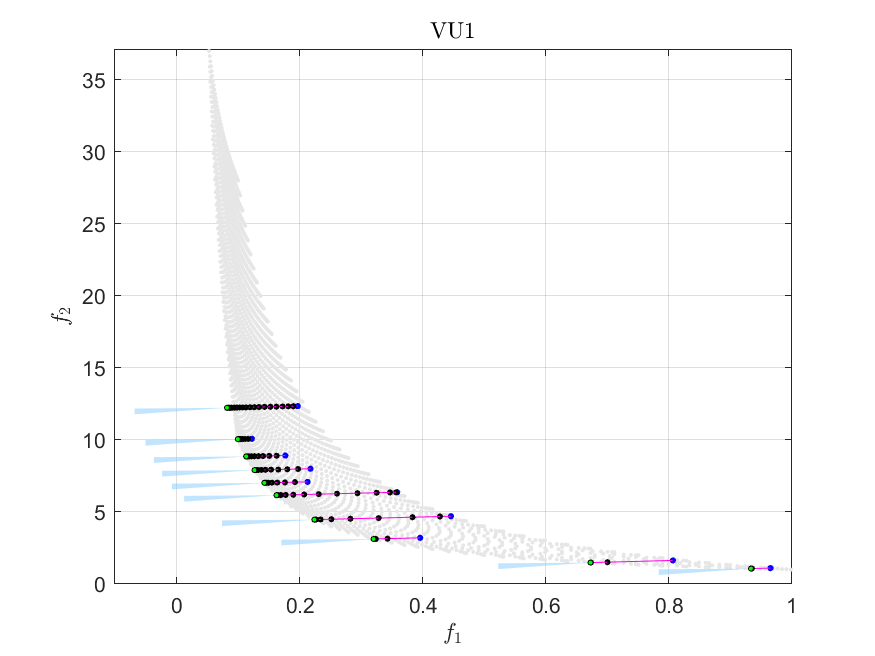}
\includegraphics[width=0.35\textwidth]{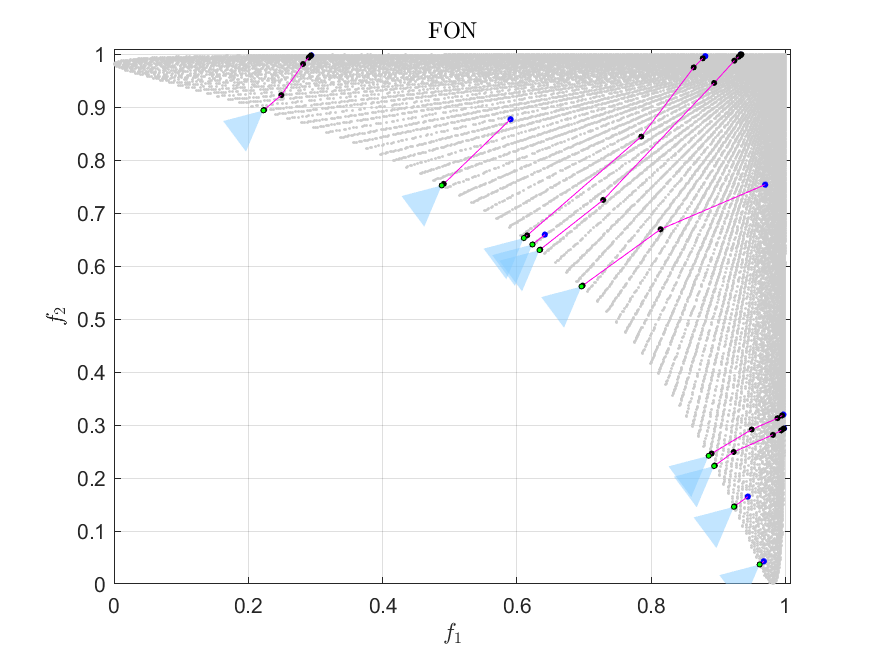}
\includegraphics[width=0.35\textwidth]{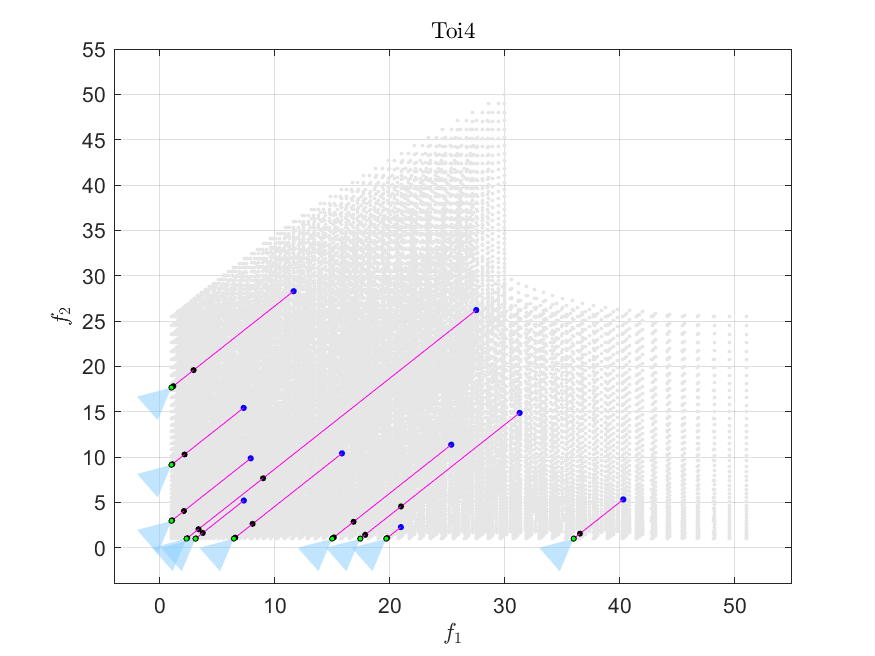}
\includegraphics[width=0.35\textwidth]{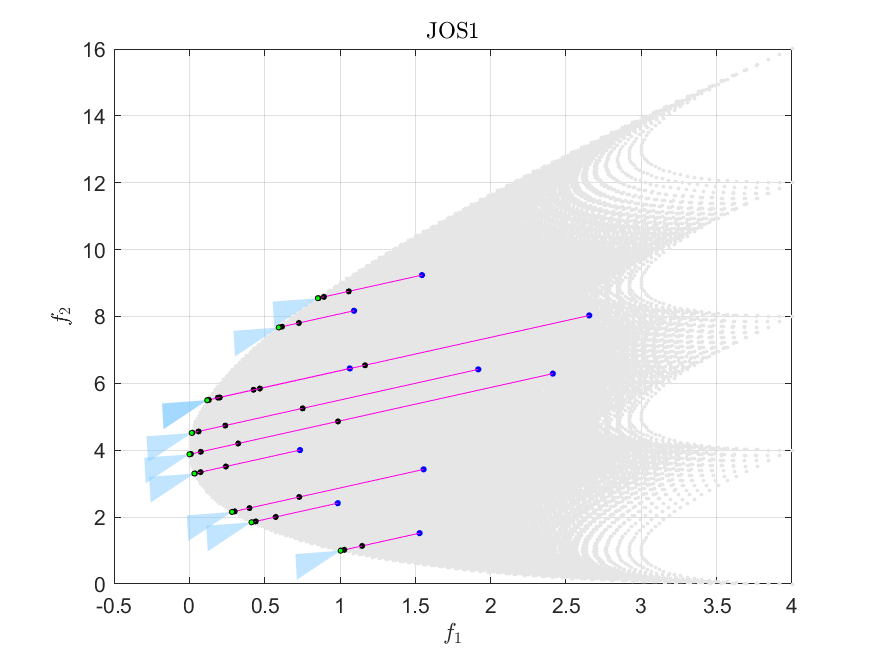}
}
\vspace{-0.6cm}
\subfloat[]{
\hspace{-0.95cm}
\includegraphics[width=0.35\textwidth]{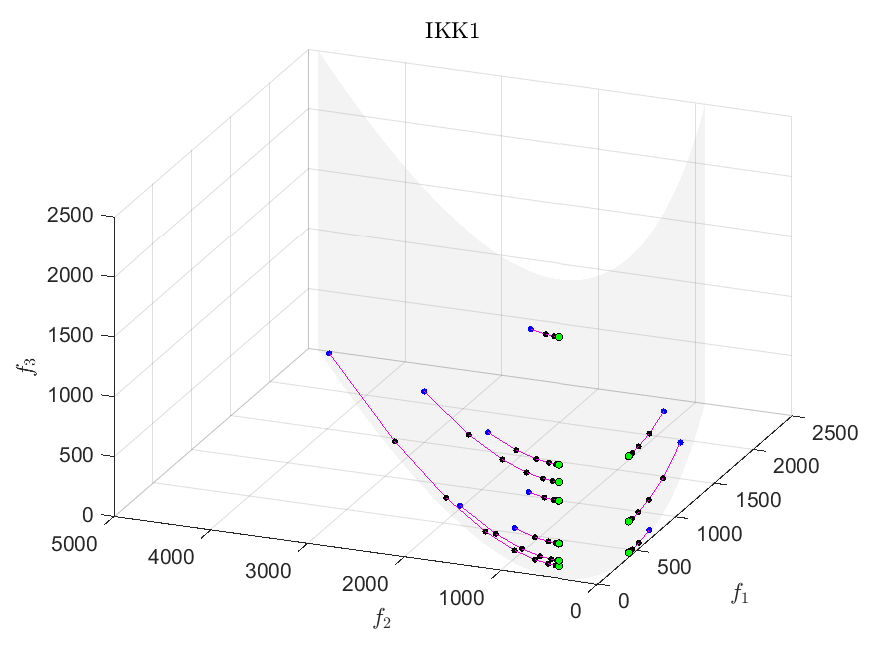}
\includegraphics[width=0.35\textwidth]{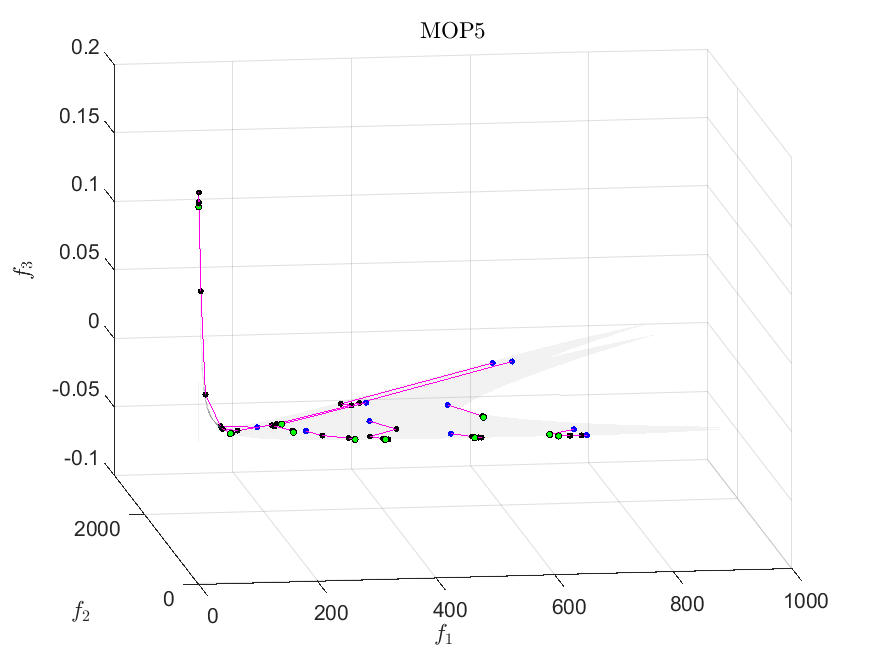}
\includegraphics[width=0.35\textwidth]{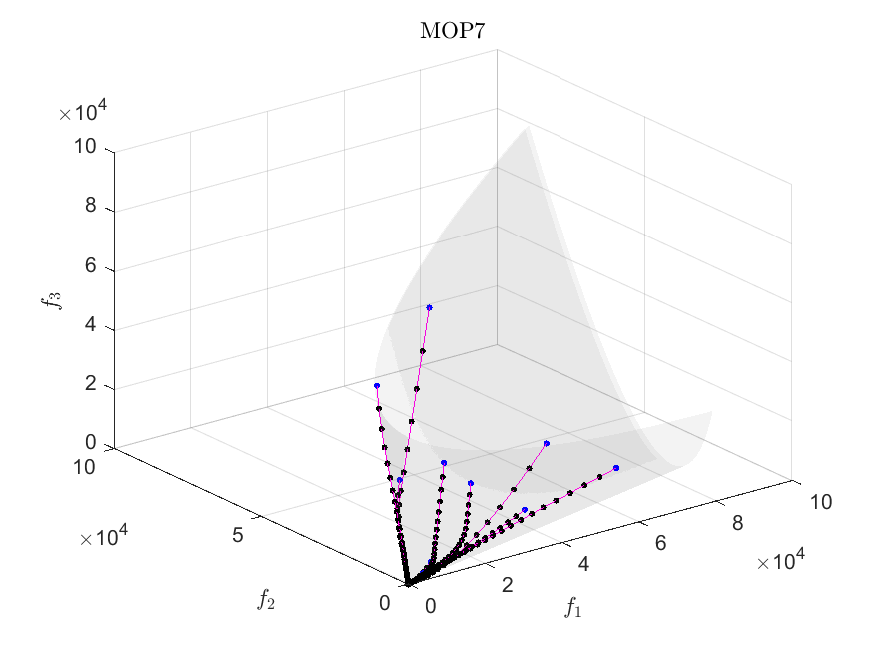}
\includegraphics[width=0.35\textwidth]{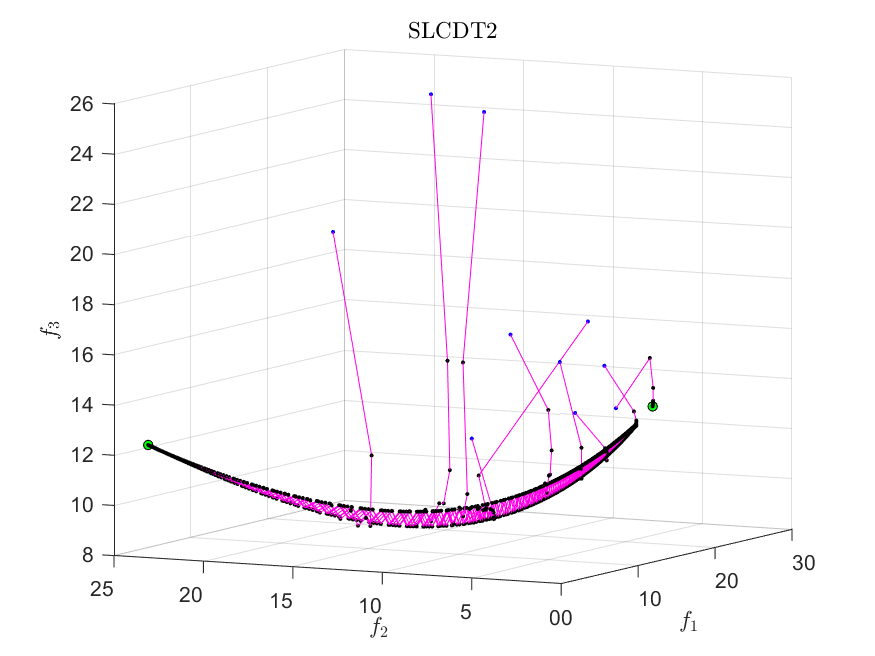}
}

\caption{Weakly Pareto critical points generated by Algorithm \ref{algo} for a few test problems for ten random initial points}\label{ten-points-performance} 
\end{figure}

\end{landscape}

\begin{landscape}

\begin{table}[h]
\caption{Performance metrics on test problems for fifty randomly chosen initial points}\label{performance-table}
\begin{tabular}{cccccccccccc}
\thickhline
Problem               & $m$                & $n$                & Domain $(S)$              & Problem Type                          & Method & Median Time (s) & Median Iteration & Hypervolume & Purity & $\Delta_p$ & $\Gamma_p$ \\ \thickhline
\multirow{5}{*}{MLF1} & \multirow{5}{*}{2} & \multirow{5}{*}{1} & \multirow{5}{*}{$[0,20]$} & \multirow{5}{*}{Non-convex} 

& CN     &   0.0370    &   1   &  2.1468   &  0.2667    &     4.5249   &    0.4334       \\ 

                      &                    &                    &                   & 
                      
& SD     &  0.0092    &  0    &  2.1470   &    0.3333   &   4.9079     &      0.4297     \\    
                      &                    &                    &                   &                                               
& QN     &   0.0197   &   0   &  2.1469   &    0.3333   &    4.7900    &      0.4297     \\

                      &                    &                    &                   &                                               
& CG     &   0.0088   &   0   &  2.1469   &  0.2667    &     5.1305   &       0.4297    \\

                      &                    &                    &                   &                                               
 & TR     &    0.0186  &  0    &  2.2583   &  0.7333    &    5.0189    &       0.4855     \\ \hline
\multirow{5}{*}{Far1} & \multirow{5}{*}{2} & \multirow{5}{*}{2} & \multirow{5}{*}{$[-1,1]^2$} & \multirow{5}{*}{Non-convex}

& CN     &   0.1290    &  3    &  3.6906   &  0.1846    &      1.0234    &      0.5891      \\ 

                      &                    &                    &                   & 
                      
& SD     &   0.2389   &   46.5   &   3.8201  &   0.2846   &   1.0151     &     0.3625      \\    
                      &                    &                    &                   &                                               
& QN     &   0.0958   &   4   &  3.6762   &   0.2308   &   1.0323     &      0.5318     \\

                      &                    &                    &                   &                                               
& CG     &  0.3695    &   5   &  2.5079   &  0.1000    &    1.0110    &      0.9024      \\

                      &                    &                    &                   &                                               
 & TR     &   0.1464   &  5    &   3.7402  &    0.2231   &    1.2309    &     0.3625      \\ \hline
\multirow{5}{*}{PNR} & \multirow{5}{*}{2} & \multirow{5}{*}{2} & \multirow{5}{*}{$[-2,2]^2$} & \multirow{5}{*}{Convex} 

& CN     &  0.1355    &   3   & 39.8090    &  0.2279    &   1.0366     &  8.0728         \\ 

                      &                    &                    &                   & 
                      
& SD     &  0.0634    &  14.5    & 37.3318    &  0.1912    &   1.0022     &  10.2422         \\    
                      &                    &                    &                   &                                               
& QN     &  0.0648    &   4   &  37.9704   &  0.2206    &   0.9992     &   7.8276        \\

                      &                    &                    &                   &                                               
& CG     &  1.5974    &   22.5   &  37.0297   &   0.1397    &   1.0345     & 10.2084          \\

                      &                    &                    &                   &                                               
 & TR     &  0.0580   &    2  &  39.8477   &    0.2500   &   0.9947     &  42.3041         \\ \hline
\multirow{5}{*}{Hil1} & \multirow{5}{*}{2} & \multirow{5}{*}{2} & \multirow{5}{*}{$[0, 1]^2$} & \multirow{5}{*}{Non-convex}

& CN     &   0.1832   &  4    &  1.0257   &  0.1832    &   1.1273     &     0.3643      \\ 

                      &                    &                    &                   & 
                      
& SD     &   0.0655   &   14   &  1.0921   &   0.2304   &  0.7890      &    0.3397       \\    
                      &                    &                    &                   &                                               
& QN     &   0.0642   &   4   &  1.0851   &   0.2251   &   0.9547     &      0.3360     \\

                      &                    &                    &                   &                                               
& CG     &  1.2442    &   22   &   1.0905  &   0.2042    &   0.9676     &    0.3129       \\

                      &                    &                    &                   &                                               
 & TR     &   0.1014   &  3    &   1.0923  &   0.2251   &   1.2265     &     0.3089      \\ \hline
\multirow{5}{*}{KW2} & \multirow{5}{*}{2} & \multirow{5}{*}{2} & \multirow{5}{*}{$[-3,3]^2$} & \multirow{5}{*}{Convex} 

& CN     &  0.1527    &  4    &  83.8527   & 0.1864     &  1.1270      & 17.5546           \\ 

                      &                    &                    &                   & 
                      
& SD     &    0.0210   &   17.5   &  78.6441   & 0.3051     & 1.0599       &   17.9064        \\    
                      &                    &                    &                   &                                               
& QN     &   ---   &   ---   &  ---   & ---     &   ---     &     ---      \\

                      &                    &                    &                   &                                               
& CG     &   3.0562   &    51  & 83.7377    & 0.2712     & 0.9961        &    17.5518       \\

                      &                    &                    &                   &                                               
 & TR     &   0.0942   &   3   &  79.7762   & 0.2373     & 1.0959       &    18.2415       \\ \hline
\multirow{5}{*}{SLCDT1} & \multirow{5}{*}{2} & \multirow{5}{*}{2} & \multirow{5}{*}{$[-1.5, 1.5]^2$} & \multirow{5}{*}{Non-convex}

& CN     &  0.1236    &  3    &  2.3640   &   0.8197   &   0.9288     &     0.2093       \\ 

                      &                    &                    &                   & 
                      
& SD     &   0.0168   &   2.5   &  2.3640   &   0.8197   &    0.9287    & 0.2091          \\    
                      &                    &                    &                   &                                               
& QN     &   0.0546   &   3   &   2.3640  &    0.8197  &  0.9287      & 0.2091          \\

                      &                    &                    &                   &                                               
& CG     &   0.1053   &  2   &  2.3640   &   0.8197   &   0.9287     &   0.2091        \\

                      &                    &                    &                   &                                               
 & TR     &  0.0551    &  2   &  2.3601   &   0.7049   &    0.9287    &  0.2091         \\ \hline
\multirow{5}{*}{MOP3} & \multirow{5}{*}{2} & \multirow{5}{*}{2} & \multirow{5}{*}{$[-\pi, \pi]^2$} & \multirow{5}{*}{Non-convex}

& CN     &    0.1594   &   4   & 1.2138$\times 10^{3}$    &0.2333      &   1.7958     &    36.9866       \\ 

                      &                    &                    &                   & 
                      
& SD     &   0.0943   &   21.5   &  1.0992$\times 10^{3}$   &    0.1333   &   1.8868     &     22.4569      \\    
                      &                    &                    &                   &                                               
& QN     &  0.0864    &   5   & 1.2061$\times 10^{3}$    &   0.1333    &     1.7610    &     31.0150      \\

                      &                    &                    &                   &                                               
& CG     &   ---   &   ---   &  ---   &  ---    &    ---    &     ---      \\

                      &                    &                    &                   &                                               
 & TR     &  0.0688    &   3.5   &  1.2193$\times 10^{3}$    &   0.5000   &   2.6745     &     73.6982      \\ \hline
\multirow{5}{*}{VU1} & \multirow{5}{*}{2} & \multirow{5}{*}{2} & \multirow{5}{*}{$[-3, 3]^2$} & \multirow{5}{*}{Non-convex}

& CN     &  0.4542    &  12    &  16.0610   &    0.2155   &   1.1158      &   1.8651         \\ 

                      &                    &                    &                   & 
                      
& SD     &   0.7149   &   692.5   &   16.0663  &    0.2155   &    1.0874    &  1.8631         \\    
                      &                    &                    &                   &                                               
& QN     &   0.2468   &   7   &   16.0547  &     0.2155  &   1.1512     &    1.9843        \\

                      &                    &                    &                   &                                               
& CG     &   67.8223   &   1051   &  16.0564   &   0.1983   &   1.1077     &  1.8633         \\

                      &                    &                    &                   &                                               
 & TR     &   0.1278   &   12   & 16.0351    &  0.1767    &    1.0734      &   1.8640         \\ \thickhline
\end{tabular}
\end{table}

\end{landscape}

\begin{landscape}
\begin{table}[h]
\ContinuedFloat
\caption{(Continued)}
\begin{tabular}{cccccccccccc}
\thickhline
Problem               & $m$                & $n$                & Domain $(S)$              & Problem Type                          & Method & Median Time (s) & Median Iteration & Hypervolume & Purity & $\Delta_p$ & $\Gamma_p$ \\ \thickhline
\multirow{5}{*}{FON} & \multirow{5}{*}{2} & \multirow{5}{*}{3} & \multirow{5}{*}{$[-4,4]^3$}& \multirow{5}{*}{Non-convex}

& CN     &   0.3190   &   8   &  0.2327   &   0.2374   &   1.1544     &  0.1726         \\ 

                      &                    &                    &                   & 
                      
& SD     &  0.1522    &    28  &   0.2299    &  0.2172    &    1.2391     &   0.1958        \\    
                      &                    &                    &                   &                                               
& QN     &   0.2060   &   5   &  0.2206   &   0.1111    &    1.0019    &     0.2294      \\

                      &                    &                    &                   &                                               
& CG     &  2.3128    &  25    &   0.2339  &   0.2172   &   1.0997     &  0.1761         \\

                      &                    &                    &                   &                                               
 & TR     &    0.1735   &   5   &  0.2385   &    0.2273   &   0.9427     &   0.1968         \\ \hline
 
\multirow{5}{*}{Toi4} & \multirow{5}{*}{2} & \multirow{5}{*}{4} & \multirow{5}{*}{$[-2,5]^4$} & \multirow{5}{*}{Convex}

& CN     &   0.2410   &  5   &  638.4558   &  0.2252    &  0.9669      &   16.0117        \\ 

                      &                    &                    &                   & 
                      
& SD     &  0.0217    &   3   &  638.4568   &   0.1757   &   0.9129      &  14.0872         \\    
                      &                    &                    &                   &                                               
& QN     &   0.0853   &   4   &  638.4568   &    0.2207  &    0.9601     &  14.0957         \\

                      &                    &                    &                   &                                               
& CG     &    0.4653  &   5   &  638.4568   &  0.2252    &     0.9918   &  14.7739         \\

                      &                    &                    &                   &                                               
 & TR     &  0.1353    &   2   &   638.4568  &   0.2207   &    1.1018    &  14.0767         \\ \hline
\multirow{5}{*}{JOS1} & \multirow{5}{*}{2} & \multirow{5}{*}{4} & \multirow{5}{*}{$[-2,2]^4$} & \multirow{5}{*}{Convex} 
 
& CN     &   0.2930   &   6   &  5.9827   &   0.9643   &   1.0148      &  0.7858         \\ 

                      &                    &                    &                   & 
                      
& SD     &   0.0851   &  13   &  5.9827   &   0.9643     &  1.0545      &   0.7858        \\    
                      &                    &                    &                   &                                               
& QN     &   0.1011   &    3  &  5.9827   &  0.9643    &     1.0148    &   0.7858        \\

                      &                    &                    &                   &                                               
& CG     &  0.2554    &   3   &  5.9827   &   0.9643   &     1.0148    &  0.7858         \\

                      &                    &                    &                   &                                               
 & TR     &    0.1135   &   3   &  5.9827   &  1.0000    &   0.8462     & 5.0000          \\ \hline
\multirow{5}{*}{IKK1} & \multirow{5}{*}{3} & \multirow{5}{*}{2} & \multirow{5}{*}{$[-50,50]^2$} & \multirow{5}{*}{Convex} 

& CN     &  0.4623    &  8    &  1.3126$\times 10^{10}$   &  0.2874     &  1.0040      &    3.3064$\times 10^{3}$       \\ 

                      &                    &                    &                   & 
                      
& SD     &  0.0028    &   1   &   1.2753$\times 10^{10}$   &  0.1609    &   1.1821     &      3.3064$\times 10^{3}$    \\    
                      &                    &                    &                   &                                               
& QN     &   0.0573   &    1  &   1.2752$\times 10^{10}$  &  0.2874    &    1.1681    &     0.6281$\times 10^{3}$     \\

                      &                    &                    &                   &                                               
& CG     &   0.1051    &   2   &  1.3158$\times 10^{10}$   &   0.2874    &   1.0368     &     0.4471$\times 10^{3}$     \\

                      &                    &                    &                   &                                               
 & TR     &   0.1094   &  4    &  1.3294 $\times 10^{10}$  &   0.2874    &    0.9792     &      0.6492$\times 10^{3}$    \\ \hline
\multirow{5}{*}{VFM1} & \multirow{5}{*}{3} & \multirow{5}{*}{2} & \multirow{5}{*}{$[-2, 2]^2$} & \multirow{5}{*}{Convex}

& CN     &    0.3923   &  4    &   8.8653   &  0.3553    &    0.9086     & 1.0002          \\ 

                      &                    &                    &                   & 
                      
& SD     &  0.0017    &   1   & 11.9956     &  0.4211    &    0.8763    &  0.5972         \\    
                      &                    &                    &                   &                                               
& QN     &  0.0207    &   1   &  11.9956    &   0.4211   &  0.8763      &  0.5972         \\

                      &                    &                    &                   &                                               
& CG     &   0.0396   &  1    &  12.5684   &   0.5395    &   0.8623     &   0.4877        \\

                      &                    &                    &                   &                                               
 & TR     &   0.0344   &   1   &  12.0738   &   0.4079   &  0.8764    &  0.5972         \\ \hline
\multirow{5}{*}{MOP5} & \multirow{5}{*}{3} & \multirow{5}{*}{2} & \multirow{5}{*}{$[-30, 30]^2$} & \multirow{5}{*}{Non-convex} 

& CN     &   0.9826   &   5   &   2.5765$\times 10^{6}$   & 0.3226     &   1.2270     &   2.8391$\times 10^{3}$       \\ 

                      &                    &                    &                   & 
                      
& SD     &  0.0007    &   0   &  2.5765$\times 10^{6}$    &  0.3226    &  1.2266      &    2.7241$\times 10^{3}$     \\    
                      &                    &                    &                   &                                               
& QN     &   0.0232   &   0   &  2.5604$\times 10^{6}$   &  0.3226    &   1.0149     &      0.9250$\times 10^{3}$    \\

                      &                    &                    &                   &                                               
& CG     &   0.0008   &  0    &   2.5675$\times 10^{6}$   &  0.3226    &   1.2265     &   2.7241$\times 10^{3}$        \\

                      &                    &                    &                   &                                               
 & TR     &    0.0350  &   0   &   2.5675$\times 10^{6}$  &  0.3226    &    1.2265    &   2.7241$\times 10^{3}$        \\ \hline
\multirow{5}{*}{MOP7} & \multirow{5}{*}{3} & \multirow{5}{*}{2} & \multirow{5}{*}{$[-400, 400]^2$} & \multirow{5}{*}{Convex} 

& CN     &   3.1269   &   36   &  0.4801   & 0.2101      &  1.2841       &  0.9500          \\ 

                      &                    &                    &                   & 
                      
& SD     &  0.1701    &  147    & 0.4915    &  0.1594    &   1.2117     & 0.9629          \\    
                      &                    &                    &                   &                                               
& QN     &  0.2147    &  7.5    &  0.2733   &  0.3116    &   1.1937     &  0.3830         \\

                      &                    &                    &                   &                                               
& CG     &   0.6366   &    11  &   0.2895   &   0.3551   &   1.3533     &  0.4784         \\

                      &                    &                    &                   &                                               
 & TR     &   0.1989    &   9   &  0.3870   &   0.0217   &   1.0133     & 0.9471          \\ \hline
\multirow{5}{*}{SLCDT2}& \multirow{5}{*}{3} & \multirow{5}{*}{10} & \multirow{5}{*}{$[-1, 1]^{10}$} & \multirow{5}{*}{Non-convex} 

& CN     &  20.0960    &   199   &  0.4574$\times 10^{3}$    & 0.1111     &  1.2284      &    20.2194       \\ 

                      &                    &                    &                   & 
                      
& SD     &  0.0101    &   9   &  3.2004$\times 10^{3}$   &  0.2222    &  0.9982      &     4.5231      \\    
                      &                    &                    &                   &                                               
& QN     &   0.1015   &   8   &   3.3851$\times 10^{3}$   & 0.2222     &   1.0243      &   4.3223        \\

                      &                    &                    &                   &                                               
& CG     &   2.4207   &   26   &  3.4056$\times 10^{3}$   & 0.2222     &   1.0193     &    4.4557       \\

                      &                    &                    &                   &                                               
 & TR     &    0.0945  &  3    & 3.3930$\times 10^{3}$    &  0.2222    &  1.0758      &     5.3611       \\ \thickhline
\end{tabular}
\end{table}
\end{landscape}

\section{Conclusion and Final Remarks} \label{section-conclusion}
This paper presented a novel cubic regularization approach to Newton's method for solving unconstrained vector optimization problems, where the components of the objective function have Lipschitz Hessian. After proposing a cubic regularized Newton direction $d_M(\bar x)$ in \eqref{cubic_direction}, in Theorem \ref{stationarity_and_beta_value}, we used the value of $\beta_M$, as defined in \eqref{beta_function}, to characterize stationarity of the point $\bar x$. To find if $d_M(\bar x)$ is a descent direction at $\bar x$, we have derived a result in Theorem \ref{decrement_if_M_bigger_than_L} which shows that although $d_M(\bar x)$ may not be a $K$-descent direction at $\bar{x}$ (see Remark \ref{aux1_17_02}), there holds $f(\bar x + d_M(\bar x)) \prec f(\bar x)$ if $M \ge L$. With the help of the result in Theorem \ref{decrement_if_M_bigger_than_L}, we have shown that a simple three-step process (see Page \pageref{three-step-process}) for a cubic regularized Newton scheme can be identified. However, this process is heavily dependent on a good estimation of $L$-value, and also merely $\beta_M(x) = 0$ is found not to be a good stopping condition. Thus, we next aimed at an effective computation of $d_M(x)$ and at bypassing the use of $M \ge L$ and $\beta_M(x) = 0$. \\

Although the definition of $d_M(x)$ (see \eqref{cubic_direction}) looks simple, we have observed that the pertaining optimization problem appears with a non-convex objective function $q_M(x, \cdot)$. However, in Theorem \ref{d_M_is_one_dim_convex_min}, we have found that the minimization problem in \eqref{cubic_direction} is equivalent to a one-dimensional convex optimization problem. An explicit formula for computing $d_M(x)$ is then found in \eqref{10_01_25_aux2}. With the help of this $d_M(x)$ and an auxiliary function $h_M$ (see Lemma \ref{aux3_18_Feb_2025}), we derived a few side results to arrive at an effective stopping condition $\mu_M(\bar x) = 0$ that ensures that $\bar x$ is a weakly efficient point of \eqref{vop}.   \\

With the help of the computed $d_M(x)$ in \eqref{10_01_25_aux2} and $\mu_M$ in \eqref{mu_M_definition}, we have provided a step-wise cubic regularized Newton process in Algorithm \ref{algo} to identify weakly efficient points of \eqref{vop}. Thereafter, the well-definedness of the algorithm has been provided, and it is found that, in general, the method may not have descent property (see Remark \ref{no_descent_property_remark}). However, in Theorem \ref{x_k_conv_to_weak_min}, it has been derived that under a very mild restriction ($L_0 \ge \frac{2L}{3}$), any sequence $\{x^k\}$ generated by Algorithm \ref{algo} has the $K$-descent property: $f(x^{k + 1}) \prec f(x^k)$ for all $k$. It has also been found that there is a global convergence property of any sequence generated by the proposed method, and the rate of convergence is $O(k^{- 2/3})$  (see Theorem \ref{aux2_28_01_25}). We have further found in Theorem \ref{theorem_max_is_skipped} that none of the subsequential limit of $\{x^k\}$ is a non-weakly efficient point (see Remark \ref{non-weakly-effieient-remark}). Finally, in Theorem \ref{theorem_3_7}, second-order local convergence of the method has been found.   \\

Note that if an estimate of the Lipschitz constant $L$ is available, then the three steps that are mentioned on Page \pageref{three-step-process},  immediately after Remark \ref{aux1_17_02}, would be sufficient for identifying weakly efficient points of \eqref{vop} only through the computation of the cubic regularized steps $d_{M_k}(x^k)$. As a good estimate of $L$-value may not be an easy task, to bypass using $M \ge L$, the function $h_M$ (see Lemma \ref{aux3_18_Feb_2025}) has been used in Step 2 of Algorithm \ref{algo}. Step 2 is found to be influential in arriving at the global convergence result reported in Theorem \ref{aux2_28_01_25}. Future research can endeavor to modify Step 2 without using $h_M$ but by utilizing ``$M \ge L$ and a suitable effective way to estimate $L$." At an iterate $x^k$, a crude way of estimating $L$ for $f : \mathbb{R}^n \to \mathbb{R}^m$ around $x^k$ can be as follows:  \\ 

\noindent
\textbf{An estimation of $L$.}
\begin{enumerate}[Step 1.]
\item For each $i = 1, 2, \ldots, m$, choose an initial guess $L_i'$ of a Lipschitz constant of $L_i$ of $f_i : \mathbb{R}^n \to \mathbb{R}$. 

\item Choose $\theta_1 > 1$ and $\theta_2 < 1$. 

\item For each $i = 1, 2, \ldots, m$, perform the following steps: 
\begin{enumerate}[(i)]
\item (Estimate a sufficiently large Lipschitz constant for $f_i$) \\ 
While $|f(x^k) - f(x^{k + 1})| > L_i' \|x^k - x^{k - 1}\|$ do $L_i' \leftarrow \theta_1 L_i'$. 

\item (Refine the large Lipschitz constant to smaller ones) \\ 
While $|f(x^k) - f(x^{k + 1})| \le L_i' \|x^k - x^{k - 1}\|$ do $L_i' \leftarrow \theta_2 L_i'$. 
\end{enumerate}

\item Return $L := \max\{L_1', L_2', \ldots, L_m'\}$ as an estimate of the Lipschitz constant of $f$ around $x^k$. \\ 
\end{enumerate}

The rate of (global) convergence of any sequence generated by the proposed cubic regularized Newton method is $O(k^{-2/3})$ (Theorem \ref{aux2_28_01_25}). Also, the local convergence rate of Algorithm \ref{algo} is found to be q-quadratic (Theorem \ref{theorem_3_7}).  Future research can find the rate of (global) convergence and order of (local) convergence of Algorithm \ref{algo} if Step 2 is modified by utilizing $M \ge L$, where $L$ is given by the above or other estimation. Furthermore, a complexity analysis to find an $L$ can be done in the future.  \\

Note that apart from an effective way to estimate $L$, we also need a computational strategy to find $M_k$ at $x^k$ for Step 2 of Algorithm \ref{algo}. In Remark \ref{process_to_evaluate_M_k}, we have provided a method to find an $M_k$. One may think of the following drawbacks of the method. ``This method may blow up the value of $M_k$. Thus, if we loosely take a very big value of $M_k$, then such $M_k$ will make the factors $\frac{2}{M_k + L}$ and $\frac{2}{M_k + 2L}$ in the computation of $\mu_{M_k}$ (see \eqref{mu_M_definition}) very small. Thereby, $\mu_{M_k}(x^k)$ may become very small yet $\min_{\xi \in C} \|\la Jf(x^k), \xi \ra \|$ is far away from zero. Hence, a very big value of $M_k$ may terminate the Algorithm \ref{algo} at a point $x^k$ that may not be a stationary point." However, note that the value of $M_k$ identified by the method in Remark \ref{process_to_evaluate_M_k} can never exceed $2L$ since by \eqref{aux5_06_03_25}, for any $M_k \ge L$, we have 
\[\max_{\xi \in C} \la f(x^k + d_{M_k}(x^k)), \xi \ra \le h_M(x^k). \]
Future research may be done on rigorous complexity analysis of this strategy for choosing $M_k$ for Step 2 of Algorithm \ref{algo}. \\

Although Theorem \ref{theorem_3_7} gives a $q$-quadratic rate of the local convergence of the sequences generated by the proposed Algorithm \ref{algo}, the result heavily depends on the condition $\gamma_0 \le \frac{1}{4}$. Future research can analyze if such a condition on the initial point can be relaxed. Also, note that the denominator of $\gamma_k$ is $(\min_{\xi \in C} \lambda_1(\la \nabla^2 f(x^k), \xi \ra))^2$, which makes the evaluation of $\gamma_k$ difficult. It would be better if we could replace the denominator of $\gamma_k$ simply by $\lambda^2_1 \left( \la \nabla^2 f(x^k), \eta^k(x^k) \ra \right).$ However, with such modification on the definition of $\gamma_k$, we could not derived the result. Future research can be done in this direction. A major difficulty in proving the results in Theorem \ref{theorem_3_7} with 
\[
\gamma_k :=  \frac{L \left\| \la Jf(x^k), \eta^k \ra \right\|}{ \lambda^2_1 \left(\la \nabla^2 f(x^k), \eta^k \ra \right)}
\]
is to derive 
\begin{equation}\label{aux9_23_02}
\la \nabla^2 f(x^{j + 1}), \eta^{j + 1} \ra \ge  \la  \nabla^2 f(x^j), \eta^j \ra - \alpha r_{M_j} (x^j) I_n 
\end{equation}

\medskip
\noindent
for some $\alpha > 0$. This inequality is somewhat similar to \eqref{aux2_22_02_25}, and without such an inequality, identification of local quadratic convergence of $\{x^k\}$ may become very difficult. To arrive at an inequality similar to \eqref{aux9_23_02}, we may proceed in the following manner. On the set $B$ as defined in Theorem \ref{theorem_3_7}, define two functions $\phi: B \to \mathbb{R}$ and $\psi: B \to \mathbb{R}$ by 
\[ \phi(x) := \min_{\xi \in C} \lambda_1(\la \nabla^2 f(x), \xi\ra) \text{ and }  \psi(x) := \max_{\xi \in C} \lambda_1(\la \nabla^2 f(x), \xi \ra). \]
Since $C$ is compact and for any given $x \in B$, the function 
$ \xi \mapsto \lambda_1(\la \nabla^2 f(x), \xi\ra) $
is continuous on $C$, the functions $\phi$ and $\psi$ are continuous on $B$. 
As $\lambda_1(\la \nabla^2 f(x^j), \xi \ra) > 0$ for all $\xi \in C$ 
and $C$ is compact, there exists $\beta > 0$ such that 
\[
\frac{\lambda_1(\la \nabla^2 f(x^j), \eta^{j + 1} \ra)}{\lambda_1(\la \nabla^2 f(x^j), \eta^{j} \ra)} \ge \frac{ \min_{\xi \in C} \lambda_1(\la \nabla^2 f(x^j), \xi \ra)}{\max_{\xi \in C}\lambda_1(\la \nabla^2 f(x^j), \xi \ra)} = \frac{\phi(x^j)}{\psi(x^j)} \ge \beta. 
\]
This inequality combined with \eqref{aux1_22_02_25} may arrive at an inequality similar to \eqref{aux9_23_02}. Future research may be performed on this modification of $\gamma_k$ and the restriction $\gamma_0 \le \frac{1}{4}$. \\

Although Remark \ref{lemma_d_M_is_1d_convex} shows that $d_M(\bar x)$ exists for any $M > 0$ and $\bar x \in S$, to find $d_M(\bar x)$ or to find a minimum point for the left-hand side of \eqref{10_1_25_aux3}, we need a suitable $r$-value from $\mathscr{D}_{\bar{x}}$ so that $\|d(\bar x, r)\| = r$. 
In the case when $\la \nabla^2 f(\bar x), \xi \ra$ is positive definite for all $\xi \in C$, identification of such an $r$ can be made crudely as follows, which is directed by \eqref{d_bar_x_r} and the definition of $\mathscr{D}_{\bar x}$. Find an $r$ that satisfies 
\begin{align*} 
& \min_{\xi \in C} \lambda_1 (\la \nabla^2 f(\bar x), \xi \ra) + \frac{Mr}{2} \ge 0 \\ 
& \text{ and } \left\| - \left( \sum_{i = 1}^{s(\bar x)} \lambda_i(\bar x) \left\{ \la \nabla^2 f(\bar x), \eta^i(\bar x) \ra + \frac{M}{2} r ~ I_n \right\}\right)^{-1} \sum_{i = 1}^{s(\bar x)} \lambda_i(\bar x) \la Jf(\bar x), \eta^i(\bar x) \ra \right\| = r, \\ 
\text{i.e., } & r \ge - \frac{2}{M} \min_{\xi \in C} \lambda_1 (\la \nabla^2 f(\bar x), \xi \ra) \text{ and } r^2 = \sum_{j = 1}^n  \frac{\sigma^2_j}{(\mu_j + \frac{Mr}{2})^2},  
\end{align*}
where $\mu_1, \mu_2, \ldots, \mu_n$ are eigenvalues of $$Q(\bar x) : = \la  \nabla^2 f(\bar x), \sum_{i = 1}^{s(\bar x)} \lambda_i(\bar x) \eta^i(\bar x) \ra + \frac{M}{2} r ~ I_n$$  
and $(\sigma_1, \sigma_2, \ldots, \sigma_n)^\top$ is the coordinates of $\sum_{i = 1}^{s(\bar x)} \lambda_i(\bar x) \la Jf(\bar x), \eta^i(\bar x) \ra$ with respect to the eigenvectors of $Q(\bar x)$.  
Future research can be done on an effective computation of this $r$-value even when $\la \nabla^2 f(\bar x), \xi \ra$ is not positive definite for some $\xi \in C$. Moreover, note that the cubic Newton step $d_M(\bar x) = d(\bar x, r)$ is such a point that satisfies $\|d_M(\bar x)\| = r$, which is somewhat similar to finding a (trust-region) step inside the trust-region of radius $r$. Future research can find a connection or a comparison between the trust-region step and the cubic Newton step for vector optimization.

Apart from the above-mentioned technical points for the computational strategies, we mention below a few more potential future scopes: 

\begin{itemize}
\item 
Note that in \eqref{cubic_direction} and \eqref{beta_function}, the cubic Newton direction $d_M(x)$ is defined with the help of the Drummond-Svaiter non-linear scalarizing function. Future research can define $d_M(x)$ through other scalarizing functions---Gerstewitz and oriented distance---and analyze the convergence of Algorithm \ref{algo}. 

\item 
In Theorem \ref{x_k_conv_to_weak_min}, we have shown that the subsequential limits of $\{x^k\}$ lie on the same level set of the objective function but could not prove the connectedness of the set of all subsequential limits. Commonly, for descent methods, there holds this connectedness property. Future research can be attempted in this direction.

\item 
To the best of the knowledge of the author, all the existing parameter-free gradient-based methods for vector optimization have been developed under the assumption that the ordering cone $K$ is polyhedral. In fact, the current study is also performed under this assumption. Derivation of a parameter-free gradient-based method for non-polyhedral $K$ is yet an open problem. In the future, the proposed method could be extended in this direction.


\item 
The proposed method can be employed for solving robust vector optimization problems or set optimization problems. Then, a comparison of computational complexities with the existing Newton method in \cite{Ghosh2025newton} or conjugate gradient method in \cite{kumar2024nonlinear} can be done in the future.

\end{itemize}

\vspace{-0.3cm}

\subsubsection*{Acknowledgment}
The author sincerely thanks the two anonymous reviewers for their insightful comments and constructive suggestions, which have significantly enhanced the quality of this paper. 

\vspace{-0.1cm}
\subsubsection*{Funding}
Core Research Grant (CRG/2022/001347) from SERB, India   

\vspace{-0.25cm}

\subsubsection*{Data availability}
There is no data associated with this paper.

\subsubsection*{Declarations}    
The author has no conflict of interest, financial or ethical, to declare. Also, there are no human/animal studies related to the work.





\end{document}